%% file: main.tex
\documentclass[letterpaper,12pt]{amsart}

\usepackage[divide={1in,*,1in}]{geometry}
\usepackage{amsfonts,amssymb}%,amsrefs}
\usepackage{mathtools}
\usepackage{stmaryrd}
\usepackage{pstricks,pst-node}
\usepackage{mywrap}
\usepackage{centerpict}
\usepackage{cobordisms}
\usepackage{diagps,diagps-arrows}
\usepackage{movies}

\usepackage[
		pdfauthor={Krzysztof K. Putyra},
		pdftitle={On a triply graded Khovanov homology},
		colorlinks=true]{hyperref}
\usepackage[savepos]{zref}
\usepackage[all]{hypcap}		%	Fix hyperlinks to figures. Use \capstart for wrapfigure, etc.

\let\cites\cite
\input{commands.tex}

\input{images/pics-tangles.tex}
\input{images/pics-cobs.tex}
\input{images/pics-module.tex}
\input{images/trefoils.tex}
\input{images/diag-changes.tex}

\COBembeddedtrue
\def\COBxsize{0.4}
\def\COBysize{1.2}
\definecolor{darkred}{rgb}{0.5,0,0}%

\psset{unit=7mm}
\allowdisplaybreaks

\def\drawSaddle{%
	\begin{centerpict}(-1pt,-2pt)(21pt,18pt)%
		\psset{linewidth=0.5pt}%
		\psline(17pt,16pt)(17pt,5pt)
		\psbezier( 0pt, 5pt)( 9pt, 3pt)( 9pt, 3pt)(17pt, 5pt)
		\psline[border=1pt]( 3pt,11pt)( 3pt,0pt)
		\psbezier[border=1pt](20pt,11pt)(12pt,12pt)(11pt,15pt)(17pt,16pt)
		\psbezier( 3pt, 0pt)(11pt, 2pt)(11pt, 2pt)(20pt, 0pt)
		\psbezier( 3pt,11pt)( 9pt,12pt)( 8pt,15pt)( 0pt,16pt)
		\psline(17pt,16pt)(17pt,14pt)
		\psline( 0pt,16pt)( 0pt,5pt)
		\psline(20pt,11pt)(20pt,0pt)
		\psbezier(6.67pt,14pt)(7pt,5pt)(13pt,5pt)(13.33pt,14pt)
	\end{centerpict}%
}

\selectfont

\begin{document}

\title{On a~triply graded Khovanov homology}
\author{Krzysztof K.\ Putyra}

\date\today


\begin{abstract}
	Cobordisms are naturally bigraded and we show that this grading extends to Khovanov homology, making it a~triply graded theory. Although the~new grading does not make the~homology a~stronger invariant, it can be used to show that odd Khovanov homology is multiplicative with respect to disjoint unions and connected sums of links; same results hold for the~generalized Khovanov homology defined by the~author in his previous work. We also examine the~module structure on both odd and even Khovanov homology, in particular computing the~effect of sliding a~basepoint through a~crossing on the~integral homology.
\end{abstract}

\maketitle

\section{Introduction}\label{sec:intro}
\input{intro.tex}

\section{Chronological TQFTs}\label{sec:chron-tqft}
\input{tqft.tex}

\section{The~generalized Khovanov complex}\label{sec:khov-def}
\input{complex.tex}

\section{Homology}\label{sec:homology}
\input{homology.tex}

\section{Khovanov complexes for composite links}\label{sec:composite}
\input{composite.tex}

\section{Module structure and homology}\label{sec:module}
\input{module.tex}

\input{references.tex}
\end{document}

%% file: commands.tex
\newtheorem{theorem}{Theorem}[section]
\newtheorem{lemma}[theorem]{Lemma}

\newtheorem{proposition}[theorem]{Proposition}
\newtheorem{corollary}[theorem]{Corollary}

\theoremstyle{definition}

\newtheorem{definition}[theorem]{Definition}
\newtheorem{example}[theorem]{Example}
\newtheorem{remark}[theorem]{Remark}

\definecolor{internalLink}{rgb}{0.5,0,0}
\definecolor{citeLink}{rgb}{0,0.5,0}
\definecolor{urlLink}{rgb}{0,0,0.5}

\hypersetup{linkcolor=internalLink}
\hypersetup{citecolor=citeLink}
\hypersetup{urlcolor=urlLink}

\DeclareMathOperator{\id}{id}		%	identity morphism
\def\cocolon{%
	\nobreak\mskip6mu plus1mu\mathpunct{}%
	\nonscript\mkern-\thinmuskip{:}%
	\mskip2mu\relax
}

\newcommand*{\Z}{\mathbb{Z}}				% Integer numbers - bold Z
\newcommand*{\R}{\mathbb{R}}				% Real numbers - bold R

\let\scalars\Bbbk
\def\invScalars{\scalars^*}

\def\Zev{\Z_{ev}}
\def\Zodd{\Z_{odd}}
\def\Zpi{\mathbb{Z}_{\pi}}

\makeatletter
\def\scalarsLong{\@ifstar
	{\mathbb{Z}[\permMM,\permSS,\permMS^{\pm1}]/(\permMM^2{=}\permSS^2{=}1)}%
	{\mathbb{Z}[\permMM,\permSS,\permMS^{\pm1}]/(\permMM^2=\permSS^2=1)}%
}
\def\ZevLong{\@ifstar{\Zpi/(\pi{-}1)}{\Zpi/(\pi-1)}}
\def\ZoddLong{\@ifstar{\Zpi/(\pi{+}1)}{\Zpi/(\pi+1)}}
\def\ZpiLong{\@ifstar{\mathbb{Z}[\pi]/(\pi^2{-}1)}{\mathbb{Z}[\pi]/(\pi^2-1)}}

\makeatother

\newcommand*{\cat}[1]{\textrm{\normalfont\textbf{#1}}}		%	Symbol of a category
\newcommand*{\catAdd}[1]{\mathrm{Mat}(#1)}

\newcommand*{\Mor}{\mathrm{Mor}}

\newcommand*{\F}{\mathcal{F}}
\newcommand*{\FA}{\mathcal{F}}

\newcommand*\Mod[1]{\cat{Mod}_{#1}}

\newcommand*{\cone}{C}

\def\newcobcategory#1#2{%
	\expandafter\def\csname #1\endcsname{#2}
	\expandafter\def\csname #1L\endcsname{#2_{/\ell}}
	\expandafter\def\csname #1D\endcsname{#2_{\bullet}}
}

\newcobcategory{Cob}{\cat{Cob}}
\newcobcategory{ChCob}{\cat{ChCob}}
\newcobcategory{kChCob}{\scalars\cat{ChCob}}

\def\vvvv#1,{v_{#1}\vvvvv}
\def\vvvvv#1,{%
	\ifx\endvvv#1\relax\else
		\otimes\ifx*#1\else v_{#1}\fi
	\expandafter\vvvvv\fi}

\newlength\fntinkwidth
\newlength\fntsize
\newlength\fntshift
\newbox\fntbox

\makeatletter
\def\defineMathSymbol#1(#2)(#3)[#4]#5{%
	\expandafter\edef\csname fntMath#1\endcsname{\noexpand\mathchoice
			{\expandafter\noexpand\csname fntMath#1@draw\endcsname\noexpand\textfont\noexpand\displaystyle}%
			{\expandafter\noexpand\csname fntMath#1@draw\endcsname\noexpand\textfont\noexpand\textstyle}%
			{\expandafter\noexpand\csname fntMath#1@draw\endcsname\noexpand\scriptfont\noexpand\scriptstyle}%
			{\expandafter\noexpand\csname fntMath#1@draw\endcsname\noexpand\scriptscriptfont\noexpand\scriptscriptstyle}%
	}%
	\expandafter\def\csname fntMath#1@draw\endcsname##1##2{\begingroup
		\fntinkwidth=\fontdimen8##13\relax
		\setbox\fntbox=\hbox{$\m@th##2\mathrm{M}$}%
		\fntsize = \wd\fntbox
		\psset{unit=\fntsize,linewidth=\fntinkwidth}%
		\begin{pspicture}[shift=#4](#2)(#3)#5\end{pspicture}%
	\endgroup}
	\edef\next{\noexpand\DeclareMathOperator\expandafter\noexpand\csname#1\endcsname{\expandafter\noexpand\csname fntMath#1\endcsname}}\next
}
\makeatother

\defineMathSymbol ldsum(-0.2,0)(0.8,0.8)[0pt]{
		\psset{arrowsize=0.4,arrowlength=0.8,arrowinset=0.5}%
		\psline{c-c}(0,0.7)(0,0.04)(0.6,0.04)(0.6,0.7)
		\psline{->}(0.0,0.55)(0.0,0.25)
		\psline{->}(0.6,0.25)(0.6,0.55)
}
\defineMathSymbol rdsum(-0.2,0)(0.8,0.8)[0pt]{
		\psset{arrowsize=0.4,arrowlength=0.8,arrowinset=0.5}%
		\psline{c-c}(0,0.7)(0,0.04)(0.6,0.04)(0.6,0.7)
		\psline{->}(0.6,0.75)(0.6,0.25)
		\psline{->}(0.0,0.25)(0.0,0.55)
}

\defineMathSymbol connsum(0,-0.2)(0.7,0.75)[-0.2]{%
		\psline{c-c}(0.15,-0.15)(0.33,0.7)
		\psline{c-c}(0.37,-0.15)(0.55,0.7)
		\psline{c-c}(0.0,0.17)(0.7,0.17)
		\psline{c-c}(0.0,0.39)(0.7,0.39)
}

\defineMathSymbol lcsum(0,-0.2)(0.7,0.75)[-0.2]{%
		\pspolygon[fillcolor=black,fillstyle=solid](0.05,-0.15)(0.097,0.07)(-0.021,0.07)(-0.068,-0.15)
		\pspolygon[fillcolor=black,fillstyle=solid](0.721,0.49)(0.605,0.49)(0.65,0.7)(0.766,0.7)
		\psline{c-c}(0.15,-0.15)(0.33,0.7)
		\psline{c-c}(0.37,-0.15)(0.55,0.7)
		\psline{c-c}(0.0,0.17)(0.7,0.17)
		\psline{c-c}(0.0,0.39)(0.7,0.39)
}

\defineMathSymbol rcsum(0,-0.2)(0.7,0.75)[-0.2]{%
		\pspolygon[fillcolor=black,fillstyle=solid](0.23,0.71)(0.184,0.49)(0.021,0.49)(0.067,0.71)
		\pspolygon[fillcolor=black,fillstyle=solid](0.47,-0.16)(0.518,0.07)(0.679,0.07)(0.631,-0.16)
		\psline{c-c}(0.15,-0.15)(0.33,0.71)
		\psline{c-c}(0.37,-0.16)(0.55,0.7)
		\psline{c-c}(0.0,0.17)(0.7,0.17)
		\psline{c-c}(0.0,0.39)(0.7,0.39)
}

\makeatletter
\def\KhBracket{\@ifstar\KhBracketScaled\KhBracketSimple}
\def\KhCube{\@ifstar\KhCubeScaled\KhCubeSimple}
\def\KhCubeSigned{\@ifstar\KhCubeSignedScaled\KhCubeSignedSimple}
\def\KhCubeGraded{\@ifstar\KhCubeGradedScaled\KhCubeGradedSimple}
\def\KhCubeGradedSigned{\@ifstar\KhCubeGradedSignedScaled\KhCubeGradedSignedSimple}
\makeatother

\newcommand*{\KhCubeScaled}[1]{\mathcal{I}\left(#1\right)}
\newcommand*{\KhCubeSimple}[1]{\mathcal{I}(#1)}
\newcommand*{\KhCubeSignedScaled}[2]{\mathcal{I}^{#2}\left(#1\right)}
\newcommand*{\KhCubeSignedSimple}[2]{\mathcal{I}^{#2}(#1)}
\newcommand*{\KhCubeGradedScaled}[1]{\mathcal{I}_{gr}\left(#1\right)}
\newcommand*{\KhCubeGradedSimple}[1]{\mathcal{I}_{gr}(#1)}
\newcommand*{\KhCubeGradedSignedScaled}[2]{\mathcal{I}_{gr}^{#2}\left(#1\right)}
\newcommand*{\KhCubeGradedSignedSimple}[2]{\mathcal{I}_{gr}^{#2}(#1)}
\newcommand*{\KhBracketScaled}[1]{\left\llbracket#1\right\rrbracket}
\newcommand*{\KhBracketSimple}[1]{\llbracket#1\rrbracket}

\newcommand*{\KhCom}{K\!h}

\newcommand*{\Kh}{\mathcal{H}}
\newcommand*{\EKh}{\Kh_{ev}}
\newcommand*{\OKh}{\Kh_{odd}}

\let\chdeg\deg
\def\sdeg{\mathrm{sdeg}}

\def\sdegCob[#1){\sdeg\left(\textcobordism[#1)\right)}

\def\permMM{X}%{\mathbb X}
\def\permSS{Y}%{\mathbb Y}
\def\permMS{Z}%{\mathbb Z}
\def\permSM{Z^{-1}}%{\permMS^{-1}}
\def\permTpos{1}
\def\permTneg{\permMM\permSS}

\makeatletter

\def\arXiv{\@ifstar\arXiv@@\arXiv@}
\def\arXiv@#1{\href{http://front.math.ucdavis.edu/#1}{arXiv:#1}}
\def\arXiv@@#1#2{\href{http://front.math.ucdavis.edu/#2}{arXiv:#1}}

\makeatother

\def\proofpart#1{\medskip\noindent\textit{#1.}\space\ignorespaces}%

%% file: images/pics-cobs.tex
\def\defPicSmallCob#1#2{%
	\expandafter\def\csname fntCob#1\endcsname{%
		\begin{centerpict}(-0.1,0)(1.1,1.2)
			\psset{linewidth=0.15pt,dotsep=1pt,dash=1pt 1.5pt,linestyle=dotted,dimen=mid}
			\psset{linewidth=0.3pt,linestyle=solid}#2
		\end{centerpict}}%
}

\def\defPicCob#1#2{\expandafter\def\csname cobInsert@#1\endcsname{#2}}%

\def\psdash#1{#1[linestyle=dashed,linewidth=0.15pt]}

\def\insertCob#1{\begin{centerpict}(-0.1,0)(1.3,2)
	\psset{xunit=1.2\psunit,linewidth=0.15pt,dotsep=1pt,dash=1pt 1pt,linestyle=dotted,dimen=mid}
	\psellipse(0.5,0.2)(0.5,0.2)\psline(0,0.2)(0,1.8)
	\psellipse(0.5,1.8)(0.5,0.2)\psline(1,0.2)(1,1.8)
	\psset{linewidth=0.4pt,linestyle=solid}
	\csname cobInsert@#1\endcsname
\end{centerpict}}%

\def\fntCob#1{\csname fntCob#1\endcsname}

\def\cobRIhtop{%
	\psellipticarc(0.65,0)(0.15,0.1){-90}{90}
	\psbezier(0.13523, 0.13679)(0.45,-0.10)(0.3 , 0.1)(0.65, 0.1)
	\psbezier(0.27639,-0.17889)(0.35, 0.05)(0.45,-0.1)(0.65,-0.1)
}
\def\cobRIvtop{%
	\psellipse(0.65,0)(0.15,0.1)
	\psbezier(0.27639,-0.17889)(0.4,0)(0.4,0)(0.13523,0.13679)
}
\def\cobRIhbottom{%
	\psellipticarc(0.65,0)(0.15,0.1){-90}{0}
	\psbezier(0.27639,-0.17889)(0.35,0.05)(0.45,-0.1)(0.65,-0.1)
	\psclip{\psframe[linestyle=none](0,-0.2)(0.277,1)}
		\psbezier(0.13523,0.13679)(0.45,-0.1)(0.3,0.1)(0.65,0.1)
	\endpsclip
	\psdash\psellipticarc(0.65,0)(0.15,0.1){0}{90}
	\psclip{\psframe[linestyle=none](1,-0.2)(0.276,1)}
		\psdash\psbezier(0.13523,0.13679)(0.45,-0.1)(0.3,0.1)(0.65,0.1)
	\endpsclip
}
\def\cobRIvbottom{%
	\psellipticarc(0.65,0)(0.15,0.1){-180}{0}
	\psdash\psellipticarc(0.65,0)(0.15,0.1){0}{180}
	\psclip{\pspolygon[linestyle=none](0,-0.2)(0.4,-0.2)(0.4,-0.028)(0.277,-0.028)(0.277,0.2)(0,0.2)}
		\psbezier(0.27639,-0.17889)(0.4,0)(0.4,0)(0.13523,0.13679)
	\endpsclip
	\psclip{\psframe[linestyle=none](0.2763,-0.0284)(0.3564,0.2)}
		\psdash\psbezier(0.27639,-0.17889)(0.4,0)(0.4,0)(0.13523,0.13679)
	\endpsclip
}

\defPicSmallCob{RId}{%
	\psline(0.27639,0.02111)(0.27639,0.82111)
	\psline(0.13523,0.33679)(0.13523,1.13679)
	\psline(0.80000,0.20000)(0.80000,1.00000)
	\rput(0,0.2)\cobRIvbottom
	\rput(0,1.0)\cobRIhtop
	\psbezier(0.35637,0.17164)(0.37,0.7)(0.49,0.7)(0.5,0.2)
}

\defPicSmallCob{RIh}{%
	\psline(0.27639,0.02111)(0.27639,0.82111)
	\psline(0.13523,0.33679)(0.13523,1.13679)
	\psbezier(0.35637,0.97164)(0.37,0.4)(0.8,0.4)(0.8,0.2)
	\rput(0,0.2)\cobRIhbottom
	\rput(0,1.0)\cobRIvtop
	\psbezier(0.5,1)(0.5,0.6)(0.8,0.6)(0.8,1)
}

\defPicSmallCob{RIg}{%
	\psline(0.27639,0.02111)(0.27639,0.82111)
	\psline(0.13523,0.33679)(0.13523,1.13679)
	\psbezier(0.35637,0.17164)(0.37,0.8)(0.8,0.8)(0.8,1)
	\rput(0,0.2)\cobRIvbottom
	\rput(0,1.0)\cobRIhtop
	\psbezier(0.5,0.2)(0.5,0.6)(0.8,0.6)(0.8,0.2)
}

\defPicSmallCob{RIfa}{%
	\psline(0.27639,0.02111)(0.27639,0.82111)
	\psline(0.13523,0.33679)(0.13523,1.13679)
	\psbezier(0.35637,0.97164)(0.37,0.1)(0.8,0.5)(0.8,0.2)
	\rput(0,0.2)\cobRIhbottom
	\rput(0,1.0)\cobRIvtop
	\psellipticarc(0.65,0.75)(0.2,0.2){-180}{0}
	\psbezier(0.5,1)(0.5,0.9)(0.45,0.85)(0.45,0.75)
	\psbezier(0.8,1)(0.8,0.9)(0.85,0.85)(0.85,0.75)
	\psellipticarc(0.65,0.78)(0.1,0.1){-170}{-10}
	\psellipticarc(0.65,0.65)(0.1,0.1){40}{140}
}

\defPicSmallCob{RIfb}{%
	\psline(0.27639,0.02111)(0.27639,0.82111)
	\psline(0.13523,0.33679)(0.13523,1.13679)
	\psline(0.80000,0.20000)(0.80000,1.00000)
	\rput(0,0.2)\cobRIhbottom
	\rput(0,1.0)\cobRIvtop
	\psbezier(0.35637,0.97164)(0.37,0.5)(0.49,0.5)(0.5,1)
}

\defPicCob{R1-df-0}{%
	\rput(0,0.2)\cobRIhbottom
	\rput(0,1.0)\cobRIvbottom
	\rput(0,1.8)\cobRIhtop
	\psline(0.27639,0.02111)(0.27639,1.62111)
	\psline(0.13523,0.33679)(0.13523,1.93679)
	\psline(0.80000,1.00000)(0.80000,1.80000)
	\psbezier(0.35637,0.97164)(0.37,0.1)(0.8,0.5)(0.8,0.2)
	\psbezier(0.35637,0.97164)(0.37,1.5)(0.49,1.5)(0.5,1)
	\psellipticarc(0.65,0.75)(0.2,0.2){-180}{0}
	\psbezier(0.5,1)(0.5,0.9)(0.45,0.85)(0.45,0.75)
	\psbezier(0.8,1)(0.8,0.9)(0.85,0.85)(0.85,0.75)
	\psellipticarc(0.65,0.78)(0.1,0.1){-170}{-10}
	\psellipticarc(0.65,0.65)(0.1,0.1){40}{140}
}
\defPicCob{R1-df-1}{%
	\rput(0,0.2)\cobRIhbottom
	\rput(0,1.0)\cobRIvbottom
	\rput(0,1.8)\cobRIhtop
	\psline(0.27639,0.02111)(0.27639,1.62111)
	\psline(0.13523,0.33679)(0.13523,1.93679)
	\psline(0.80000,0.20000)(0.80000,1.80000)
	\psbezier(0.35637,0.97164)(0.37,1.5)(0.49,1.5)(0.5,1)
	\psbezier(0.35637,0.97164)(0.37,0.5)(0.49,0.5)(0.5,1)
}

\defPicCob{R1-df-2}{%
	\rput(0,0.2)\cobRIhbottom
	\rput(0,1.8)\cobRIhtop
	\psline(0.27639,0.02111)(0.27639,1.62111)
	\psline(0.13523,0.33679)(0.13523,1.93679)
	\psbezier(0.35637,0.97164)(0.37,0.1)(0.8,0.5)(0.8,0.2)
	\rput(0,1){%
		\psclip{\pspolygon[linestyle=none](0,-0.2)(0.4,-0.2)(0.4,-0.028)(0.277,-0.028)(0.277,0.2)(0,0.2)}
			\psbezier(0.27639,-0.17889)(0.4,0)(0.4,0)(0.13523,0.13679)
		\endpsclip
		\psclip{\psframe[linestyle=none](0.2763,-0.0284)(0.3564,0.2)}
			\psdash\psbezier(0.27639,-0.17889)(0.4,0)(0.4,0)(0.13523,0.13679)
		\endpsclip
		\psellipticarc(0.55,0)(0.1,0.06){-180}{0}
		\psellipticarc(0.85,0)(0.1,0.06){-180}{0}
		\psdash\psellipticarc(0.55,0)(0.1,0.06){0}{180}
		\psdash\psellipticarc(0.85,0)(0.1,0.06){0}{180}
	}
	\psbezier(0.35637,0.97164)(0.37,1.3)(0.44,1.3)(0.45,1)
	\psellipticarc(0.7,1)(0.25,0.5){-180}{0}
	\psellipticarc(0.7,1)(0.05,0.3){-180}{0}
	\psellipticarc(0.7,1)(0.05,0.4){0}{180}
	\psbezier(0.95,1)(0.95,1.5)(0.8,1.5)(0.8,1.8)
}

\defPicCob{R1-df-3}{%
	\rput(0,0.2)\cobRIhbottom
	\rput(0,1.8)\cobRIhtop
	\psline(0.27639,0.02111)(0.27639,1.62111)
	\psline(0.13523,0.33679)(0.13523,1.93679)
	\psline(0.8,0.7)(0.8,1.8)
	\rput(0,1){%
		\psclip{\pspolygon[linestyle=none](0,-0.2)(0.5,-0.2)(0.5,0)(0.277,0)(0.277,0.2)(0,0.2)}
			\psbezier(0.27639,-0.17889)(0.49,0)(0.49,0)(0.13523,0.13679)
		\endpsclip
		\psclip{\psframe[linestyle=none](0.2763,0)(0.5,0.2)}
			\psdash\psbezier(0.27639,-0.17889)(0.49,0)(0.49,0)(0.13523,0.13679)
		\endpsclip
		\psellipticarc(0.7,0)(0.1,0.06){-180}{0}
		\psdash\psellipticarc(0.7,0)(0.1,0.06){0}{180}
	}
	\psellipse(0.5,1.2)(0.12,0.3)
	\psarc(0.72,0.7){0.08}{-180}{0}
	\psarc(0.54,0.7){0.1}{0}{180}
	\psbezier(0.44,0.7)(0.44,0.4)(0.8,0.5)(0.8,0.2)
}

\defPicCob{R1-GF-0}{%
	\rput(0,0.2)\cobRIhbottom
	\rput(0,1.0)\cobRIvbottom
	\rput(0,1.8)\cobRIhtop
	\psline(0.27639,0.02111)(0.27639,1.62111)
	\psline(0.13523,0.33679)(0.13523,1.93679)
	\psbezier(0.35637,0.97164)(0.37,0.1)(0.8,0.5)(0.8,0.2)
	\psbezier(0.35637,0.97164)(0.37,1.6)(0.8,1.6)(0.8,1.8)
	\psellipticarc(0.65,0.75)(0.2,0.2){-180}{0}
	\psbezier(0.5,1)(0.5,0.9)(0.45,0.85)(0.45,0.75)
	\psbezier(0.8,1)(0.8,0.9)(0.85,0.85)(0.85,0.75)
	\psellipticarc(0.65,0.78)(0.1,0.1){-170}{-10}
	\psellipticarc(0.65,0.65)(0.1,0.1){40}{140}
	\psbezier(0.5,1)(0.5,1.4)(0.8,1.4)(0.8,1)
}

\defPicCob{R1-GF-1}{%
	\rput(0,0.2)\cobRIhbottom
	\rput(0,1.0)\cobRIvbottom
	\rput(0,1.8)\cobRIhtop
	\psline(0.27639,0.02111)(0.27639,1.62111)
	\psline(0.13523,0.33679)(0.13523,1.93679)
	\psbezier(0.35637,0.97164)(0.37,1.6)(0.8,1.6)(0.8,1.8)
	\psline(0.8,0.2)(0.8,1.0)
	\psbezier(0.35637,0.97164)(0.37,0.5)(0.49,0.5)(0.5,1)
	\psbezier(0.5,1)(0.5,1.4)(0.8,1.4)(0.8,1)
}

\defPicCob{R1-GF-2}{%
	\rput(0,0.2)\cobRIhbottom
	\rput(0,1.8)\cobRIhtop
	\rput(0,1.0){%
		\psclip{\pspolygon[linestyle=none](0,-0.2)(0.4,-0.2)(0.4,-0.028)(0.277,-0.028)(0.277,0.2)(0,0.2)}
			\psbezier(0.27639,-0.17889)(0.4,0)(0.4,0)(0.13523,0.13679)
		\endpsclip
		\psclip{\psframe[linestyle=none](0.2763,-0.0284)(0.3564,0.2)}
			\psdash\psbezier(0.27639,-0.17889)(0.4,0)(0.4,0)(0.13523,0.13679)
		\endpsclip
	}
	\psline(0.27639,0.02111)(0.27639,1.62111)
	\psline(0.13523,0.33679)(0.13523,1.93679)
	\psbezier(0.35637,0.97164)(0.37,0.1)(0.8,0.5)(0.8,0.2)
	\psbezier(0.35637,0.97164)(0.37,1.6)(0.8,1.6)(0.8,1.8)
}

\defPicCob{R1-GF-3}{%
	\rput(0,0.2)\cobRIhbottom
	\rput(0,1.0)\cobRIhbottom
	\rput(0,1.8)\cobRIhtop
	\psline(0.27639,0.02111)(0.27639,1.62111)
	\psline(0.13523,0.33679)(0.13523,1.93679)
	\psline(0.80000,0.20000)(0.80000,1.80000)
}

\def\cobRIFourTubes{%
	\rput(0,0.2)\cobRIvbottom
	\rput(0,1.8)\cobRIvtop
	\psline(0.27639,0.02111)(0.27639,1.62111)
	\psline(0.13523,0.33679)(0.13523,1.93679)
	\psline(0.35637,1.77164)(0.35637,0.6)
	\psellipticarc(0.85,0.8)(0.08,0.05){-180}{0}
	\psellipticarc(0.55,0.8)(0.08,0.05){-180}{0}
	\psdash\psellipticarc(0.85,0.8)(0.08,0.05){0}{180}
	\psdash\psellipticarc(0.55,0.8)(0.08,0.05){0}{180}
	\psbezier(0.35637,0.6)(0.35637,0.5)(0.47,0.6)(0.47,0.8)
	\psbezier(0.35637,0.2)(0.35637,0.3)(0.63,0.4)(0.63,0.8)
	\psbezier(0.5,0.2)(0.5,0.3)(0.77,0.4)(0.77,0.8)
	\psbezier(0.8,0.2)(0.8,0.3)(0.93,0.4)(0.93,0.8)
	\psellipticarc(0.85,1.4)(0.08,0.05){-180}{0}
	\psellipticarc(0.55,1.4)(0.08,0.05){-180}{0}
	\psdash\psellipticarc(0.85,1.4)(0.08,0.05){0}{180}
	\psdash\psellipticarc(0.55,1.4)(0.08,0.05){0}{180}
	\psbezier(0.5,1.8)(0.5,1.6)(0.47,1.6)(0.47,1.4)
	\psbezier(0.8,1.8)(0.8,1.6)(0.93,1.6)(0.93,1.4)
	\psellipticarc(0.7,1.4)(0.07,0.15){0}{180}
}

\defPicCob{R1-FG-0}{%
	\cobRIFourTubes
	\psline(0.47,0.8)(0.47,1.4)
	\psline(0.63,0.8)(0.63,1.4)
	\psbezier(0.77,0.8)(0.77,1.1)(0.93,1.1)(0.93,0.8)
	\psbezier(0.77,1.4)(0.77,1.1)(0.93,1.1)(0.93,1.4)
}

\defPicCob{R1-FG-1}{%
	\cobRIFourTubes
	\psbezier(0.47,0.8)(0.47,1.1)(0.63,1.1)(0.63,0.8)
	\psbezier(0.47,1.4)(0.47,1.1)(0.63,1.1)(0.63,1.4)
	\psline(0.77,0.8)(0.77,1.4)
	\psline(0.93,0.8)(0.93,1.4)
}

\defPicCob{R1-FG-2}{%
	\cobRIFourTubes
	\psbezier(0.47,1.4)(0.47,1.1)(0.63,1.1)(0.63,1.4)
	\psbezier(0.77,1.4)(0.77,1.2)(0.93,1.2)(0.93,1.4)
	\psbezier(0.47,0.8)(0.47,1.2)(0.93,1.2)(0.93,0.8)
	\psbezier(0.63,0.8)(0.63,1.0)(0.77,1.0)(0.77,0.8)
}

\defPicCob{R1-FG-3}{%
	\cobRIFourTubes
	\psbezier(0.47,1.4)(0.47,1.0)(0.93,1.0)(0.93,1.4)
	\psbezier(0.63,1.4)(0.63,1.2)(0.77,1.2)(0.77,1.4)
	\psbezier(0.47,0.8)(0.47,1.1)(0.63,1.1)(0.63,0.8)
	\psbezier(0.77,0.8)(0.77,1.0)(0.93,1.0)(0.93,0.8)
}

\defPicCob{R1-FG-4}{%
	\rput(0,0.2)\cobRIvbottom
	\rput(0,1.0)\cobRIhbottom
	\rput(0,1.8)\cobRIvtop
	\psline(0.27639,0.02111)(0.27639,1.62111)
	\psline(0.13523,0.33679)(0.13523,1.93679)
	\psbezier(0.35637,0.17164)(0.37,0.8)(0.8,0.8)(0.8,1)
	\psline(0.8,1)(0.8,1.8)
	\psbezier(0.5,0.2)(0.5,0.6)(0.8,0.6)(0.8,0.2)
	\psbezier(0.35637,1.77164)(0.37,1.3)(0.49,1.3)(0.5,1.8)
}

\defPicCob{R1-FG-5}{%
	\rput(0,0.2)\cobRIvbottom
	\rput(0,1.0)\cobRIvbottom
	\rput(0,1.8)\cobRIvtop
	\psline(0.27639,0.02111)(0.27639,1.62111)
	\psline(0.13523,0.33679)(0.13523,1.93679)
	\psline(0.35637,0.17164)(0.35637,1.77164)
	\psline(0.5,0.2)(0.5,1.8)
	\psline(0.8,0.2)(0.8,1.8)
}

\defPicCob{R1-FG-6}{%
	\rput(0,0.2)\cobRIvbottom
	\rput(0,1.0)\cobRIhbottom
	\rput(0,1.8)\cobRIvtop
	\psline(0.27639,0.02111)(0.27639,1.62111)
	\psline(0.13523,0.33679)(0.13523,1.93679)
	\psline(0.8,0.2)(0.8,1)
	\psbezier(0.35637,1.77164)(0.37,1.2)(0.8,1.2)(0.8,1)
	\psbezier(0.35637,0.17164)(0.37,0.7)(0.49,0.7)(0.5,0.2)
	\psbezier(0.5,1.8)(0.5,1.4)(0.8,1.4)(0.8,1.8)
}

\defPicCob{R1-FG-7}{%
	\rput(0,0.2)\cobRIvbottom
	\rput(0,1.0)\cobRIhbottom
	\rput(0,1.8)\cobRIvtop
	\psline(0.27639,0.02111)(0.27639,1.62111)
	\psline(0.13523,0.33679)(0.13523,1.93679)
	\psbezier(0.35637,1.77164)(0.37,0.9)(0.8,1.3)(0.8,1)
	\psbezier(0.35637,0.17164)(0.37,0.8)(0.8,0.8)(0.8,1)
	\psellipticarc(0.65,1.55)(0.2,0.2){-180}{0}
	\psbezier(0.5,1.8)(0.5,1.7)(0.45,1.65)(0.45,1.55)
	\psbezier(0.8,1.8)(0.8,1.7)(0.85,1.65)(0.85,1.55)
	\psellipticarc(0.65,1.58)(0.1,0.1){-170}{-10}
	\psellipticarc(0.65,1.45)(0.1,0.1){40}{140}
	\psbezier(0.5,0.2)(0.5,0.6)(0.8,0.6)(0.8,0.2)
}

\def\cobRIIhT{%
	\psbezier(0.20,-0.160)(0.52,-0.03)(0.52,-0.03)(0.95,-0.087)
	\psbezier(0.05, 0.087)(0.48, 0.03)(0.48, 0.03)(0.80, 0.160)
}

\def\cobRIIvvT{%
	\psbezier(0.2,-0.16)(0.3,-0.05)(0.25,0)(0.05, 0.087)
	\psbezier(0.8, 0.16)(0.7, 0.05)(0.75,0)(0.95,-0.087)
	\psbezier(0.45,0.07)(0.61,0.10)(0.71,-0.04)(0.55,-0.07)
	\psbezier(0.45,0.07)(0.29,0.04)(0.39,-0.10)(0.55,-0.07)
}

\def\cobRIIvhT{%
	\psbezier(0.2,-0.16)(0.3,-0.05)(0.25,0)(0.05, 0.087)
	\psbezier(0.45,0.07)(0.29,0.04)(0.39,-0.10)(0.55,-0.07)
	\psbezier(0.80, 0.160)(0.75,0.00)(0.61, 0.10)(0.45, 0.07)
	\psbezier(0.95,-0.087)(0.73,0.06)(0.71,-0.04)(0.55,-0.07)
}

\def\cobRIIhvT{%
	\psbezier(0.8, 0.16)(0.7, 0.05)(0.75,0)(0.95,-0.087)
	\psbezier(0.45,0.07)(0.61,0.10)(0.71,-0.04)(0.55,-0.07)
	\psbezier(0.20,-0.160)(0.25, 0.00)(0.39,-0.10)(0.55,-0.07)
	\psbezier(0.05, 0.087)(0.27,-0.06)(0.29, 0.04)(0.45, 0.07)
}

\def\cobRIIhhT{%
	\psbezier(0.80, 0.160)(0.75,0.00)(0.61, 0.10)(0.45, 0.07)
	\psbezier(0.95,-0.087)(0.73,0.06)(0.71,-0.04)(0.55,-0.07)
	\psbezier(0.20,-0.160)(0.25, 0.00)(0.39,-0.10)(0.55,-0.07)
	\psbezier(0.05, 0.087)(0.27,-0.06)(0.29, 0.04)(0.45, 0.07)
}

\def\cobRIIhB{%
	\psbezier(0.20,-0.160)(0.52,-0.03)(0.52,-0.03)(0.95,-0.087)
	\psclip{\psframe[linestyle=none](0,-0.2)(0.2,0.2)}%
		\psbezier(0.05, 0.087)(0.48, 0.03)(0.48, 0.03)(0.80, 0.160)
	\endpsclip
	\psclip{\psframe[linestyle=none](0.2,-0.2)(1,0.2)}%
		\psdash\psbezier(0.05, 0.087)(0.48, 0.03)(0.48, 0.03)(0.80, 0.160)
	\endpsclip
}

\def\cobRIIvvB{%
	\psclip{\pspolygon[linestyle=none](0,-0.2)(0,0.2)(0.2,0.2)(0.2,-0.0664)(0.25,-0.0664)(0.368,0)(0.631,0)(0.75,0.0664)(1,0.0664)(1,-0.2)}%
		\psbezier(0.2,-0.16)(0.3,-0.05)(0.25,0)(0.05, 0.087)
		\psbezier(0.8, 0.16)(0.7, 0.05)(0.75,0)(0.95,-0.087)
		\psbezier(0.45,0.07)(0.61,0.10)(0.71,-0.04)(0.55,-0.07)
		\psbezier(0.45,0.07)(0.29,0.04)(0.39,-0.10)(0.55,-0.07)
	\endpsclip
	\psclip{\pspolygon[linestyle=none](0.2,0.2)(0.2,-0.0664)(0.25,-0.0664)(0.368,0)(0.631,0)(0.75,0.0664)(1,0.0664)(1,0.2)}
		\psdash\psbezier(0.2,-0.16)(0.3,-0.05)(0.25,0)(0.05, 0.087)
		\psdash\psbezier(0.8, 0.16)(0.7, 0.05)(0.75,0)(0.95,-0.087)
		\psdash\psbezier(0.45,0.07)(0.61,0.10)(0.71,-0.04)(0.55,-0.07)
		\psdash\psbezier(0.45,0.07)(0.29,0.04)(0.39,-0.10)(0.55,-0.07)
	\endpsclip
}

\def\cobRIIvhB{%
	\psclip{\pspolygon[linestyle=none](0,0.2)(0.2,0.2)(0.2,-0.0664)(0.25,-0.0664)(0.368,0)(1,0.11)(1,-0.2)(0,-0.2)}
		\psbezier(0.2,-0.16)(0.3,-0.05)(0.25,0)(0.05, 0.087)
		\psbezier(0.45,0.07)(0.29,0.04)(0.39,-0.10)(0.55,-0.07)
		\psbezier(0.95,-0.087)(0.73,0.06)(0.71,-0.04)(0.55,-0.07)
	\endpsclip
	\psclip{\pspolygon[linestyle=none](0.2,0.2)(0.2,-0.0664)(0.25,-0.0664)(0.368,0)(1,0.11)(1,0.2)}
		\psdash\psbezier(0.2,-0.16)(0.3,-0.05)(0.25,0)(0.05, 0.087)
		\psdash\psbezier(0.45,0.07)(0.29,0.04)(0.39,-0.10)(0.55,-0.07)
		\psdash\psbezier(0.80, 0.160)(0.75,0.00)(0.61, 0.10)(0.45, 0.07)
	\endpsclip
}

\def\cobRIIhvB{%
	\psclip{\pspolygon[linestyle=none](1,-0.2)(1,0.0664)(0.75,0.0664)(0.631,0)(0.2,-0.02)(0.2,0.2)(0,0.2)(0,-0.2)}
		\psbezier(0.8, 0.16)(0.7, 0.05)(0.75,0)(0.95,-0.087)
		\psbezier(0.45,0.07)(0.61,0.10)(0.71,-0.04)(0.55,-0.07)
		\psbezier(0.20,-0.160)(0.25, 0.00)(0.39,-0.10)(0.55,-0.07)
		\psbezier(0.05, 0.087)(0.27,-0.06)(0.29, 0.04)(0.45, 0.07)
	\endpsclip
	\psclip{\pspolygon[linestyle=none](1,0.2)(1,0.0664)(0.75,0.0664)(0.631,0)(0.2,-0.11)(0.2,0.2)}
		\psdash\psbezier(0.8, 0.16)(0.7, 0.05)(0.75,0)(0.95,-0.087)
		\psdash\psbezier(0.45,0.07)(0.61,0.10)(0.71,-0.04)(0.55,-0.07)
		\psdash\psbezier(0.05, 0.087)(0.27,-0.06)(0.29, 0.04)(0.45, 0.07)
	\endpsclip
}

\def\cobRIIhhB{%
	\psbezier(0.95,-0.087)(0.73,0.06)(0.71,-0.04)(0.55,-0.07)
	\psbezier(0.20,-0.160)(0.25, 0.00)(0.39,-0.10)(0.55,-0.07)
	\psclip{\psframe[linestyle=none](0,0.2)(0.2,-0.2)}
		\psbezier(0.05, 0.087)(0.27,-0.06)(0.29, 0.04)(0.45, 0.07)
	\endpsclip
	\psdash\psbezier(0.80, 0.160)(0.75,0.00)(0.61, 0.10)(0.45, 0.07)
	\psclip{\psframe[linestyle=none](0.2,0.2)(1,-0.2)}
		\psdash\psbezier(0.05, 0.087)(0.27,-0.06)(0.29, 0.04)(0.45, 0.07)
	\endpsclip
}

\defPicSmallCob{R2-d0*}{%
	\psline(0.05, 0.287)(0.05, 1.087)
	\psline(0.20, 0.040)(0.20, 0.840)
	\psdash\psline(0.80, 0.360)(0.80, 0.997)
	\psline(0.80, 0.997)(0.80, 1.160)
	\psline(0.95, 0.113)(0.95, 0.913)
	\psline(0.369,0.2)(0.369,1)
	\psline(0.25,0.1336)(0.25,0.9336)
	\rput(0,0.2)\cobRIIvhB
	\rput(0,1.0)\cobRIIvvT
	\psbezier(0.631,1)(0.631,0.3)(0.75,0.3)(0.75,1.0664)
}

\defPicSmallCob{R2-d*0}{%
	\psline(0.05, 0.287)(0.05, 1.087)
	\psline(0.20, 0.040)(0.20, 0.840)
	\psdash\psline(0.80, 0.360)(0.80, 0.987)
	\psline(0.80, 0.987)(0.80, 1.160)
	\psline(0.95, 0.113)(0.95, 0.913)
	\rput(0,0.2)\cobRIIvhB
	\rput(0,1.0)\cobRIIhhT
	\psbezier(0.369,0.2)(0.369,0.9)(0.25,0.9)(0.25,0.1336)
}

\defPicSmallCob{R2-d1*}{%
	\psline(0.05, 0.287)(0.05, 1.087)
	\psline(0.20, 0.040)(0.20, 0.840)
	\psdash\psline(0.80, 0.360)(0.80, 0.997)
	\psline(0.80, 0.997)(0.80, 1.160)
	\psline(0.95, 0.113)(0.95, 0.913)
	\rput(0,0.2)\cobRIIhhB
	\rput(0,1.0)\cobRIIhvT
	\psbezier(0.631,1)(0.631,0.3)(0.75,0.3)(0.75,1.0664)
}

\defPicSmallCob{R2-d*1}{%
	\psline(0.05, 0.287)(0.05, 1.087)
	\psline(0.20, 0.040)(0.20, 0.840)
	\psdash\psline(0.80, 0.360)(0.80, 0.997)
	\psline(0.80, 0.997)(0.80, 1.160)
	\psline(0.95, 0.113)(0.95, 0.913)
	\psline(0.631,0.2)(0.631,1)
	\psline(0.75,0.2664)(0.75,1.0664)
	\rput(0,0.2)\cobRIIvvB
	\rput(0,1.0)\cobRIIhvT
	\psbezier(0.369,0.2)(0.369,0.9)(0.25,0.9)(0.25,0.1336)
}

\defPicSmallCob{R2-h0*}{%
	\psline(0.05, 0.287)(0.05, 1.087)
	\psline(0.20, 0.040)(0.20, 0.840)
	\psdash\psline(0.80, 0.360)(0.80, 0.987)
	\psline(0.80, 0.987)(0.80, 1.160)
	\psline(0.95, 0.113)(0.95, 0.913)
	\psline(0.25,0.1336)(0.25,0.9336)
	\rput(0,0.2)\cobRIIvvB
	\rput(0,1.0)\cobRIIvhT
	\psbezier(0.75,0.2664)(0.75,0.9)(0.369,0.8)(0.369,1)
	\psbezier(0.369,0.2)(0.369,0.6)(0.631,0.6)(0.631,0.2)
}

\defPicSmallCob{R2-h*1}{%
	\psline(0.05, 0.287)(0.05, 1.087)
	\psline(0.20, 0.040)(0.20, 0.840)
	\psdash\psline(0.80, 0.360)(0.80, 0.997)
	\psline(0.80, 0.997)(0.80, 1.160)
	\psline(0.95, 0.113)(0.95, 0.913)
	\psline(0.75,0.2664)(0.75,1.0664)
	\rput(0,0.2)\cobRIIhvB
	\rput(0,1.0)\cobRIIvvT
	\psbezier(0.25,0.9336)(0.25,0.3)(0.631,0.4)(0.631,0.2)
	\psbezier(0.369,1)(0.369,0.6)(0.631,0.6)(0.631,1)
}

\defPicSmallCob{R2-F}{%
	\psline(0.05, 0.287)(0.05, 1.087)
	\psline(0.20, 0.040)(0.20, 0.840)
	\psdash\psline(0.80, 0.360)(0.80, 0.997)
	\psline(0.80, 0.997)(0.80, 1.160)
	\psline(0.95, 0.113)(0.95, 0.913)
	\rput(0,0.2)\cobRIIhB
	\rput(0,1.0)\cobRIIvvT
	\psbezier(0.369,1)(0.369,0.6)(0.631,0.6)(0.631,1)
	\psbezier(0.25,0.9336)(0.25,0.2336)(0.75,0.1664)(0.75,1.0664)
	\psbezier(0.534,0.148)(0.534,0.3)(0.52,0.4)(0.4923,0.4)
	\psdash\psbezier(0.466,0.252)(0.466,0.35)(0.475,0.4)(0.4923,0.4)
}

\defPicSmallCob{R2-G}{%
	\psline(0.05, 0.287)(0.05, 1.087)
	\psline(0.20, 0.040)(0.20, 0.840)
	\psdash\psline(0.80, 0.360)(0.80, 0.933)
	\psline(0.80, 0.933)(0.80, 1.160)
	\psline(0.95, 0.113)(0.95, 0.913)
	\rput(0,0.2)\cobRIIvvB
	\rput(0,1.0)\cobRIIhT
	\psbezier(0.369,0.2)(0.369,0.6)(0.631,0.6)(0.631,0.2)
	\psbezier(0.25,0.1336)(0.25,1.0336)(0.75,0.9664)(0.75,0.2664)
	\psbezier(0.534,0.948)(0.534,0.9)(0.52,0.8)(0.4923,0.8)
	\psclip{\psframe[linestyle=none](0.4,1.1)(0.5,0.935)}
		\psbezier(0.466,1.052)(0.466,0.85)(0.475,0.8)(0.4923,0.8)
	\endpsclip
	\psclip{\psframe[linestyle=none](0.4,0.7)(0.5,0.935)}
		\psdash\psbezier(0.466,1.052)(0.466,0.85)(0.475,0.8)(0.4923,0.8)
	\endpsclip
}

\def\cobRIIFourTubes{%
	\rput(0,0.2)\cobRIIvvB
	\rput(0,1.8)\cobRIIvvT
	\psline(0.05, 0.287)(0.05, 1.887)
	\psline(0.20, 0.040)(0.20, 1.640)
	\psdash\psline(0.80, 0.360)(0.80, 1.797)
	\psline(0.80, 1.797)(0.80, 1.960)
	\psline(0.95, 0.113)(0.95, 1.713)
	\psline(0.25,1.7336)(0.25,0.8)
	\psline(0.75,0.2664)(0.75,1.3)
	\psellipticarc(0.39,1.3)(0.06,0.04){-180}{0}
	\psellipticarc(0.61,1.3)(0.06,0.04){-180}{0}
	\psdash\psellipticarc(0.39,1.3)(0.06,0.04){0}{180}
	\psdash\psellipticarc(0.61,1.3)(0.06,0.04){0}{180}
	\psbezier(0.361,1.8)(0.361,1.5)(0.33,1.6)(0.33,1.3)
	\psbezier(0.639,1.8)(0.639,1.6)(0.45,1.5)(0.45,1.3)
	\psbezier(0.75,1.8664)(0.75,1.5)(0.55,1.5)(0.55,1.3)
	\psbezier(0.75,1.3)(0.75,1.6)(0.67,1.4)(0.67,1.3)
	\psellipticarc(0.39,0.8)(0.06,0.04){-180}{0}
	\psellipticarc(0.61,0.8)(0.06,0.04){-180}{0}
	\psdash\psellipticarc(0.39,0.8)(0.06,0.04){0}{180}
	\psdash\psellipticarc(0.61,0.8)(0.06,0.04){0}{180}
	\psbezier(0.361,0.2)(0.361,0.6)(0.33,0.5)(0.33,0.8)
	\psbezier(0.45,0.8)(0.45,0.5)(0.639,0.4)(0.639,0.2)
	\psclip{\pspolygon[linestyle=none](0.36,0.3)(0.35,0.7)(0.42,0.7)(0.62,0.3)}
		\psdash\psbezier(0.25,0.1336)(0.25,0.4)(0.67,0.4)(0.67,0.8)
		\psdash\psbezier(0.25,0.8)(0.25,0.2)(0.55,0.6)(0.55,0.8)
	\endpsclip
	\psclip{\pspolygon[linestyle=none](0,-0.2)(0,1)(0.35,1)(0.35,0.7)(0.36,0.3)(0.62,0.3)(0.42,0.7)(0.42,1)(1,1)(1,-0.2)}
		\psbezier(0.25,0.1336)(0.25,0.4)(0.67,0.4)(0.67,0.8)
		\psbezier(0.25,0.8)(0.25,0.2)(0.55,0.6)(0.55,0.8)
	\endpsclip
}

\defPicCob{R2-FG-0a}{%
	\cobRIIFourTubes
	\psline(0.33,0.8)(0.33,1.3)
	\psline(0.45,0.8)(0.45,1.3)
	\psbezier(0.55,0.8)(0.55,1)(0.67,1)(0.67,0.8)
	\psbezier(0.55,1.3)(0.55,1.05)(0.67,1.05)(0.67,1.3)
}

\defPicCob{R2-FG-0b}{%
	\cobRIIFourTubes
	\psbezier(0.33,0.8)(0.33,1)(0.45,1)(0.45,0.8)
	\psbezier(0.33,1.3)(0.33,1.05)(0.45,1.05)(0.45,1.3)
	\psline(0.55,0.8)(0.55,1.3)
	\psline(0.67,0.8)(0.67,1.3)
}

\defPicCob{R2-FG-0c}{%
	\cobRIIFourTubes
	\psbezier(0.33,0.8)(0.33,1.1)(0.67,1.1)(0.67,0.8)
	\psbezier(0.45,0.8)(0.45,0.9)(0.55,0.9)(0.55,0.8)
	\psbezier(0.33,1.3)(0.33,1.0)(0.45,1.0)(0.45,1.3)
	\psbezier(0.55,1.3)(0.55,1.1)(0.67,1.1)(0.67,1.3)
}

\defPicCob{R2-FG-0d}{%
	\cobRIIFourTubes
	\psbezier(0.33,1.3)(0.33,1.0)(0.67,1.0)(0.67,1.3)
	\psbezier(0.45,1.3)(0.45,1.2)(0.55,1.2)(0.55,1.3)
	\psbezier(0.33,0.8)(0.33,1.05)(0.45,1.05)(0.45,0.8)
	\psbezier(0.55,0.8)(0.55,0.95)(0.67,0.95)(0.67,0.8)
}

\defPicCob{R2-FG-1a}{%
	\rput(0,0.2)\cobRIIvvB
	\rput(0,1.0)\cobRIIvvB
	\rput(0,1.8)\cobRIIvvT
	\psline(0.05, 0.287)(0.05, 1.887)
	\psline(0.20, 0.040)(0.20, 1.640)
	\psdash\psline(0.80, 0.360)(0.80, 1.797)
	\psline(0.80, 1.797)(0.80, 1.960)
	\psline(0.95, 0.113)(0.95, 1.713)
	\psline(0.25,0.1336)(0.25,1.7336)
	\psline(0.369,0.2)(0.369,1.8)
	\psline(0.631,0.2)(0.631,1.8)
	\psline(0.75,0.2664)(0.75,1.8664)
}

\defPicCob{R2-FG-1b}{%
	\rput(0,0.2)\cobRIIvvB
	\rput(0,1.8)\cobRIIvvT
	\psline(0.05, 0.287)(0.05, 1.887)
	\psline(0.20, 0.040)(0.20, 1.640)
	\psdash\psline(0.80, 0.360)(0.80, 1.797)
	\psline(0.80, 1.797)(0.80, 1.960)
	\psline(0.95, 0.113)(0.95, 1.713)
	\rput(0,0.75)\cobRIIhvB
	\rput(0,1.25)\cobRIIvhB
	\psbezier(0.25,1.1836)(0.25,0.65)(0.369,0.75)(0.369,1.25)
	\psbezier(0.75,0.8164)(0.75,1.25)(0.631,1.15)(0.631,0.75)
	\psbezier(0.361,0.2)(0.361,0.5)(0.639,0.5)(0.639,0.2)
	\psbezier(0.361,1.8)(0.361,1.5)(0.639,1.5)(0.639,1.8)
	\psline(0.25,1.7336)(0.25,1.1836)
	\psline(0.75,0.2664)(0.75,0.8164)
	\psbezier(0.25,0.1336)(0.25,0.55)(0.631,0.55)(0.631,0.75)
	\psbezier(0.75,1.8664)(0.75,1.4)(0.369,1.45)(0.369,1.25)
}

\defPicCob{R2-FG-1c}{%
	\rput(0,0.2)\cobRIIvvB
	\rput(0,1.0)\cobRIIhvB
	\rput(0,1.8)\cobRIIvvT
	\psline(0.05, 0.287)(0.05, 1.887)
	\psline(0.20, 0.040)(0.20, 1.640)
	\psdash\psline(0.80, 0.360)(0.80, 1.797)
	\psline(0.80, 1.797)(0.80, 1.960)
	\psline(0.95, 0.113)(0.95, 1.713)
	\psline(0.75,0.2664)(0.75,1.8664)
	\psbezier(0.369,0.2)(0.369,0.9)(0.25,0.9)(0.25,0.1336)
	\psline(0.631,0.2)(0.631,1.0)
	\psbezier(0.25,1.7336)(0.25,1.1)(0.631,1.2)(0.631,1.0)
	\psbezier(0.369,1.8)(0.369,1.4)(0.631,1.4)(0.631,1.8)
}

\defPicCob{R2-FG-1d}{%
	\rput(0,0.2)\cobRIIvvB
	\rput(0,1.0)\cobRIIvhB
	\rput(0,1.8)\cobRIIvvT
	\psline(0.05, 0.287)(0.05, 1.887)
	\psline(0.20, 0.040)(0.20, 1.640)
	\psdash\psline(0.80, 0.360)(0.80, 1.797)
	\psline(0.80, 1.797)(0.80, 1.960)
	\psline(0.95, 0.113)(0.95, 1.713)
	\psline(0.25,0.1336)(0.25,1.7336)
	\psbezier(0.75,0.2664)(0.75,0.9)(0.369,0.8)(0.369,1)
	\psbezier(0.369,0.2)(0.369,0.6)(0.631,0.6)(0.631,0.2)
	\psbezier(0.631,1.8)(0.631,1.1)(0.75,1.1)(0.75,1.80664)
	\psline(0.361,1.0)(0.361,1.8)
}

\defPicCob{R2-FG-2}{%
	\rput(0,0.2)\cobRIIvvB
	\rput(0,1.0)\cobRIIhB
	\rput(0,1.8)\cobRIIvvT
	\psline(0.05, 0.287)(0.05, 1.887)
	\psline(0.20, 0.040)(0.20, 1.640)
	\psdash\psline(0.80, 0.360)(0.80, 1.797)
	\psline(0.80, 1.797)(0.80, 1.960)
	\psline(0.95, 0.113)(0.95, 1.713)
	\psbezier(0.369,0.2)(0.369,0.6)(0.631,0.6)(0.631,0.2)
	\psbezier(0.25,0.1336)(0.25,1.0336)(0.75,0.9664)(0.75,0.2664)
	\psbezier(0.534,0.948)(0.534,0.9)(0.52,0.8)(0.4923,0.8)
	\psdash\psbezier(0.466,1.052)(0.466,0.85)(0.475,0.8)(0.4923,0.8)
	\psbezier(0.369,1.8)(0.369,1.4)(0.631,1.4)(0.631,1.8)
	\psbezier(0.25,1.7336)(0.25,1.0336)(0.75,0.9664)(0.75,1.8664)
	\psbezier(0.534,0.948)(0.534,1.1)(0.52,1.2)(0.4923,1.2)
	\psdash\psbezier(0.466,1.052)(0.466,1.15)(0.475,1.2)(0.4923,1.2)
}

\def\pictRelS{\begin{centerpict}(-0.6,-0.5)(0.6,0.5)
	\psset{linewidth=0.5pt,dash=1pt 1.5pt,dimen=middle}%
	\pscircle(0,0){0.5}
	\psellipticarc(0,0)(0.5,0.15){-180}{0}
	\psdash\psellipticarc(0,0)(0.5,0.15){0}{180}
\end{centerpict}}

\def\pictRelT{\begin{centerpict}(-0.1,-0.7)(1.3,0.7)
	\psset{linewidth=0.5pt,dash=1pt 1.5pt,dimen=middle}
	\psdash\psellipticarc(0.2,0)(0.2,0.1){0}{180}
	\psdash\psellipticarc(1,0)(0.2,0.1){0}{180}
	\psellipticarc(0.2,0)(0.2,0.1){-180}{0}
	\psellipticarc(1,0)(0.2,0.1){-180}{0}
	\psbezier(0.0,0)(0.00,0.9)(1.20,0.9)(1.2,0)
	\psbezier(0.4,0)(0.35,0.4)(0.85,0.4)(0.8,0)
	\psline[linewidth=0.3pt,arrowsize=3pt]{<-}(0.45,0.1)(0.75,0.1)
	\psbezier(0.0,0)(0.00,-0.9)(1.20,-0.9)(1.2,0)
	\psbezier(0.4,0)(0.35,-0.4)(0.85,-0.4)(0.8,0)
	\psline[linewidth=0.3pt,border=1pt,arrowsize=3pt]{->}(0.40,-0.4)(0.70,-0.1)
\end{centerpict}}

\def\pictRelTuCircles{%
	\psset{linewidth=0.5pt,dash=1pt 1.5pt,dimen=middle}
	\psellipticarc(0.2,0.1)(0.2,0.07\psxunit){-180}{0}
	\psellipticarc(1.0,0.1)(0.2,0.07\psxunit){-180}{0}
	\psdash\psellipticarc(0.2,0.1)(0.2,0.07\psxunit){0}{180}
	\psdash\psellipticarc(1.0,0.1)(0.2,0.07\psxunit){0}{180}
	\psellipse(0.2,1.4)(0.2,0.07\psxunit)
	\psellipse(1.0,1.4)(0.2,0.07\psxunit)
}

\def\pictRelTuL{\begin{centerpict}(-0.1,0)(1.3,1.5)
	\pictRelTuCircles
	\psline(0.0,0.1)(0.0,1.4)\psline(0.4,0.1)(0.4,1.4)
	\psbezier(0.8,0.1)(0.8,0.7)(1.2,0.7)(1.2,0.1)
	\psbezier(0.8,1.4)(0.8,0.8)(1.2,0.8)(1.2,1.4)
\end{centerpict}}

\def\pictRelTuR{\begin{centerpict}(-0.1,0)(1.3,1.5)
	\pictRelTuCircles
	\psbezier(0.0,0.1)(0.0,0.7)(0.4,0.7)(0.4,0.1)
	\psbezier(0.0,1.4)(0.0,0.8)(0.4,0.8)(0.4,1.4)
	\psline(0.8,0.1)(0.8,1.4)\psline(1.2,0.1)(1.2,1.4)
\end{centerpict}}

\def\pictRelTuB{\begin{centerpict}(-0.1,0)(1.3,1.5)
	\pictRelTuCircles
	\psbezier(0.0,0.1)(0.0,1.0)(1.2,1.0)(1.2,0.1)
	\psbezier(0.4,0.1)(0.35,0.5)(0.85,0.5)(0.8,0.1)
	\psline[linewidth=0.3pt,arrowsize=3pt]{<-}(0.45,0.2)(0.75,0.2)
	\psbezier(0.0,1.4)(0.0,1.0)(0.4,1.0)(0.4,1.4)
	\psbezier(0.8,1.4)(0.8,0.8)(1.2,0.8)(1.2,1.4)
\end{centerpict}}

\def\pictRelTuT{\begin{centerpict}(-0.1,0)(1.3,1.5)
	\pictRelTuCircles
	\psbezier(0.0,1.4)(0.0,0.5)(1.2,0.5)(1.2,1.4)
	\psbezier(0.4,1.4)(0.35,1.0)(0.85,1.0)(0.8,1.4)
	\psline[linewidth=0.3pt,border=1pt,arrowsize=3pt]{->}(0.40,1.0)(0.70,1.3)
	\psbezier(0.0,0.1)(0.0,0.5)(0.4,0.5)(0.4,0.1)
	\psbezier(0.8,0.1)(0.8,0.7)(1.2,0.7)(1.2,0.1)
\end{centerpict}}

\COBnewobject deathNoBirth|D-B[1,0]<0,1>(0,1){%
	\COButurn(0,1,0){0.5}
}

\COBnewobject deathDotNoBirth|D*-B[1,0]<0,1>(0,1){%
	\COButurn(0,1,0){0.5}
	\psdot[dotsize=3pt](0.5,0.2)
}

\COBnewobject deathDotsNoBirth|D**-B[1,0]<0,1>(0,1){%
	\COButurn(0,1,0){0.3}
	\pscircle(0.5,0.5){0.25}
	\psdot[dotsize=3pt](0.5,0.58)
	\psdot[dotsize=3pt](0.5,0.42)
}

\COBnewobject deathBirth|DB[1,1]<0,1>(0,1){%
	\COButurn(0,1,0){0.5}
	\COButurn(0,1,1){0.5}
}

\COBnewobject deathDotBirth|D*B[1,1]<0,1>(0,1){%
	\COButurn(0,1,0){0.5}
	\COButurn(0,1,1){0.5}
	\psdot[dotsize=3pt](0.5,0.2)
}

\COBnewobject deathBirthDot|DB*[1,1]<0,1>(0,1){%
	\COButurn(0,1,0){0.5}
	\COButurn(0,1,1){0.5}
	\psdot[dotsize=3pt](0.5,0.8)
}

\COBnewobject deathDotsBirth|D**B[1,1]<0,1>(0,1){%
	\COButurn(0,1,0){0.3}
	\COButurn(0,1,1){0.7}
	\pscircle(0.5,0.5){0.25}
	\psdot[dotsize=3pt](0.5,0.58)
	\psdot[dotsize=3pt](0.5,0.42)
}

\COBnewobject mergeDot|M*[2,1]<1,1>(0,3){%
	\COBsline(0,0)(1,1)%
	\COBsline(3,0)(2,1)%
	\COButurn(1,2,0){0.5}%
	\rput[c]{*0}(1.5,0.48){$\scriptstyle#1$}
	\psdot[dotsize=3pt](1.5,0.7)
}
\COBnewobject mergeDotFrRight|M*-R[2,1]<1,1>(0,3){%
	\COBsline(0,0)(1,1)%
	\COBsline(3,0)(2,1)%
	\COButurn(1,2,0){0.5}%
	\rput[c]{*0}(1.5,0.48){$\scriptstyle#1$}
	\psarrow(1.125,0.125)(1.875,0.125)%
	\psdot[dotsize=3pt](1.5,0.7)
}
\COBnewobject mergeDotFrLeft|M*-L[2,1]<1,1>(0,3){%
	\COBsline(0,0)(1,1)%
	\COBsline(3,0)(2,1)%
	\COButurn(1,2,0){0.5}%
	\rput[c]{*0}(1.5,0.48){$\scriptstyle#1$}
	\psarrow(1.875,0.125)(1.125,0.125)%
	\psdot[dotsize=3pt](1.5,0.7)
}
\COBnewobject mergeDotFrBack|M*-B[2,1]<1,1>(0,3){%
	\COBsline(0,0)(1,1)%
	\COBsline(3,0)(2,1)%
	\COButurn(1,2,0){0.5}%
	\pscarrow(1.1,0.5)(1.2,0.1)(1.6,0.1)(1.75,0.25)
	\rput[c]{*0}(1.5,0.48){$\scriptstyle#1$}
	\psdot[dotsize=3pt](1.5,0.7)
}
\COBnewobject mergeDotFrFront|M*-F[2,1]<1,1>(0,3){%
	\COBsline(0,0)(1,1)%
	\COBsline(3,0)(2,1)%
	\COButurn(1,2,0){0.5}%
	\pscarrow(1.75,0.25)(1.6,0.1)(1.2,0.1)(1.1,0.5)
	\rput[c]{*0}(1.5,0.48){$\scriptstyle#1$}
	\psdot[dotsize=3pt](1.5,0.7)
}

\COBnewobject dotCylinderVert|I*[1,1]<0,1>(0,1){%
	\psline(0,0)(0,1)%
	\psline(1,0)(1,1)%
	\psdot[dotsize=3pt](0.5,0.5)
}

\COBnewobject shortDotCylinderVert|sI*[1,1]<0,0.5>(0,1){%
	\psline(0,0)(0,0.5)%
	\psline(1,0)(1,0.5)%
	\psdot[dotsize=3pt](0.5,0.25)
}

\COBnewobject cutCyliderVert|cI[1,1]<0,1>(-0.3,1.3){%
	\psline(0,0)(0,1)%
	\psline(1,0)(1,1)%
	\psline[linestyle=dashed,dash=2pt 2pt,linewidth=1pt,linecolor=red](-0.3,0.5)(1.3,0.5)%
}

\defPicSmallCob{PassLeftID}{%
	\pscustom{%
		\moveto(0.812,0.844)	%	dir: <0.5,0.2>
		\lineto(0.812,0.044)\lineto(0.188,0.356)\lineto(0.188,1.156)
		\curveto(0.375,1.062)(0.644,1.148)(0.800,1.070) %(0.5,0.95)
		\curveto(0.956,0.992)(0.625,0.938)(0.812,0.844)\stroke}
	\psbezier(0.843,1.025)(0.843,0.5)(0.5,0.7)(0.5,0.2)
	\psellipse(0.47,0.95)(0.15,0.08)
	\psellipse(0.27,0.15)(0.15,0.08)
	\psbezier(0.12,0.15)(0.12,0.65)(0.32,0.6)(0.32,0.95)
	\psbezier(0.42,0.15)(0.42,0.65)(0.62,0.6)(0.62,0.95)
}

\defPicSmallCob{PassRightID}{%
	\pscustom{%
		\moveto(0.188,0.356)	%	dir: <0.5,0.2>
		\lineto(0.188,1.156)\lineto(0.812,0.844)\lineto(0.812,0.044)
		\curveto(0.625,0.138)(0.356,0.052)(0.200,0.130) %(0.5,0.95)
		\curveto(0.044,0.208)(0.375,0.262)(0.188,0.356)\stroke}
	\psbezier(0.157,0.175)(0.157,0.65)(0.5,0.5)(0.5,1.0)
	\psellipse(0.53,0.25)(0.15,0.08)
	\psellipse(0.73,1.05)(0.15,0.08)
	\psbezier(0.88,1.05)(0.88,0.55)(0.68,0.6)(0.68,0.25)
	\psbezier(0.58,1.05)(0.58,0.55)(0.38,0.6)(0.38,0.25)
}

\defPicSmallCob{PassSomeStrangeCobordism}{%
	\pscustom{%
		\moveto(0.188,1.156)	%	dir: <0.5,0.2>
		\lineto(0.188,0.356)\lineto(0.812,0.044)\lineto(0.812,0.844)
		\curveto(0.625,0.938)(0.356,0.852)(0.200,0.930) %(0.5,0.95)
		\curveto(0.044,1.008)(0.375,1.062)(0.188,1.156)\stroke}
	\psellipse(0.53,1.05)(0.15,0.08)
	\psbezier(0.68,1.05)(0.68,0.5)(0.5  ,0.6)(0.5  ,0.2)
	\psbezier(0.38,1.05)(0.38,0.4)(0.157,0.8)(0.157,0.975)
	\psline(0.247,1.096)(0.247,0.912)
	\psdash\psline(0.247,0.912)(0.247,0.724)
	\psbezier(0.332,0.714)(0.432,0.764)(0.569,0.439)(0.519,0.389)
	\psclip{\pscustom[linestyle=none]{%
			\moveto(0.332,0.714)
			\curveto(0.432,0.764)(0.569,0.439)(0.519,0.389)
			\lineto(0.519,0)\lineto(0,0)\lineto(0,0.764)
			\closepath}}
		\psellipse(0.56,0.83)(0.2,0.35)
	\endpsclip
	\psellipse(0.27,0.15)(0.15,0.08)
	\psbezier(0.12,0.15)(0.12,0.6)(0.42,0.6)(0.42,0.15)
}

\defPicSmallCob{PassLeftNonID}{%
	\pscustom{%				edge of the curtain: from bottom corner anticlockwise
		\moveto(0.35,0.275)\lineto(0.812,0.044)\lineto(0.812,0.844)
		\curveto(0.625,0.938)(0.356,0.852)(0.200,0.930) %(0.5,0.95)
		\curveto(0.044,1.008)(0.375,1.062)(0.188,1.156)
		\lineto(0.188,1.022)\stroke}
	\psdash\psline(0.188,1.022)(0.188,0.356)(0.35,0.275)
	\psellipse(0.53,1.05)(0.15,0.08)
	\psclip{\pscustom[linestyle=none]{%
				\moveto(0.812,0.844)
				\curveto(0.625,0.938)(0.356,0.852)(0.2,0.93)
				\lineto(0.2,1.2)\lineto(0.812,1.2)\closepath}}
		\psbezier(0.38,1.05)(0.38,0.65)(0.68,0.65)(0.68,1.05)
	\endpsclip
	\psclip{\pscustom[linestyle=none]{%
				\moveto(0.812,0.844)
				\curveto(0.625,0.938)(0.356,0.852)(0.2,0.93)
				\lineto(0.2,0)\lineto(0.812,0)\closepath}}
		\psdash\psbezier(0.38,1.05)(0.38,0.65)(0.68,0.65)(0.68,1.05)
	\endpsclip
	\psellipticarc(0.2,0.15)(0.15,0.08){180}{360}
	\psdash\psellipticarc(0.2,0.15)(0.15,0.08){0}{180}
	\psbezier(0.05,0.15)(0.05,0.5)(0.157,0.5)(0.157,0.975)
	\psbezier(0.35,0.15)(0.35,0.4)(0.5  ,0.6)(0.5  ,0.2)
	\psline(0.247,1.096)(0.247,0.912)
}

\defPicSmallCob{PassRightNonID}{%
	\psbezier(0.65,1.05)(0.65,0.8)(0.5,0.6)(0.5,1.0)
	\psclip{\psframe[linestyle=none,fillstyle=solid,fillcolor=white](0.555,0.5)(0.8,0.93)}%
		\psdash\psbezier(0.65,1.05)(0.65,0.8)(0.5  ,0.6)(0.5  ,1.0)
	\endpsclip
	\pscustom{%
		\moveto(0.182,0.356)
		\curveto(0.375,0.262)(0.644,0.348)(0.800,0.270) %(0.5,0.95)
		\curveto(0.956,0.192)(0.625,0.138)(0.812,0.044)
		\stroke}%
	\psclip{\psframe[linestyle=none,fillstyle=solid,fillcolor=white](0.34,0.2)(0.6,0.5)}%
		\psdash\psbezier(0.182,0.356)(0.375,0.262)(0.644,0.348)(0.800,0.270)
	\endpsclip
	\psclip{\pspolygon[linestyle=none,fillstyle=solid,fillcolor=white]%
			(0.74,0.104)(0.81,0.104)(0.81,0.23)(0.84,0.23)(0.84,0.5)(0.74,0.5)}%
		\psdash\pscustom{%
			\moveto(0.182,0.356)
			\curveto(0.375,0.262)(0.644,0.348)(0.800,0.270) %(0.5,0.95)
			\curveto(0.956,0.192)(0.625,0.138)(0.812,0.044)
			\stroke}%
	\endpsclip
	\psline(0.812,0.044)(0.812,0.844)(0.182,1.156)(0.182,0.356)
	\pscustom{%
		\moveto(0.753,0.104)
		\curveto(0.753,0.4)(0.705,0.673)(0.555,0.773)
		\curveto(0.405,0.873)(0.305,0.508)(0.305,0.318)
		\stroke}%
	\psellipse(0.8,1.05)(0.15,0.08)
	\psbezier(0.95,1.05)(0.95,0.7)(0.843,0.7)(0.843,0.225)
	\psellipticarc(0.47,0.15)(0.15,0.08){180}{360}
	\psdash\psellipticarc(0.47,0.15)(0.15,0.08){0}{180}
	\psbezier(0.62,0.15)(0.62,0.55)(0.32,0.55)(0.32,0.15)
}

\def\pictPassTotalBdry#1{%	0.8 or 1.6
	\psset{xunit=1.2\psunit,linewidth=0.4pt,dotsep=1pt,%
	       dash=1pt 1pt,linestyle=solid,dimen=mid}%
	\psellipticarc(0.35,0.2)(0.15,0.1){-180}{0}
	\psdash\psellipticarc(0.35,0.2)(0.15,0.1){0}{180}
	\psbezier(0.72361,0.37889)(0.6,0.2)(0.6,0.2)(0.86477,0.06321)
	\psclip{\psframe[linestyle=none,fillstyle=solid,fillcolor=white]%
				(0.6,0.2284)(0.8,0.5)}%
		\psdash\psbezier(0.72361,0.37889)(0.6,0.2)(0.6,0.2)(0.86477,0.06321)
	\endpsclip
	\rput(0,#1){%
		\psellipse(0.35,0.2)(0.15,0.1)
		\psbezier(0.72361,0.37889)(0.6,0.2)(0.6,0.2)(0.86477,0.06321)
		\psline{c-c}(0.72361,0.15   )(0.72361,0.37889)
	}%
	\rput(0.72361,0.15){\psdash\psline{c-c}(0,0.22889)(0,#1)}
	\rput(0.86477,0.06321){\psline{c-c}(0,0)(0,#1)}
}
\def\pictPassLower{\begin{centerpict}(-0.1,0)(1.3,1.4)
	\pictPassTotalBdry{1}%
	\psbezier(0.2,1.2)(0.2,0.8)(0.5,0.8)(0.5,1.2)
	\psbezier(0.2,0.2)(0.2,0.6)(0.64363,0.6)(0.64363,1.22836)
	\psbezier(0.5,0.2)(0.5,0.3)(0.64363,0.6)(0.64363,0.22836)
\end{centerpict}}

\def\pictPassUpper{\begin{centerpict}(-0.1,0)(1.3,1.4)
	\pictPassTotalBdry{1}%
	\psbezier(0.2,0.2)(0.2,0.6)(0.5,0.6)(0.5,0.2)
	\psbezier(0.2,1.2)(0.2,0.8)(0.64363,0.6)(0.64363,0.22836)
	\psbezier(0.5,1.2)(0.5,1.1)(0.64363,0.6)(0.64363,1.22836)
\end{centerpict}}

\def\pictPassId{\begin{centerpict}(-0.1,0)(1.3,1.4)
	\pictPassTotalBdry{1}%
	\psline(0.2,0.2)(0.2,1.2)
	\psline(0.5,0.2)(0.5,1.2)
	\psline(0.64363,0.22836)(0.64363,1.22836)
\end{centerpict}}

\def\pictPassTorus{\begin{centerpict}(-0.1,0)(1.3,1.4)
	\pictPassTotalBdry{1}%
	\pscustom{%
		\moveto(0.2,1.2)
		\curveto(0.2,1.0)(0.07,1.0)(0.07,0.8)
		\curveto(0.07,0.4)(0.55,0.4)(0.55,0.8)
		\curveto(0.55,1.0)(0.50,1.0)(0.50,1.2)
		\stroke}%
	\pscustom{%
		\moveto(0.23,0.8)
		\curveto(0.23,0.6)(0.39,0.6)(0.39,0.8)
		\curveto(0.39,1.0)(0.23,1.0)(0.23,0.8)
		\stroke}%
	\psellipticarc(0.15,0.8)(0.08,0.05){180}{360}
	\psellipticarc(0.47,0.8)(0.08,0.05){180}{360}
	\psdash\psellipticarc(0.15,0.8)(0.08,0.05){0}{180}
	\psdash\psellipticarc(0.47,0.8)(0.08,0.05){0}{180}
	\psbezier(0.2,0.2)(0.2,0.5)(0.5,0.5)(0.5,0.2)
	\psline(0.64363,0.22836)(0.64363,1.22836)
\end{centerpict}}

\def\pictPassTubeDiag{%
	\pictPassTotalBdry{1.8}%
	\psellipticarc(0.10,0.8)(0.08,0.05){180}{360}
	\psellipticarc(0.42,0.8)(0.08,0.05){180}{360}
	\psellipticarc(0.10,1.5)(0.08,0.05){180}{360}
	\psellipticarc(0.42,1.5)(0.08,0.05){180}{360}
	\psdash\psellipticarc(0.10,0.8)(0.08,0.05){0}{180}
	\psdash\psellipticarc(0.42,0.8)(0.08,0.05){0}{180}
	\psdash\psellipticarc(0.10,1.5)(0.08,0.05){0}{180}
	\psdash\psellipticarc(0.42,1.5)(0.08,0.05){0}{180}
	\psbezier(0.02,0.8)(0.02,0.4)(0.2,0.4)(0.2,0.2)
	\psbezier(0.18,0.8)(0.18,0.5)(0.5,0.5)(0.5,0.2)
	\psline(0.5,1.5)(0.5,2)
	\psbezier(0.02,1.5)(0.02,1.8)(0.2 ,1.8)(0.2 ,2  )
	\psbezier(0.18,1.5)(0.18,1.7)(0.34,1.7)(0.34,1.5)
	\psbezier{c-c}(0.64363,0.22836)(0.64363,0.5)(0.34,0.5)(0.34,0.8)
	\psbezier{c-c}(0.5,0.8)(0.5,0.7)(0.64363,0.4)(0.64363,0.8)
	\psline{c-c}(0.64363,0.8)(0.64363,2.02836)
}

\def\pictPassLowerTu{\begin{centerpict}(-0.1,0)(1.3,2.2)
	\pictPassTubeDiag
	\psbezier(0.02,1.5)(0.02,1.25)(0.18,1.25)(0.18,1.5)
	\psbezier(0.34,1.5)(0.34,1.15)(0.50,1.15)(0.50,1.5)
	\psbezier(0.02,0.8)(0.02,1.25)(0.50,1.25)(0.50,0.8)
	\psbezier(0.18,0.8)(0.18,1.0 )(0.34,1.0 )(0.34,0.8)
\end{centerpict}}

\def\pictPassUpperTu{\begin{centerpict}(-0.1,0)(1.3,2.2)
	\pictPassTubeDiag
	\psbezier(0.02,1.5)(0.02,1.2)(0.18,1.2)(0.18,1.5)
	\psbezier(0.02,0.8)(0.02,1.1)(0.18,1.1)(0.18,0.8)
	\psline(0.34,0.8)(0.34,1.5)
	\psline(0.50,0.8)(0.50,1.5)
\end{centerpict}}

\def\pictPassIdTu{\begin{centerpict}(-0.1,0)(1.3,2.2)
	\pictPassTubeDiag
	\psbezier(0.34,1.5)(0.34,1.2)(0.5,1.2)(0.5,1.5)
	\psbezier(0.34,0.8)(0.34,1.1)(0.5,1.1)(0.5,0.8)
	\psline(0.02,0.8)(0.02,1.5)
	\psline(0.18,0.8)(0.18,1.5)
\end{centerpict}}

\def\pictPassTorusTu{\begin{centerpict}(-0.1,0)(1.3,2.2)
	\pictPassTubeDiag
	\psbezier(0.02,1.5)(0.02,1.05)(0.50,1.05)(0.50,1.5)
	\psbezier(0.18,1.5)(0.18,1.3 )(0.34,1.3 )(0.34,1.5)
	\psbezier(0.02,0.8)(0.02,1.05)(0.18,1.05)(0.18,0.8)
	\psbezier(0.34,0.8)(0.34,1.15)(0.50,1.15)(0.50,0.8)
\end{centerpict}}

%% file: images/pics-module.tex
\def\pictActionBottomDiff{%
	\psclip{\psframe[linestyle=none](2.4,0)(1,0.5)}%
		\psdash\psbezier(2.4,-0.08)(1.7,-0.08)(2,0.08)(2.7,0.08)
	\endpsclip
	\psclip{\pspolygon[linestyle=none](2.4,0)(1,0)(1,-1)(3,-1)(3,1)(2.4,1)}%
		\psbezier(2.4,-0.08)(1.7,-0.08)(2,0.08)(2.7,0.08)
	\endpsclip
	\psclip{\pspolygon[linestyle=none]%
			(7.4,-0.024)(6.5,-0.024)(6.4,0.024)(5,0.024)(5,0.2)(7.4,0.2)}%
		\psdash\psbezier( 5.3,-0.08)( 6,-0.08)(6.3,0.08)( 5.6,0.08)
		\psdash\psbezier(7.4,-0.08)(6.7,-0.08)(7  ,0.08)(7.7, 0.08)
	\endpsclip
	\psclip{\pspolygon[linestyle=none]%
			(7.4,-0.024)(6.5,-0.024)(6.4,0.024)(5,0.024)(5,-1)(8,-1)(8,0.5)(7.4,0.5)}%
		\psbezier( 5.3,-0.08)( 6,-0.08)(6.3,0.08)( 5.6,0.08)
		\psbezier(7.4,-0.08)(6.7,-0.08)(7  ,0.08)(7.7, 0.08)
	\endpsclip
}

\def\pictActionTopDiff{%
	\psbezier(2.4,-0.08)(1.7,-0.08)(2, 0.08)(2.7, 0.08)
	\psbezier(5.3,-0.08)(6.1,-0.02)(6.6,-0.02)(7.4,-0.08)
	\psbezier(5.6, 0.08)(6.4, 0.02)(6.9, 0.02)(7.7, 0.08)
}

\def\pictActionTopDiffEmb{%
	\psclip{\psframe[linestyle=none](2.4,0)(1,0.5)}%
		\psdash\psbezier(2.4,-0.08)(1.7,-0.08)(2,0.08)(2.7,0.08)
	\endpsclip
	\psclip{\pspolygon[linestyle=none](2.4,0)(1,0)(1,-1)(3,-1)(3,1)(2.4,1)}%
		\psbezier(2.4,-0.08)(1.7,-0.08)(2,0.08)(2.7,0.08)
	\endpsclip
	\psbezier(5.3,-0.08)(6.1,-0.02)(6.6,-0.02)(7.4,-0.08)
	\psclip{\psframe[linestyle=none](5,-0.5)(7.4,0.5)}%
		\psdash\psbezier(5.6, 0.08)(6.4, 0.02)(6.9, 0.02)(7.7, 0.08)
	\endpsclip
	\psclip{\psframe[linestyle=none](7.4,-0.5)(8,0.5)}%
		\psbezier(5.6, 0.08)(6.4, 0.02)(6.9, 0.02)(7.7, 0.08)
	\endpsclip
}

\def\pictActionDegShiftLower#1{\begingroup
	\psset{linecolor=shift,linewidth=1.5pt}%
	\pictActionBottomDiff
	\psline[linestyle=dashed](1.3, 0   )(2  ,0    )
	\psset{linewidth=1pt}%
	\psline[linestyle=dashed](2.4,-0.08)(5.3,-0.08)
	\psline[linestyle=dashed](2.7, 0.08)(5.3, 0.08)
	\psline[linestyle=dashed](7.4,-0.08)(8.4,-0.08)
	\psline[linestyle=dashed](7.7, 0.08)(8.7, 0.08)
	\rput(4,0){\psframebox[boxsep=false,framesep=2pt,fillcolor=white,fillstyle=solid,linewidth=0.5pt]%
					{$\scriptstyle#1$}}%
\endgroup}

\def\pictActionDegShiftUpper#1{\begingroup
	\psset{linecolor=shift,linewidth=1.5pt}%
	\pictActionTopDiffEmb
	\psline[linestyle=dashed](1.3, 0   )(2  ,0    )
	\psset{linewidth=1pt}%
	\psline[linestyle=dashed](2.4,-0.08)(5.3,-0.08)
	\psline[linestyle=dashed](2.7, 0.08)(5.3, 0.08)
	\psline[linestyle=dashed](7.4,-0.08)(8.4,-0.08)
	\psline[linestyle=dashed](7.7, 0.08)(8.7, 0.08)
	\rput(4,0){\psframebox[boxsep=false,framesep=2pt,fillcolor=white,fillstyle=solid,linewidth=0.5pt]%
					{$\scriptstyle#1$}}%
\endgroup}

\def\pictActionMerge#1#2{%
	\pscustom{%
		\moveto(-1,0)\lineto(-1,#1)
		\rcurveto(0,1.1)(3,0.9)(3,1.5)
		\lineto(2,#2)\stroke}%
	\pscustom{%
		\moveto(0,0)\lineto(0,#1)
		\rcurveto(0,0.4)(1,0.5)(1.5,0.6)
		\rcurveto(0.75,0.15)(0.5,0.4)(0.5,-1.1)
		\lineto(2,0)\stroke}%
	\pscustom{\moveto(2.4,-0.08)\rlineto(0,#2)\stroke}%
	\pscustom{\moveto(2.7, 0.08)\rlineto(0,#2)\stroke}%
	\rput(0,#1){\psarrow(1.9,0.4)(1.2,0.45)}%
}

\def\pictActionSaddle#1#2{%
	\rput(0,-0.08){%
		\psclip{\psframe[linestyle=none](5,#2)(6,-1)}%
			\psdash\pscustom{\moveto(5.6,0.16)\rlineto(0,#2)\stroke}%
		\endpsclip
		\psclip{\psframe[linestyle=none](5,#2)(6,5)}%
			\pscustom{\moveto(5.6,0.16)\rlineto(0,#2)\stroke}%
		\endpsclip
	}%
	\pscustom{\moveto(5.3,-0.08)\rlineto(0,#2)}%
	\pscustom{\moveto(7.4,-0.08)\rlineto(0,#2)}%
	\pscustom{\moveto(7.7, 0.08)\rlineto(0,#2)}%
	\pscustom{%
		\moveto(6,0.024)\lineto(6,#1)
		\rcurveto(0,0.8)(1,0.8)(1,-0.048)
		\lineto(7,-0.024)\stroke}%
}

\def\pictActionFirst{\begin{centerpict}(-1.5,-0.15)(9.2,4.15)
	\rput(-1,0)\COBembcircle
	\rput(0,0  )\pictActionBottomDiff
	\rput(0,0.5){\pictActionDegShiftLower{-a}}%
	\rput(0,1.5){\pictActionDegShiftUpper{a'}}%
	\rput(0,2.0){\pictActionDegShiftUpper{-a'}}%
	\rput(0,3.5){\pictActionDegShiftUpper{a'}}%
	\rput(0,4.0)\pictActionTopDiff
	\pictActionMerge{2}{4}%
	\rput{-20}(-1,2){\COBposShift{x=0.4,y=0.15}(0,0)}%
	\pictActionSaddle{0.5}{4}%
\end{centerpict}}

\def\pictActionSecond{\begin{centerpict}(-1.5,-0.15)(9.2,4.15)
	\rput(-1,0)\COBembcircle
	\rput(0,0  )\pictActionBottomDiff
	\rput(0,0.5){\pictActionDegShiftLower{-a}}%
	\rput(0,3.5){\pictActionDegShiftUpper{a'}}%
	\rput(0,4.0)\pictActionTopDiff
	\pictActionMerge{2}{4}%
	\COBposShift{y=0}(-1,2)%
	\pictActionSaddle{0.75}{4}%
\end{centerpict}}

\def\pictActionThird{\begin{centerpict}(-1.5,-0.15)(9.2,4.15)
	\rput(-1,0)\COBembcircle
	\rput(0,0  )\pictActionBottomDiff
	\rput(0,0.5){\pictActionDegShiftLower{-a}}%
	\rput(0,3.5){\pictActionDegShiftUpper{a'}}%
	\rput(0,4.0)\pictActionTopDiff
	\pictActionMerge{1}{4}%
	\COBposShift{y=0}(-1,1)%
	\pictActionSaddle{2.5}{4}%
\end{centerpict}}

\def\pictActionForth{\begin{centerpict}(-1.5,-0.15)(9.2,4.15)
	\rput(-1,0)\COBembcircle
	\rput(0,0  )\pictActionBottomDiff
	\rput(0,0.5){\pictActionDegShiftLower{-a}}%
	\rput(0,2.0){\pictActionDegShiftLower{a}}%
	\rput(0,2.5){\pictActionDegShiftLower{-a}}%
	\rput(0,3.5){\pictActionDegShiftUpper{a'}}%
	\rput(0,4.0)\pictActionTopDiff
	\pictActionMerge{0.5}{4}%
	\rput{-20}(-1,0.5){\COBposShift{x=0.4,y=0.15}(0,0)}%
	\pictActionSaddle{2.7}{4}%
\end{centerpict}}

%% file: images/trefoils.tex
\def\trefoil#1#2#3{%
	\begin{centerpict}(-1,-0.8)(1,1.1)%
		\psbezier(0,1)( 0.5,1)( 0.59641,0.23302)( 0.34641,-0.2)			%	top right
		\psbezier(-0.86603,-0.5)(-1.11503,-0.06698)(-0.5,0.4)(0,0.4)		% middle left
		\psbezier( 0.86603,-0.5)( 0.61603,-0.93302)(-0.09641,-0.63302)(-0.34641,-0.2)		% bottom right
		\psbezier[border=4pt](0,1)(-0.5,1)(-0.59641,0.23302)(-0.34641,-0.2)			% top left
		\psbezier[border=4pt]( 0.86603,-0.5)( 1.11503,-0.06698)( 0.5,0.4)(0,0.4)		% middle right
		\psbezier[border=4pt](-0.86603,-0.5)(-0.61603,-0.93302)( 0.09641,-0.63302)( 0.34641,-0.2)		% bottom left
		\psline[linewidth=0.5pt,arrowsize=5pt,arrowinset=0.4]{->}(0.625,0.01519)(0.225,0.70801)
		\psline[linewidth=0.5pt,arrowsize=5pt,arrowinset=0.4]{->}(0.3,-0.55)(-0.5,-0.55)
		\psline[linewidth=0.5pt,arrowsize=5pt,arrowinset=0.4]{->}(-0.625,0.01519)(-0.225,0.70801)
		\uput[ 30]( 0.475, 0.275){$1$}
		\uput[-90]( 0.000,-0.550){$2$}
		\uput[150](-0.475, 0.275){$3$}
	\end{centerpict}%
}

\def\trefoilshadow{%
	\psbezier(0,1)( 0.5,1)( 0.59641,0.23302)( 0.34641,-0.2)%													top right
	\psbezier(-0.86603,-0.5)(-1.11503,-0.06698)(-0.5,0.4)(0,0.4)% 										middle left
	\psbezier( 0.86603,-0.5)( 0.61603,-0.93302)(-0.09641,-0.63302)(-0.34641,-0.2)%		bottom right
	\psbezier(0,1)(-0.5,1)(-0.59641,0.23302)(-0.34641,-0.2)% 													top left
	\psbezier( 0.86603,-0.5)( 1.11503,-0.06698)( 0.5,0.4)(0,0.4)% 										middle right
	\psbezier(-0.86603,-0.5)(-0.61603,-0.93302)( 0.09641,-0.63302)( 0.34641,-0.2)%		bottom left
}

\def\trefoilresolution#1#2#3{%
	\trefoilshadow
	\begingroup
		\psset{linestyle=none,fillstyle=solid,fillcolor=white}%
		\pscircle( 0    ,-0.55){0.2}
		\pscircle( 0.475,0.275){0.2}
		\pscircle(-0.475,0.275){0.2}
	\endgroup
	\ifx1#1\relax
		\psbezier(0.47105,0.46555)(0.48679,0.33682)(0.53343,0.254  )(0.63713,0.17613)
		\psbezier(0.45367,0.0879 )(0.48132,0.21841)(0.42964,0.30766)(0.30286,0.34897)
	\else
		\psbezier(0.47105,0.46555)(0.48679,0.33682)(0.42964,0.30766)(0.30286,0.34897)
		\psbezier(0.45367,0.0879 )(0.48132,0.21841)(0.53343,0.254  )(0.63713,0.17613)
		\ifx0#1\relax\else
			\psline[linewidth=0.5pt,arrowsize=3pt,arrowinset=0.4]{->}(0.625,0.01519)(0.225,0.70801)
		\fi
	\fi
	\ifx1#2\relax
		\psbezier(-0.16765,-0.64072)(-0.04831,-0.58998)( 0.04688,-0.58917)( 0.1662 ,-0.6401 )
		\psbezier( 0.15071,-0.43686)( 0.05151,-0.52606)(-0.05151,-0.52606)(-0.15071,-0.43686)
	\else
		\psbezier(-0.16765,-0.64072)(-0.04831,-0.58998)(-0.05151,-0.52606)(-0.15071,-0.43686)
		\psbezier( 0.15071,-0.43686)( 0.05151,-0.52606)( 0.04688,-0.58917)( 0.1662 ,-0.6401 )
		\ifx0#2\relax\else
			\psline[linewidth=0.5pt,arrowsize=3pt,arrowinset=0.4]{->}(0.4,-0.55)(-0.5,-0.55)
		\fi
	\fi
	\ifx1#3\relax
		\psbezier(-0.47105,0.46555)(-0.48679,0.33682)(-0.53343,0.254  )(-0.63713,0.17613)
		\psbezier(-0.45367,0.0879 )(-0.48132,0.21841)(-0.42964,0.30766)(-0.30286,0.34897)
	\else
		\psbezier(-0.47105,0.46555)(-0.48679,0.33682)(-0.42964,0.30766)(-0.30286,0.34897)
		\psbezier(-0.45367,0.0879 )(-0.48132,0.21841)(-0.53343,0.254  )(-0.63713,0.17613)
		\ifx0#3\relax\else
			\psline[linewidth=0.5pt,arrowsize=3pt,arrowinset=0.4]{->}(-0.625,0.01519)(-0.225,0.70801)
		\fi
	\fi
}

%% file: images/diag-changes.tex
\newgray{blackRegion}{0.625}
\def\diagArrowLeft{<-}
\def\diagArrowRight{->}

\def\diagArrowGen#1#2/#3,#4/(#5,#6)(#7,#8)(#9){%
	\edef\testArrow{#2}%
	\dimen20=#5\psxunit	\dimen21=#6\psyunit
	\dimen22=#7\psxunit	\dimen23=#8\psyunit
	\dimen24=#3\psxunit \dimen25=#4\psyunit
	\ifx\testArrow\diagArrowLeft
		\advance\dimen20 by \dimen24
		\advance\dimen21 by \dimen25
	\else\ifx\testArrow\diagArrowRight
		\advance\dimen22 by -\dimen24
		\advance\dimen23 by -\dimen25
	\fi\fi
	\psline{#2}(\the\dimen20,\the\dimen21)(\the\dimen22,\the\dimen23)%
	\rput[c](#9){$\scriptstyle{#1}$}
\endgroup\ignorespaces}

\def\diagArrow{\begingroup\psset{linewidth=2pt}\diagArrowGen}
\def\diagDblArrow{\begingroup\psset{linewidth=0.5pt,doubleline=true,doublesep=2pt}\diagArrowGen}

\newcommand*\diagConnT[4]{%
	\psset{linewidth=0.5pt,arrowlength=0.8,arrowsize=6pt,dimen=outer}
	\edef\testArrow{#4}%
	\ifx\testArrow\diagArrowLeft
		\psbezier[linewidth=2pt]{<-}(0.4,0)(0.8,0)(0.8,0.6)(0,0.6)
	\else
		\psbezier[linewidth=2pt,arrowlength=0.8,arrowsize=6pt]{-}(0.4,0)(0.8,0)(0.8,0.6)(0,0.6)
	\fi
	\ifx\testArrow\diagArrowRight
		\psbezier[linewidth=2pt]{<-}(-0.4,0)(-0.8,0)(-0.8,0.6)(0,0.6)
	\else
		\psbezier[linewidth=2pt,arrowlength=0.8,arrowsize=6pt]{-}(-0.4,0)(-0.8,0)(-0.8,0.6)(0,0.6)
	\fi
	\rput[c](0,0.8){$\scriptstyle{#3}$}
	\psclip{\pscircle(0,0){0.4}}%
		\diagArrow{#1}{#2}/0,0.02/(0,-0.4)(0,0.4)(0.15,0)
	\endpsclip
}

\newcommand*\diagConnTwb[4]{%
	\psset{linewidth=0.5pt,arrowlength=0.8,arrowsize=6pt,dimen=outer}
	\edef\testArrow{#4}%
	\ifx\testArrow\diagArrowLeft
		\psbezier[linewidth=2pt]{<-}(0.4,0)(0.8,0)(0.8,0.6)(0,0.6)
	\else
		\psbezier[linewidth=2pt,arrowlength=0.8,arrowsize=6pt]{-}(0.4,0)(0.8,0)(0.8,0.6)(0,0.6)
	\fi
	\ifx\testArrow\diagArrowRight
		\psbezier[linewidth=2pt]{<-}(-0.4,0)(-0.8,0)(-0.8,0.6)(0,0.6)
	\else
		\psbezier[linewidth=2pt,arrowlength=0.8,arrowsize=6pt]{-}(-0.4,0)(-0.8,0)(-0.8,0.6)(0,0.6)
	\fi
	\rput[c](0,0.8){$\scriptstyle{#3}$}
	\diagDblArrow{#1}{#2}/0,0.05/(0,-0.4)(0,0.4)(0.15,0)
	\pscircle(0,0){0.4}
}

\newcommand*\diagConnTbw[4]{%
	\psframe[linestyle=none,fillstyle=solid,fillcolor=blackRegion](-1,-0.5)(1,0.9)
	\psset{arrowlength=0.8,arrowsize=6pt,dimen=outer,linewidth=0.5pt,doubleline=true,doublesep=2pt}%
	\edef\testArrow{#4}%
	\ifx\testArrow\diagArrowLeft
		\psbezier{<-}(0.4,0)(0.8,0)(0.8,0.6)(0,0.6)
	\else
		\psbezier{-}(0.4,0)(0.8,0)(0.8,0.6)(0,0.6)
	\fi
	\ifx\testArrow\diagArrowRight
		\psbezier{<-}(-0.4,0)(-0.8,0)(-0.8,0.6)(0,0.6)
	\else
		\psbezier{-}(-0.4,0)(-0.8,0)(-0.8,0.6)(0,0.6)
	\fi
	\rput[c](0,0.8){$\scriptstyle{#3}$}
	\psset{doubleline=false}%
	\begin{psclip}{\pscircle(0,0){0.4}}%
		\diagArrow{#1}{#2}/0,0.05/(0,-0.4)(0,0.4)(0.15,0)
	\end{psclip}
}

\makeatletter
\def\diagToPictOne#1(#2,#3)(#4,#5){%
	\expandafter\edef\csname pictLowered#1\endcsname##1{\noexpand\begingroup
		\dimen20=#3\psyunit
		\advance\dimen20 by 0.2\baselineskip
		\noexpand\pspicture[shift=\noexpand\the\dimen20](#2,#3)(#4,#5)
			\expandafter\noexpand\csname diag#1\endcsname{##1}%
		\noexpand\endpspicture
	\noexpand\endgroup}%
	\expandafter\edef\csname pictCentered#1\endcsname##1{%
		\noexpand\centerpict(#2,#3)(#4,#5)
			\expandafter\noexpand\csname diag#1\endcsname{##1}%
		\noexpand\endcenterpict
	}%
	\expandafter\edef\csname pict#1\endcsname{%
		\noexpand\@ifstar
			{\expandafter\noexpand\csname pictCentered#1\endcsname}%
			{\expandafter\noexpand\csname pictLowered#1\endcsname}%
	}
}
\def\diagToPictTwo#1(#2,#3)(#4,#5){%
	\expandafter\edef\csname pictLowered#1\endcsname##1##2{\noexpand\begingroup
		\dimen20=#3\psyunit
		\advance\dimen20 by 0.2\baselineskip
		\noexpand\pspicture[shift=\noexpand\the\dimen20](#2,#3)(#4,#5)
			\expandafter\noexpand\csname diag#1\endcsname{##1}{##2}%
		\noexpand\endpspicture
	\noexpand\endgroup}%
	\expandafter\edef\csname pictCentered#1\endcsname##1##2{%
		\noexpand\centerpict(#2,#3)(#4,#5)
			\expandafter\noexpand\csname diag#1\endcsname{##1}{##2}%
		\noexpand\endcenterpict
	}%
	\expandafter\edef\csname pict#1\endcsname{%
		\noexpand\@ifstar
			{\expandafter\noexpand\csname pictCentered#1\endcsname}%
			{\expandafter\noexpand\csname pictLowered#1\endcsname}%
	}
}
\def\diagToPictThree#1(#2,#3)(#4,#5){%
	\expandafter\edef\csname pictLowered#1\endcsname##1##2##3{\noexpand\begingroup
		\dimen20=#3\psyunit
		\advance\dimen20 by 0.2\baselineskip
		\noexpand\pspicture[shift=\noexpand\the\dimen20](#2,#3)(#4,#5)
			\expandafter\noexpand\csname diag#1\endcsname{##1}{##2}{##3}%
		\noexpand\endpspicture
	\noexpand\endgroup}%
	\expandafter\edef\csname pictCentered#1\endcsname##1##2##3{%
		\noexpand\centerpict(#2,#3)(#4,#5)
			\expandafter\noexpand\csname diag#1\endcsname{##1}{##2}{##3}%
		\noexpand\endcenterpict
	}%
	\expandafter\edef\csname pict#1\endcsname{%
		\noexpand\@ifstar
			{\expandafter\noexpand\csname pictCentered#1\endcsname}%
			{\expandafter\noexpand\csname pictLowered#1\endcsname}%
	}
}
\def\diagToPictFour#1(#2,#3)(#4,#5){%
	\expandafter\edef\csname pictLowered#1\endcsname##1##2##3##4{\noexpand\begingroup
		\dimen20=#3\psyunit
		\advance\dimen20 by 0.2\baselineskip
		\noexpand\pspicture[shift=\noexpand\the\dimen20](#2,#3)(#4,#5)
			\expandafter\noexpand\csname diag#1\endcsname{##1}{##2}{##3}{##4}%
		\noexpand\endpspicture
	\noexpand\endgroup}%
	\expandafter\edef\csname pictCentered#1\endcsname##1##2##3##4{%
		\noexpand\centerpict(#2,#3)(#4,#5)
			\expandafter\noexpand\csname diag#1\endcsname{##1}{##2}{##3}{##4}%
		\noexpand\endcenterpict
	}%
	\expandafter\edef\csname pict#1\endcsname{%
		\noexpand\@ifstar
			{\expandafter\noexpand\csname pictCentered#1\endcsname}%
			{\expandafter\noexpand\csname pictLowered#1\endcsname}%
	}
}
\makeatother

\diagToPictTwo{M}(-1.1,-0.7)(1.1,0.7)
\diagToPictTwo{S}(-0.7,-0.7)(0.7,0.7)
\diagToPictOne{B}(-0.7,-0.7)(0.5,0.5)
\diagToPictOne{D}(-0.7,-0.7)(0.5,0.5)

\diagToPictOne{CreateBirth}(-0.7,-0.7)(1.3,0.7)
\diagToPictOne{CreateDeath}(-0.7,-0.7)(1.3,0.7)

\diagToPictTwo{DisBB}(-0.9,-0.7)(0.9,0.5)
\diagToPictThree{DisBM}(-0.9,-0.7)(2.3,0.7)
\diagToPictThree{DisBS}(-0.9,-0.7)(1.5,0.7)
\diagToPictTwo{DisBD}(-0.9,-0.7)(0.9,0.5)

\diagToPictTwo{DisDB}(-0.9,-0.7)(0.9,0.5)
\diagToPictThree{DisDM}(-0.9,-0.7)(2.3,0.7)
\diagToPictThree{DisDS}(-0.9,-0.7)(1.5,0.7)
\diagToPictTwo{DisDD}(-0.9,-0.7)(0.9,0.5)

\diagToPictThree{DisMB}(-2.3,-0.7)(0.9,0.7)
\diagToPictFour{DisMM}(-2.3,-0.7)(2.3,0.7)
\diagToPictFour{DisMS}(-2.3,-0.7)(1.5,0.7)
\diagToPictThree{DisMD}(-2.3,-0.7)(0.9,0.7)

\diagToPictThree{DisSB}(-1.5,-0.7)(0.9,0.7)
\diagToPictFour{DisSM}(-1.5,-0.7)(2.3,0.7)
\diagToPictFour{DisSS}(-1.5,-0.7)(1.5,0.7)
\diagToPictThree{DisSD}(-1.5,-0.7)(0.9,0.7)

\diagToPictFour{ConnMM}(-1.7,-0.7)(1.7,0.7)
\diagToPictFour{ConnMS}(-1.1,-0.7)(1.5,0.7)
\diagToPictFour{ConnSM}(-1.5,-0.7)(1.1,0.7)
\diagToPictFour{ConnSS}(-0.9,-0.7)(0.9,0.7)
\diagToPictFour{ConnX}(-1.2,-0.7)(1.2,0.7)

\diagToPictFour{ConnT}(-0.8,-0.7)(0.8,1.0)
\diagToPictFour{ConnTpos}(-0.8,-0.7)(0.8,1.0)
\diagToPictFour{ConnTneg}(-0.8,-0.7)(0.8,1.0)
\diagToPictFour{ConnTwb}(-0.8,-0.7)(0.8,1.0)
\diagToPictFour{ConnTwbpos}(-0.8,-0.7)(0.8,1.0)
\diagToPictFour{ConnTwbneg}(-0.8,-0.7)(0.8,1.0)
\diagToPictFour{ConnTbw}(-1,-0.7)(1,1.0)
\diagToPictFour{ConnTbwpos}(-1,-0.7)(1,1.0)
\diagToPictFour{ConnTbwneg}(-1,-0.7)(1,1.0)

%% file: intro.tex
The~category $\cat{2Cob}$ of $2$-dimensional cobordisms is usually considered as $\Z$-graded, with the~degree function given by the~Euler characteristic of a~cobordism. It was shown in \cites{KhUnified,ChCob} that this degree can be split into two numbers, one counting merges and births, whereas the~other splits and deaths:
\begin{align*}
	\deg\left(\textcobordism[2](M)\right) &= (-1,0) & \deg\left(\textcobordism[0](sB)\right) &= (1,0) \\
	\deg\left(\textcobordism[1](S)\right) &= (0,-1) & \deg\left(\textcobordism[1](sD)\right) &= (0,1)
\end{align*}
Indeed, the~only relations that affect the~set of critical points either create or remove a~pair birth--merge or split--death, which does not change the~two numbers. Because of that one can try to enhance the~construction of Khovanov homology $\Kh^{i,j}(L)$ \cite{KhHom} to a~triply graded homology $\widetilde\Kh(L)^{i,p,q}$ with
\begin{equation*}
	\Kh^{i,j}(L) = \bigoplus_{p+q=j}\widetilde\Kh^{i,p,q}(L).
\end{equation*}
Indeed, the~chain complex computing $\Kh(L)$ is constructed from the~cube of resolutions of a~diagram $D$ of $L$, vertices and edges of which are labeled with collections of circles and cobordisms between them respectively. Shifting degrees of the~vertices appropriately results in a~cube of graded maps, from which one can obtain a~chain complex of bigraded groups.

Unfortunately, one does not obtain a~stronger invariant in this way, as after a~normalization $\widetilde\Kh(L)^{i,p,q}=0$ unless $p=q$. This is the~reason why the~author dropped this idea in his earlier works \cites{KhUnified,ChCob} on unification of the~Khovanov homology with its odd variant \cite{OddKh}.

The~additional grading, however, appeared to be a~key ingredient to understand the~odd and generalized Khovanov homologies of composite links. It is well known that the~original Khovanov homology, which we call in this paper the~\emph{even Khovanov homology}, is multiplicative with both disjoint unions and connected sums of links \cites{KhHom,KhPatterns}, i.e.\
\begin{equation*}
	\EKh(L\sqcup L') \cong \EKh(L) \underset{\mathclap{\Z}}{\hat\otimes} \EKh(L')
	\qquad\text{and}\qquad
	\EKh(L\connsum L') \cong \EKh(L) \underset{\mathclap{A}}{\hat\otimes} \EKh(L'),
\end{equation*}
where $\hat\otimes$ denotes a~derived tensor product and $A$ is the~Khovanov's algebra associated to a~circle. To construct such isomorphisms for the~odd and generalized Khovanov homology one has to understand the~role of sign assignments better---the~naive tensor product of complexes for links $L$ and $L'$ does not give a~priori a~chain complex for $L\sqcup L'$. The~new grading is the~additional information that helps to deal with this situation. Namely, it tells us that the~naive isomorphism $\OKh(L\sqcup L') \to \OKh(L)\hat\otimes\OKh(L')$, i.e.\ the~one given by the~identities on chain groups, is not the~correct one (it is even not a~chain map). Instead, it has to be twisted by certain signs, which are controlled by the~new grading.

The~above is enough to derive a~formula for disjoint unions of links, but to compute homology of a~connected sum a~module structure on the~homology must be specified. It is defined naturally at the~level of link diagrams with basepoints: the~action of the~algebra $A$ is induced by merging a~circle with the~diagram at the~basepoint. Although it works nicely in the~even setting, there are several issues in the~case of generalized Khovanov homology. For instance, the~algebra associated to a~circle is not associative. This can happen, as the~product is not a~graded map---once the~degree is shifted accordingly one obtains an~associative algebra $A'$, tensor products over which are well-defined.

Sliding a~basepoint on a~link does not change the~module structure on the~even Khovanov homology up to an~isomorphism, and we proof the~same for the~generalized variant. It is a~bit surprising that the~module structure behaves nicer for the~odd Khovanov homology---sliding a~basepoint though a~crossing does not change the~module structure at all. On the~other hand, one cannot move a~dot from one link component to another. Therefore, we follow the~idea of \cite{KhDetectsUnlinks} and we construct $c$ actions of $A'$ on the~homology of a~$c$-component link. When computing homology of $L\connsum L'$ one should choose the~actions determined by the~link components that are joined together.

Our construction of a~triply graded Khovanov homology is not the~only one. There is another construction expected to result in a~triply graded chain complex coming from a~filtration on the~HOMFLYPT homology \cite{HOMFLYPTfiltr}, but so far it is not proven to be invariant under all Reidemeister moves. We do not know how it compares with our grading.

\subsection*{Outline}
We begin with a~brief description of the~category of chronological cobordisms $\kChCob$ and graded tensor categories, introduced in \cite{ChCob}, which provide a~framework for the~construction of the~generalized Khovanov homology $\Kh(L)$. The~construction of the~generalized Khovanov complex $\KhCom(D)$ is presented in Section~\ref{sec:khov-def} following \cite{ChCob}. The~only difference is in regarding it as an~object graded by $\Z\times\Z$, instead of graded by integers as in the~original construction. The~section ends with a~proof that all the~homotopy equivalences used in \cite{ChCob} to prove invariance of $\KhCom(D)$ under Reidemeister moves preserve the~new grading. Then in Section~\ref{sec:homology} we show that the~new grading does not lead to new invariants.

The~main part of the~paper begins in Section~\ref{sec:composite}, in which we derive the~formulas for $\KhCom(L\sqcup L')$ and $\KhCom(L\connsum L')$, first at the~level of complexes in $\kChCob$. The~application of chronological TQFT functors and the~formulas for homology is delayed till the~end of Section~\ref{sec:module}, after we construct the~module structure on homology. We compute here the~result of sliding a~basepoint through a~crossing---this was done in \cite{KhDetectsUnlinks} over $\Z_2$, but our map is defined over integers.

%% file: tqft.tex
In this paper $\scalars$ will always stand for the~ring $\scalarsLong$.

%% ============================================================================
%%	Chronological cobordisms
%% ============================================================================
\begin{definition}[cf. \cite{ChCob}]
	Let $W$ be a~cobordism with a~Riemann metric. A~\emph{chronology} on $W$ consists of a~Morse function $h\colon W\to I$ that separates critical points, and a~choice of an~orientation of $E^-(p)$, the~space of unstable directions in the~gradient flow induced by $h$, at each critical point $p$. We require $h^{-1}(0)$ and $h^{-1}(1)$ to be the~input and output of $W$ respectively.
\end{definition}

%%
%%	Generators
%%
A~standard argument from Morse theory implies that every 2-dimensional chronological cobordism can be built from six surfaces:
\begin{equation}\label{diag:mor-gens}\psset{unit=1cm}\begin{centerpict}(12.4,2.4)
	\COBmergeFrLeft(0,1.1)\COBsplitFrBack(2.6,1.1)
	\COBbirth(5.2,1.1)\COBposDeath(7.0,1.1)\COBnegDeath(9.2,1.1)
	\COBpermutation(11.2,1.1)
	\rput[B](0.6,1.5ex){\textnormal{a~merge}}
	\rput[B](2.8,1.5ex){\textnormal{a~split}}
	\rput[B](5.0,1.5ex){\textnormal{a~birth}}
	\rput[B](7.2,1.5ex){\parbox{10ex}{\centering\textnormal{a~positive death}}}
	\rput[B](9.4,1.5ex){\parbox{10ex}{\centering\textnormal{a~negative death}}}
	\rput[B](11.8,1.5ex){\textnormal{a~twist}}
\end{centerpict}\end{equation}
The~little arrows visualize orientations of critical points. One merge and one split is sufficient, as the~little arrow can be reversed by composing the~cobordism with a~twist.

%%
%%	Disjoint union and connected sum
%%
Chronological cobordisms admit two disjoint unions: the~`left-then-right' $W\ldsum W'$ and the~`right-then-left' one $W\rdsum W'$. Both are diffeomorphic to the~standard disjoint union $W\sqcup W'$, but to avoid a~situation with two critical points at the~same level one has to pull all critical points of $W$ below $\frac{1}{2}$ and those of $W'$ over $\frac{1}{2}$ (for $\ldsum$) or the~other way (for $\rdsum$):
\begin{displaymath}
	\textcobordism[2](M-L)(sI)\textcobordism[1](sI)(S-B)
		\from<3em>^{\ldsum}
	\left(\textcobordism[2](M-L), \textcobordism[1](S-B)\right)
		\to<3em>^{\rdsum}
	\textcobordism[2](sI)(M-L)\hskip 0.3\psxunit\textcobordism[1](S-B)(sI)
\end{displaymath}
Likewise, there are two versions of a~connected sum $W\lcsum W'$ and $W\rcsum W'$:
\begin{displaymath}
	\textcobordism[2](lM-L)(rS-F)
		\from<3em>^{\lcsum}
	\left(\textcobordism[2](lM-L), \textcobordism[1](rS-F)\right)
		\to<3em>^{\rcsum}
	\textcobordism[2](rS-F,2)(lM-L)
\end{displaymath}
However, this operation requires more choices---a~vertical line at each cobordism, along which they are glued together. We shall make this choice implicitly.

%%
%%	Z x Z degree
%%
\begin{definition}
	Define the~\emph{degree} $\chdeg W\in \Z\times\Z$ of a~chronological cobordism $W$ by setting
	\begin{equation}
		\chdeg W := (\#\text{births}-\#\text{merges}, \#\text{deaths}-\#\text{splits}).
	\end{equation}
\end{definition}

\noindent
The~chronological degree is clearly additive with respect to composition of chronological cobordisms as well as both disjoint unions and connected sums.

\begin{lemma}\label{lem:chdeg-vs-bdry}
	Let $W$ be a~chronological cobordism of degree $\chdeg W = (a,b)$ with $n$ inputs and $m$ outputs. Then $a+n = b+m$.
\end{lemma}
\begin{proof}
	Straightforward, by checking for generating cobordisms \eqref{diag:mor-gens}.
\end{proof}

%% ============================================================================
%%	Linearization and chronological relations
%% ============================================================================
Let $\kChCob$ be a~$\scalars$-linear category with finite disjoint unions of circles as objects, and formal $\scalars$-linear combinations of 2-dimensional chronological cobordisms as morphisms, modulo the~following \emph{chronological relations}:
\begin{gather}
	%
	%	Reversing orientations
		\label{rel:reverse-orientation}
		\textcobordism[2](M-R) = \permMM\textcobordism[2](M-L)
			\hskip 1cm
		\textcobordism[1](S-F) = \permSS\textcobordism[1](S-B)
			\hskip 1cm
		\textcobordism[1](sD-) = \permSS\textcobordism[1](sD+)
	\\
	%
	%	Annihilations and creations
		\label{rel:annihilation-creation}
		\textcobordism[1](sI)(I) = \textcobordism*[1](slB)(M-L)
			\hskip 1.5cm
		\textcobordism[1](I)(sI) = \textcobordism*[1](S-B)(srD+,2)
			\hskip 1.5cm
		\textcobordism[1](I)(sI) = \textcobordism*[1](S-B)(slD-,1)
	\\
	%
	% Connected sums
		\label{rel:connected-sum}
		\begin{centerpict}(-0.1,-0.1)(2.9,2.5)
			\COBcylinder(0.1,0)(0.1,2.4)\COBcylinder(1.2,0)(1.2,2.4)\COBcylinder(2.3,0)(2.3,2.4)
			\rput[c](0.87,0.1){$\scriptstyle\cdots$}\rput[c](0.87,2.25){$\scriptstyle\cdots$}
			\rput[c](1.97,0.1){$\scriptstyle\cdots$}\rput[c](1.97,2.25){$\scriptstyle\cdots$}
			\psframe[framearc=0.5,fillstyle=solid](-0.05,0.3)(1.40,1.2)
			\psframe[framearc=0.5,fillstyle=solid]( 1.40,1.2)(2.85,2.1)
			\rput[c](0.725,0.75){$W'$}
			\rput[c](2.125,1.65){$W\phantom'$}
		\end{centerpict}
			= \lambda(\chdeg W,\chdeg W')
		\begin{centerpict}(-0.1,-0.1)(2.9,2.5)
			\COBcylinder(0.1,0)(0.1,2.4)\COBcylinder(1.2,0)(1.2,2.4)\COBcylinder(2.3,0)(2.3,2.4)
			\rput[c](0.87,0.1){$\scriptstyle\cdots$}\rput[c](0.87,2.25){$\scriptstyle\cdots$}
			\rput[c](1.97,0.1){$\scriptstyle\cdots$}\rput[c](1.97,2.25){$\scriptstyle\cdots$}
			\psframe[framearc=0.5,fillstyle=solid](-0.05,1.2)(1.40,2.1)
			\psframe[framearc=0.5,fillstyle=solid]( 1.40,0.3)(2.85,1.2)
			\rput[c](0.725,1.65){$W'$}
			\rput[c](2.125,0.75){$W\phantom'$}
		\end{centerpict}
	\\
	%
	% Disjoint unions
		\label{rel:disjoint-union}
		\begin{centerpict}(-0.1,-0.1)(3.5,2.5)
			\COBcylinder(0.1,0)(0.1,2.4)\COBcylinder(1.0,0)(1.0,2.4)
			\COBcylinder(2.0,0)(2.0,2.4)\COBcylinder(2.9,0)(2.9,2.4)
			\rput[c](0.77,0.1){$\scriptstyle\cdots$}\rput[c](0.77,2.25){$\scriptstyle\cdots$}
			\rput[c](2.67,0.1){$\scriptstyle\cdots$}\rput[c](2.67,2.25){$\scriptstyle\cdots$}
			\psframe[framearc=0.5,fillstyle=solid](-0.05,0.3)(1.55,1.2)
			\psframe[framearc=0.5,fillstyle=solid]( 1.85,1.2)(3.45,2.1)
			\rput[c](0.8,0.75){$W'$}
			\rput[c](2.7,1.65){$W\phantom'$}
		\end{centerpict}
			= \lambda(\deg W,\deg W')
		\begin{centerpict}(-0.1,-0.1)(3.5,2.5)
			\COBcylinder(0.1,0)(0.1,2.4)\COBcylinder(1.0,0)(1.0,2.4)
			\COBcylinder(2.0,0)(2.0,2.4)\COBcylinder(2.9,0)(2.9,2.4)
			\rput[c](0.77,0.1){$\scriptstyle\cdots$}\rput[c](0.77,2.25){$\scriptstyle\cdots$}
			\rput[c](2.67,0.1){$\scriptstyle\cdots$}\rput[c](2.67,2.25){$\scriptstyle\cdots$}
			\psframe[framearc=0.5,fillstyle=solid](-0.05,1.2)(1.55,2.1)
			\psframe[framearc=0.5,fillstyle=solid]( 1.85,0.3)(3.45,1.2)
			\rput[c](0.8,1.65){$W'$}
			\rput[c](2.7,0.75){$W\phantom'$}
		\end{centerpict}
\end{gather}
where $W$ and $W'$ stand for any cobordisms, and $\lambda(a,b,a',b') := \permMM^{aa'}\permSS^{bb'}\permMS^{ab'-a'b}$. Notice that the~following associativity and Frobenius-type relations are special cases of \eqref{rel:connected-sum}:
\begin{gather}
	% Associativity
		\label{rel:associativity}
		\textcobordism*[3](M-L)(M-L) = \permMM\textcobordism*[3](M-L,2)(M-L)
			\hskip 1.5cm
		\textcobordism*[1](S-B)(S-B) = \permSS\textcobordism*[1](S-B)(S-B,2)
	\\
	%
	% Frobenius
		\label{rel:frobenius}
		\textcobordism*[2](S-B)(M-L,2) = \permMS\textcobordism*[2](M-L)(S-B)
															 = \textcobordism*[2](S-B,2)(M-L)
\end{gather}
We proved in \cite{ChCob} the~following non-degeneracy result for $\kChCob$.

%%
%%	Nondegeneracy of kChCob
%%
\begin{theorem}\label{thm:kChCob-nondeg}
	Suppose $kW=0$ for a~chronological cobordism $W$ and a~nonzero $k\in\scalars$. Then $W$ has either positive genus or at least two closed components, and $k$ is divisible by $(\permMM\permSS-1)$. In particular, cobordisms cannot be annihilated by monomials.
\end{theorem}

%% ============================================================================
%%	k-modules and chronological TQFTs
%% ============================================================================
Consider now the~category $\Mod\scalars$ of $\scalars$-modules graded by the~group $\Z\times\Z$. We redefine the~tensor product by setting for homogeneous homomorphisms $f$ and $g$
\begin{equation}\label{eq:graded-tensor-morphisms}
		(f\otimes g)(m\otimes n) := \lambda(\deg g, \deg m) f(m)\otimes g(n),
\end{equation}
where $\lambda(a,b,a',b')=\permMM^{aa'}\permSS^{bb'}\permMS^{ab'-a'b}$ is defined as for $\kChCob$. One checks directly that
\begin{equation}\label{eq:graded-tensor-functoriality}
	(f'\otimes g')\circ(f\otimes g) = \lambda(\deg g',\deg f)(f'\circ f)\otimes(g'\circ g).
\end{equation}
Hence, $\Mod\scalars$ is a~\emph{graded tensor category} in the~sense of \cite{ChCob}. There is a~symmetry $\tau_{M,N}\colon M\otimes N\to N\otimes M$ given by the~formula $\tau_{M,N}(m\otimes n) = \lambda(\deg m,\deg n)\,n\otimes m$ for homogeneous elements $m\in M$ and $n\in N$.

%%
%%	Chronological TQFT
%%
A~linear category is said to be \emph{graded} by an~abelian group $G$ if its morphism spaces are $G$-graded modules, and the~grading is preserved by composition of morphisms. We also require an~additive family of \emph{degree shift functors} $A\mapsto A\{g\}$ parametrized by $g\in G$, i.e.\ $A\{g\}\{h\} = A\{g+h\}$, such that the~modules $\Mor(A\{m\},B\{n\})$ and $\Mor(A,B)$ are naturally isomorphic up to grading: if a~morphism $f\in\Mor(A,B)$ has degree $d$, then $\deg f = d + n - m$ when regarded as an~element of $\Mor(A\{m\},B\{n\})$.

The~category $\Mod\scalars$ is clearly graded by $\Z\times\Z$, but $\kChCob$ is not---it lacks the~degree shift functors. We introduce them formally by replacing the~objects of $\kChCob$ with symbols $\Sigma\{a,b\}$, where $\Sigma$ is a~finite disjoint union of circles and $a,b\in\Z$. The~degree of a~chronological cobordism is extended over the~new morphisms in a~natural way:
\begin{equation}\label{eq:deg-in-ChCob0}
	\chdeg\Big(\Sigma\{a,b\}\to^{\ W\ }\Sigma'\{a',b'\}\Big) :=
	\chdeg\Big(\Sigma\to^{\ W\ }\Sigma'\Big) + (a'-a,b'-b).
\end{equation}
For example, the~following morphism has degree $(0,0)$:
\begin{equation}
	\rule[0pt]{0pt}{10mm}
	\fntCircle\fntCircle \to^{\quad\psset{unit=6mm}\textcobordism[2](M)\quad} \fntCircle\{1,0\}.
\end{equation}
We reserve the~symbol $\kChCob_0$ for the~subcategory of $\kChCob$ spanned by morphisms of degree $(0,0)$. It is an~abelian category.
%The~last condition is equivalent to a~choice of canonical isomorphisms $i_m\colon A\longrightarrow A\{m\}$ of degree $\deg i_m = m$, such that $i_0 = \id$ and $i_m\circ i_n = i_{m+n}$. Indeed, think of $i_m\in\Mor(A,A\{m\})$ as the~morphism corresponding to the~identity $\id_A\in\Mor(A,A)$.

\begin{definition}
	A~\emph{chronological TQFT} is a~graded functor $\F\colon\kChCob \to \Mod\scalars$ that maps $\rdsum$ into the~graded tensor product $\otimes$ and the~twist {\psset{unit=4mm}\textcobordism[2](P)} into the~symmetry $\tau$.
\end{definition}

\begin{example}\label{ex:Kh-tqft}
	We defined in \cite{ChCob} a~chronological TQFT $\FA\colon\scalars\ChCob \to \Mod\scalars$, which maps a~circle to a~$\scalars$-module $A$ freely degenerated by $v_+$ in degree $(1,0)$ and $v_-$ in degree $(0,-1)$. On generating cobordisms $\FA$ is defined as follows:
	\begin{align}
	\label{eq:F-merge}
		\FA\left(\textcobordism[2](M-L)\right)&\colon A\otimes A\to A,\phantom{R}\quad
			\begin{cases}
					v_+\otimes v_+ \mapsto v_+, &\qquad v_+\otimes v_- \mapsto v_-,\\
					v_-\otimes v_- \mapsto 0,   &\qquad v_-\otimes v_+ \mapsto \permMM\permMS v_-,
				\end{cases} \\
	\label{eq:F-split}
		\FA\left(\textcobordism[1](S-B)\right)&\colon A\to A\otimes A,\phantom{R}\quad
			\begin{cases}
				v_+\mapsto v_-\otimes v_+ + \permSS\permMS v_+\otimes v_-,\\
				v_-\mapsto v_-\otimes v_-,
			\end{cases} \\
	\label{eq:F-birth}
		\FA\left(\textcobordism[0](sB)\right)&\colon\scalars\to A,\phantom{A\otimes A}\quad
			\begin{cases}
				1\mapsto v_+,
			\end{cases}\\
	\label{eq:F-death}
		\FA\left(\textcobordism[1](sD-)\right)&\colon A\to\scalars,\phantom{A\otimes A}\quad
			\begin{cases}
				v_+\mapsto 0,\\
				v_-\mapsto 1.
			\end{cases}
	\end{align}
	It is easy to see that $\FA$ is a~graded functor.
\end{example}

%%
%%	Other coefficients
%%
\begin{remark}
	Given a~$\scalars$-algebra $R$ we define likewise categories $R\ChCob$ and $\Mod R$ together with a~chronological TQFT $\F_{\!R}\colon R\ChCob\to\Mod R$. In particular, if we consider $\Z$ as a~trivial module, i.e.\ $\permMM$, $\permSS$, and $\permMS$ act as the~identity, $\Z\ChCob$ is the~linear extension of ordinary cobordisms---the~relations \eqref{rel:reverse-orientation}--\eqref{rel:disjoint-union} become equalities---and $\F_\Z$ is the~Khovanov's functor \cite{KhHom}.
\end{remark}

Each of the~parameters $\permMM$, $\permSS$ and $\permMS$ is invertible, so that there are eight $\scalars$-algebra structures on the~ring $\Z$. We shall distinguish two of them:
\begin{itemize}
	\item $\Zev$, on which all $\permMM$, $\permSS$, and $\permMS$ act trivially, and
	\item $\Zodd$, on which $\permMM$ and $\permMS$ act trivially, but $\permSS$ acts as $-1$.
\end{itemize}
We call them the~\emph{even} and \emph{odd integers} respectively. Both are quotients of $\,\Zpi:=\ZpiLong*$, on which $\permMM$ and $\permMS$ act trivially, but $\permSS\cdot x := \pi x$.

%% file: complex.tex
%% ====================================================================
%%	Complex
%%
We shall now briefly describe the~construction of the~generalized Khovanov complex. We encourage the~reader to refer to Fig.~\ref{fig:trefoil-cube} frequently while reading this section; it illustrates the~construction for the~right-handed trefoil.

\subsection{The~cube of resolutions}
%%
%%	Kauffman states and resolutions
%%
Fix a~link diagram $D$ and enumerate its crossings. Given a~sequence $\xi=(\xi_1,\dots,\xi_n)$, where $\xi_i\in\{0,1\}$ and $n$ is the~number of crossings in $D$, let $D_\xi$ be the~collection of circles obtained by resolving each crossing according to the~following rule:
\begin{displaymath}
		\psset{unit=5mm}
		\begin{centerpict}(-1,-1.2)(1,1.2)
			\psbezier(-1,-1)(0,-0.1)(0,-0.1)(1,-1)
			\psbezier(-1, 1)(0, 0.1)(0, 0.1)(1, 1)
		\end{centerpict}
			\quad\xleftarrow{\ \xi_i=0\ }\quad
		\begin{centerpict}(-1,-1.2)(1,1.2)
			\psline(-1,-1)(1,1)
			\psline[border=5\pslinewidth](-1,1)(1,-1)
			\rput[b](0,1.3ex){$\scriptstyle i$}
		\end{centerpict}
			\quad\xrightarrow{\ \xi_i=1\ }\quad
		\begin{centerpict}(-1,-1.2)(1,1.2)
			\psbezier(-1,-1)(-0.1,0)(-0.1,0)(-1,1)
			\psbezier( 1,-1)( 0.1,0)( 0.1,0)( 1,1)
		\end{centerpict}
\end{displaymath}
The~diagrams $D_\xi$ decorate vertices of an~$n$-dimensional cube $\KhCube{D}$, called the~\emph{cube of resolutions} of $D$. Let $\|\xi\|:=\xi_1+\ldots+\xi_n$ be the~\emph{weight} of the~vertex $\xi$. Consider an~edge $\zeta\colon\xi\to\xi'$ oriented towards the~vertex with higher weight. The~diagrams $D_\xi$ and $D_{\xi'}$ differ only in a~smoothing of a~single crossing, and we decorate the~edge $\zeta$ with the~simplest possible cobordisms between the~two pictures $D_\zeta\subset\R^2\times I$, which is a~vertical surface except a~small neighborhood of the~crossing whose smoothing is changed---here we insert a~saddle \drawSaddle.\footnote{
	In Fig.~\ref{fig:trefoil-cube} we use the~surgery description for cobordisms: the~input circles together with an~arc, a~surgery along which results in the~output circles. The~arc is oriented, which determines an~orientation of the~saddle. A~3D picture of one cobordism is provided in the~left bottom corner.
} Decorate crossings of the~link diagram $D$ with small arrows, which connects the~two arcs in type 0 resolution---they determine uniquely orientations of saddle points of the~cobordisms $D_\zeta$, so that $\KhCube{D}$ can be regarded as a~diagram in the~category $\kChCob$.

%%
%%	Picture
%%
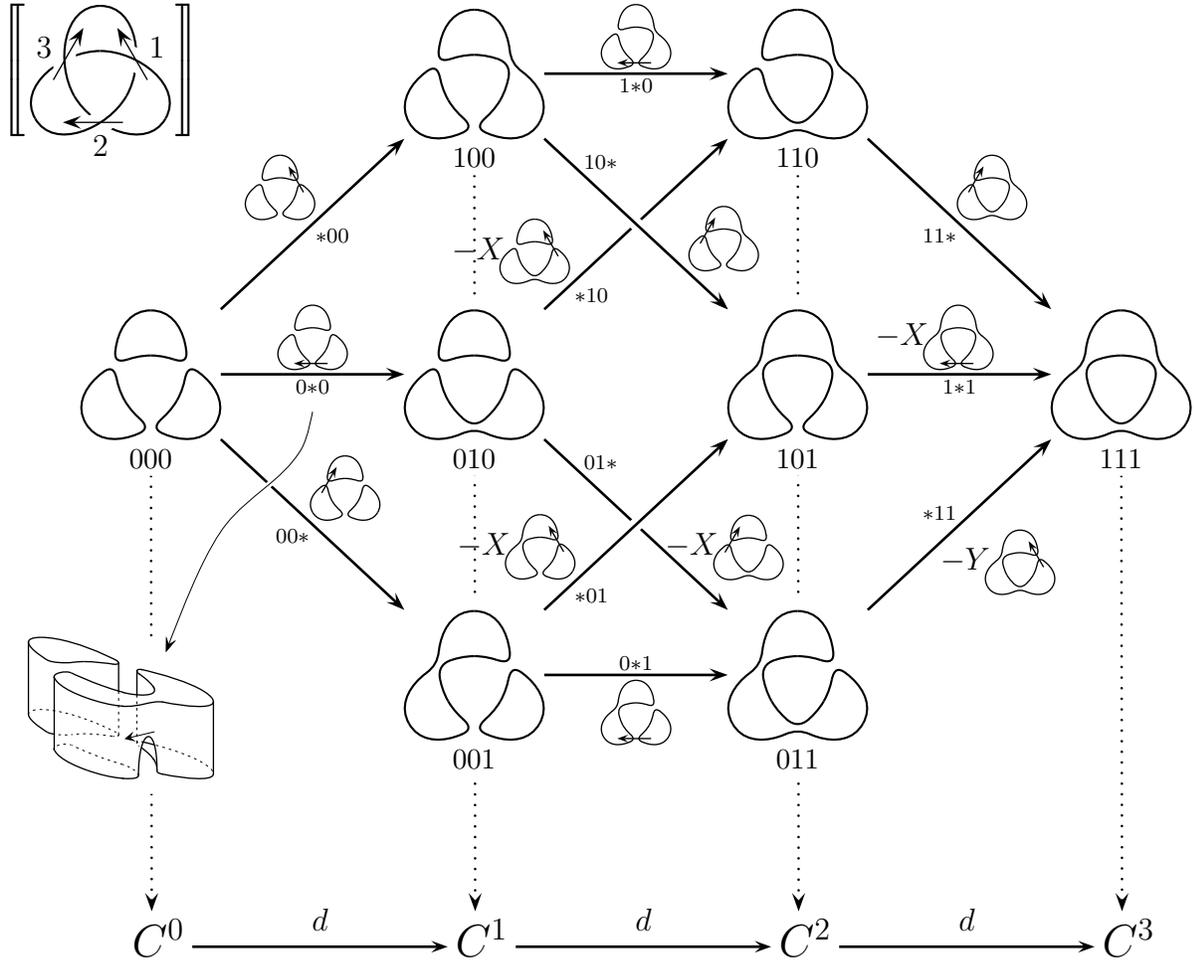
\begin{figure}
	\begin{center}\input{images/khovanov-pict.tex}\end{center}
	\caption{The~cube of resolutions and the~generalized Khovanov complex for the~right-handed trefoil.}
	\label{fig:trefoil-cube}
\end{figure}

\subsection{Sign assignments}
%%
%%	Sign assignments
%%
The~cube $\KhCube{D}$ does not commute, but there is a~cubical cocycle $\psi\in C^2(I^n;\invScalars)$, the~\emph{commutativity obstruction}, such that for every face $S$ of the~form
\begin{equation}\label{diag:ex-face}
	\begin{diagps}(-2cm,-1.1cm)(2cm,1.4cm)
		\psset{labelsep=1pt}
		\node 00(-1.5cm, 0cm)[D_{00}]
		\node 01( 0.0cm,-1cm)[D_{01}]
		\node 10( 0.0cm, 1cm)[D_{10}]
		\node 11( 1.5cm, 0cm)[D_{11}]
		\arrow|a|{->}[00`10;W_{{\star}0}]
		\arrow|a|{->}[10`11;W_{1{\star}}]
		\arrow|b|{->}[00`01;W_{0{\star}}]
		\arrow|b|{->}[01`11;W_{{\star}1}]
	\end{diagps}
\end{equation}
%
%\noindent
%\wrapfigure[r]<0>{\begin{pspicture}(-0.8,0)(0.8,2.5)
%	\rput(0,0.4){\diagConnT{}{->}{}{<-}}
%	\rput(0,1.8){\diagConnT{}{->}{}{->}}
%\end{pspicture}}
the~twisted commutativity $W_{1{\star}}W_{{\star}0} = \psi(S)W_{{\star}1}W_{0{\star}}$ holds. Given a~face $S$ as in \eqref{diag:ex-face} consider the~resolution $D_{00}$ with two arrows orienting the~four cobordisms. After removing isolated circles, i.e.\ those not touched by any arrow, we end up with one of the~diagrams listed in Tab.~\ref{tab:psi-values}; the~table defines the~value of $\psi(S)$. In the~view of Theorem~\ref{thm:kChCob-nondeg} in most cases $\psi(S)$ is determined by the~relations \eqref{rel:reverse-orientation}, \eqref{rel:connected-sum}, and \eqref{rel:disjoint-union}. The~only exceptions are the~one circle configurations at the~bottom---in these cases the~two compositions in \eqref{diag:ex-face} represent diffeomorphic cobordisms with positive genus, so that $\psi(S)$ can be either $1$ or $\permMM\permSS$.

If $\delta\epsilon=-\psi$ for a~cubical cochain $\epsilon\in C^1(I^n;\invScalars)$, the~corrected cube $\KhCubeSigned{D}\epsilon$, in which each cobordism $D_\zeta$ is multiplied by $\epsilon(\zeta)$, anticommutes. We call such a~1-cochain a~\emph{sign assignment} following \cite{OddKh}. It is shown in \cite{ChCob} that such a~sign assignment always exists and is unique up to an~isomorphism of cubes.

\input{tables/cochain.tex}

%We list the~values of $\psi$ in Tab.~\ref{tab:psi-values}. Here a~face is represented by its \emph{configuration diagram}. Namely, take the~resolution $D_{00}$ as in \eqref{diag:ex-face} with two arrows decorating the~two crossings which smoothings are changed, and remove circles that are not touched by the~arrows. What remains is one of the~diagrams in Tab.~\ref{tab:psi-values}. Notice that in most cases the~value of $\psi$ is determined by relations \eqref{rel:reverse-orientation}, \eqref{rel:connected-sum}, and \eqref{rel:disjoint-union}.

\subsection{Grading}

Let $\ell(D_\xi)$ be the~number of circles in the~state $D_\xi$ of a~link diagram $D$. We refine the~cube of resolutions $\KhCubeSigned{D}{\epsilon}$ to the~\emph{graded cube} $\KhCubeGradedSigned{D}{\epsilon}$ by setting
\begin{equation}
	\KhCubeGradedSigned{D}{\epsilon}_\xi :=
		D_\xi\left\{\deg W_\xi\right\} =
		D_\xi\left\{
			\tfrac{\|\xi\|-\ell(\xi)+\ell_0}{2},
			\tfrac{\|\xi\|+\ell(\xi)-\ell_0}{2}
		\right\},
\end{equation}
where $\|\xi\| := \xi_1+\dots+\xi_n$ is the~weight of the~vertex $\xi=(\xi_1,\dots,\xi_n)$, $W_\xi$ is a~cobordism encoded by any directed path from the~initial vertex $(0\cdots 0)$ to $\xi$, and $\ell_0:=\ell(D_0)$ is the~number of circles in the~initial state (all crossing are smoothed in type 0). Commutativity of the~cube guarantees $\deg W_\xi$ is independent of the~path chosen, and it can be computed from the~number of input and output circles using Lemma~\ref{lem:chdeg-vs-bdry}. All morphisms in $\KhCubeGradedSigned{D}{\epsilon}$ are graded.\footnote{
	A~morphism $f$ is \emph{graded} if $\deg f = (0,0)$.
}

The~integral grading from \cites{KhHom,OddKh,ChCob}, wich we write as $\deg_q$, can be recovered from this construction by adding the~two components of $\chdeg$. For cobordisms it is equal to the~Euler characteristic: $\deg_q W = \chi(W)$. We shall refer to it as the~\emph{collapsed grading}.

\subsection{The~complex}
%%
%%	The Khovanov complex
%%
The~generalized Khovanov complex is constructed in the~\emph{additive closure} $\catAdd{\kChCob_0}$ of the~category $\kChCob_0$, which objects are finite sequence (vectors) of objects from $\kChCob_0$, and morphisms are matrices with linear combinations of chronological cobordisms as its entries; a~direct sum in $\catAdd{\kChCob_0}$ is realized by concatenation of sequences.

\begin{definition}
	Let $D$ be a~link diagram with enumerated and oriented crossings. Given a~sign assignment $\epsilon$ for the~cube $\KhCube{D}$ we define the~\emph{generalized Khovanov bracket} as the~chain complex $\KhBracket{D}_\epsilon$ in the~category $\catAdd{\kChCob_0}$ with
	\begin{equation}\label{eq:bracket-def}
		\KhBracket{D}_\epsilon^i := \bigoplus_{|\xi|=i} D_\xi\left\{
			\tfrac{\|\xi\|-\ell(\xi)+\ell_0}{2},
			\tfrac{\|\xi\|+\ell(\xi)-\ell_0}{2}
		\right\},
			\hskip 1cm
		d^i|_{D_\xi} := \sum_{\mathclap{\zeta\colon\xi\to<0.7em>\xi'}} \epsilon(\zeta)D_\zeta.
	\end{equation}
	The~\emph{generalized Khovanov complex} $\KhCom(D)$ is obtained from $\KhBracket{D}_\epsilon$ by shifting both homological and internal grading: $\KhCom^i(D) := \KhBracket{D}_\epsilon^{i+n_-}\!\left\{\frac{n_+-\ell_0}{2}-n_-,\frac{n_++\ell_0}{2}-n_-\right\}$.
\end{definition}

One can think of $\KhBracket{D}_\epsilon$ as the~cube $\KhCubeGradedSigned D\epsilon$ collapsed along diagonals, which is illustrated by dotted arrows in Fig.~\ref{fig:trefoil-cube}.

\begin{remark}\label{rmk:half-int-deg-shift}
	It seems more natural to shift the~degrees in $\KhCube{D}$ by $\left(\tfrac{\|\xi\|-\ell(\xi)}{2},\tfrac{\|\xi\|+\ell(\xi)}{2}\right)$ without the~terms $\ell_0$, as the~latter is removed in $\KhCom(D)$. However, the~numbers $\|\xi\|\pm \ell(\xi)$ can be odd leading to a~grading by half-integers, in which case the~formula \eqref{rel:disjoint-union} requires square roots of $\permMM$, $\permSS$, and $\permMS$ to make sense. We shall not encounter this problem with the~chosen convention---all construction will be defined for the~bracket, and a~global degree shift does not introduce any additional signs.
\end{remark}

\begin{theorem}\label{thm:invariance}
	The~homotopy type of the~generalized Khovanov complex $\KhCom(D)$ is a~link invariant, when regarded as a~complex in the~category $\catAdd{\kChCob_0}$ modulo the~following three local relations:
	\begin{align*}
		\textnormal{(\textit{S})}&\psset{unit=1cm}\quad\pictRelS = 0
		\hskip 2cm
		\textnormal{(\textit{T})}\quad\pictRelT = \permMS(\permMM+\permSS)\\
		\textnormal{(\textit{4Tu})}&\quad\psset{unit=0.75cm}%
			\permMS\pictRelTuL + \permMS\pictRelTuR = \permMM\pictRelTuB + \permSS\pictRelTuT
	\end{align*}
	in which all deaths are oriented clockwise.
\end{theorem}

\begin{proof}
	We showed in \cite{ChCob} that the~isomorphism class of $\KhBracket{D}_\epsilon$ depends only on the~diagram $D$, and all isomorphisms involved are induced by cobordisms with no critical points. Hence,	this still holds in the~bigraded framework. The~relations \textit{S}, \textit{T}, \textit{4Tu} are clearly homogeneous, so that the~quotient category is still graded. It remains to show the~chain homotopy equivalences from \cite{ChCob} are graded.
	
	\proofpart{First Reidemeister move}
	\wrapfigure[r]<4>{\psset{unit=0.75cm}%
		\begin{pspicture}(-5.2,0)(5.5,5.7)
			\diagnode tl(0,5)[\vcenter{\hbox{\tangleRIh}}]
			\diagnode tr(5,5)[0]
			\diagnode bl(0,1)[\vcenter{\hbox{\tangleRIv}}]
			\diagnode br(5,1)[\vcenter{\hbox{\tangleRIh}}]
			\diagarrow|a|{<->}[tl`tr;0]
			\diagarrow|a|{<->}[tr`br;0]
			\diagarrow[offset=-2pt]|b{labelsep=3pt}|{->}[bl`br;d\,=\,\epsilon\fntCobRId]
			\diagarrow[offset= 2pt]|a{labelsep=3pt}|{<-}[bl`br;h\,=\,-\epsilon^{-1}\fntCobRIh]
			\diagarrow[offset=-2pt]|b{labelsep=3pt}|{->}[tl`bl;\frac\permSS\alpha\left(\!\permMM\fntCobRIfa -\,\permMS\fntCobRIfb\!\right)\,=\,f]
			\diagarrow[offset= 2pt]|a{npos=0.3,labelsep=3pt}|{<-}[tl`bl;g\,=\,\alpha\permMM\permSM\fntCobRIg]
		\end{pspicture}}
	The~bracket $\KhBracket*\fntRIx$ is the~mapping cone of $\KhBracket*\fntRIv\to\KhBracket*\fntRIh$, the chain map induced by the~edges in $\KhCube\fntRIx$ associated with the~distinguished crossing. The~chain homotopy equivalences between $\KhBracket*\fntRIx$ and $\KhBracket*\fntRIh$ from \cite{ChCob} are induced by morphisms of cubes $f\colon \KhCube{\fntRIh} \twoways \KhCube{\fntRIv} \cocolon g$ with components visualized in the~diagram below. Here, $\epsilon$ comes from the~sign assignment used to build $\KhBracket*\fntRIx$, and $\alpha\in\scalars$ is chosen for each component of $f$ and $g$ separately to make them commute with other edge morphisms in the~cubes.	Because $g\colon \KhBracket\fntRIx \to \KhBracket\fntRIh$ is induced by a~death, $\deg g = (0,1)$. Likewise, the~two components of $f\colon \KhBracket\fntRIh \to \KhBracket\fntRIx$ are cobordisms of degree $(0,-1)$. Since $\fntRIv$ has one more circle than $\fntRIh$, both $f$ and $g$ become graded morphisms when the~degree shifts are applied.

	\parshape=0
	\proofpart{Second Reidemeister move}
	Here we consider $\KhBracket*{\fntRIIxx}$ as a~total complex of the~bicomplex
	\begin{displaymath}
		0	\to	\KhBracket*{\fntRIIvh}
			\to \KhBracket*{\fntRIIvv} \oplus \KhBracket*{\fntRIIhh}
			\to \KhBracket*{\fntRIIhv}
			\to 0.
	\end{displaymath}
	The~chain homotopy equivalences between $\KhBracket*{\fntRIIxx}$ and $\KhBracket*{\fntRIIhhs}$ are then induced by maps shown in Fig.~\ref{fig:R-II}. The~morphisms $\fntRIIhhs \twoways \fntRIIhh$ are graded: both diagrams have the~same number of circles, and the~difference between heights of corresponding vertices in cubes ($\fntRIIhh$ has one more type $1$ resolution than $\fntRIIhhs$) is compensated by the~difference between numbers of crossings. The~other components of $f$ and $g$ are also graded morphisms, which follows from their definitions
	\begin{align}
		\left(\fntRIIhhs \to^f \fntRIIvv\right) &:= 
			\left(\fntRIIhhs 
				\to<2em>^{\id}\fntRIIhh
				\to<2em>^{d_{1{*}}}\fntRIIhv
				\to<2em>^{h_{{*}1}}\fntRIIvv\right) \\
		\left(\fntRIIvv \to^g \fntRIIhhs\right)	&:= 
			\left(\fntRIIvv
				\to<2em>^{h_{0{*}}}\fntRIIvh
				\to<2em>^{d_{{*}0}}\fntRIIhh
				\to<2em>^{\id}\fntRIIhhs\right)
	\end{align}
	and the~fact that both $d$ and $h$ are graded morphisms in the~graded complex ($d$ is graded by the~construction and $h$ has the~opposite degree to $d$).

\begin{figure}[h]
	\psset{unit=0.75cm}
	\begin{pspicture}(0,0)(12,9)
		\newpsobject{diagarc}{ncarc}{%
				linewidth=0.5pt,doublesep=1.5pt,
				arrowsize=4pt 1,arrowlength=0.8,arrowinset=0.6,
				labelsep=2pt,nodesep=3pt}
		\rput[l](-1,8.0){$\KhBracket*{\vcenter{\hbox{\tangleRIIhhs}}}:$}
		\rput[l](-1,2.5){$\KhBracket*{\vcenter{\hbox{\tangleRIIxx}}}\![\,1]:$}
		\psline{<->}( 0,6.25)(0,4.25)\uput[ur]( 0,6.25){$\scriptstyle g$}\uput[dr](0,4.25){$\scriptstyle f$}
		\psline{<->}(-1,5.25)(1,5.25)\uput[ul](-1,5.25){$\scriptstyle h$}\uput[ur](1,5.25){$\scriptstyle d$}
		\rput( 3.5,8.0){\rnode{t-1}{$0$}}
		\rput( 8.0,8.0){\rnode{t 0}{$\KhBracket*{\vcenter{\hbox{\tangleRIIhhs}}}$}}
		\rput(12.5,8.0){\rnode{t 1}{$0$}}
		\rput( 3.5,2.5){%
			\rlap{\raisebox{-3ex}[0pt][0pt]{$\scriptstyle\mskip\thickmuskip 00$}}%
			\rnode{b00}{$\KhBracket*{\vcenter{\hbox{\tangleRIIvh}}}$}%
		}
		\rput( 7.0,1.0){%
			\rlap{\raisebox{-3ex}[0pt][0pt]{$\scriptstyle\mskip\thickmuskip 01$}}%
			\rnode{b01}{$\KhBracket*{\vcenter{\hbox{\tangleRIIvv}}}$}%
		}
		\rput( 9.0,4.0){%
			\rlap{\raisebox{-3ex}[0pt][0pt]{$\scriptstyle\mskip\thickmuskip 10$}}%
			\rnode{b10}{$\KhBracket*{\vcenter{\hbox{\tangleRIIhh}}}$}%
		}
		\rput(12.5,2.5){%
			\rlap{\raisebox{-3ex}[0pt][0pt]{$\scriptstyle\mskip\thickmuskip 11$}}%
			\rnode{b11}{$\KhBracket*{\vcenter{\hbox{\tangleRIIhv}}}$}%
		}
		\diagline{->}{t-1}{t 0}\naput{$\scriptstyle 0$}
		\diagline{->}{t 0}{t 1}\naput{$\scriptstyle 0$}
		\diagline{->}{b00}{b10}\naput[labelsep=-1pt,npos=0.3]{$\scriptstyle \epsilon_{{*}0}\fntCob{R2-d*0}$}
		\diagline{->}{b10}{b11}\naput[labelsep=-7pt,npos=0.4]{$\scriptstyle \epsilon_{1{*}}\fntCob{R2-d1*}$}
		\diagline[offset= 2pt]{->}{b00}{b01}
				\naput[labelsep=-6pt,npos=0.5]{$\scriptstyle \epsilon_{0{*}}\fntCob{R2-d0*}$}
		\diagline[offset= 2pt]{->}{b01}{b11}
				\naput[labelsep=-1pt,npos=0.3]{$\scriptstyle \epsilon_{{*}1}\fntCob{R2-d*1}$}
		\diagline[offset=-2pt]{<-}{b00}{b01}
				\nbput[labelsep= 0pt,npos=0.5]{$\scriptstyle-\epsilon_{0{*}}^{-1}\permSS\fntCob{R2-h0*}$}
		\diagline[offset=-2pt]{<-}{b01}{b11}
				\nbput[labelsep=-6pt,npos=0.5]{$\scriptstyle-\epsilon_{{*}1}^{-1}\fntCob{R2-h*1}$}
		\diagline{<->}{b00}{t-1}\naput{$\scriptstyle 0$}
		\diagline{<->}{b11}{t 1}\nbput{$\scriptstyle 0$}
		\diagarc[arcangle= 12]{<->}{t 0}{b10}
				\naput[npos=0.3]{$\scriptstyle\id$}
		\diagarc[arcangle= 12,offset= 2pt,border=3\pslinewidth]{->}{b01}{t 0}
				\naput[labelsep=-2pt,npos=0.75]{$\scriptstyle\gamma\fntCob{R2-G}$}
		\diagarc[arcangle= 12,offset=-2pt,border=3\pslinewidth]{<-}{b01}{t 0}
				\nbput[labelsep=-1pt,npos=0.65]{$\scriptstyle\varphi\!\fntCob{R2-F}$}
		\rput[c](10.7,6){$\setlength\arraycolsep{0pt}\begin{array}{rl}
				\scriptstyle\gamma  &\scriptstyle= -\epsilon_{{*}0}^{\vphantom{-1}}\epsilon_{0{*}}^{-1}\permSS\\
				\scriptstyle\varphi &\scriptstyle= -\epsilon_{1{*}}^{\vphantom{-1}}\epsilon_{{*}1}^{-1}
			\end{array}$}
	\end{pspicture}
	\caption{Chain homotopy equivalences for the~second Reidemeister move.}\label{fig:R-II}
\end{figure}

	\proofpart{Third Reidemeister move}
	This case follows from a~strictly algebraic argument: the~complex $\KhBracket\fntRIIIax$ is the~mapping cone of the~chain map $\KhBracket\fntRIIIac\colon \KhBracket\fntRIIIah \to \KhBracket\fntRIIIav$, and composing it with the~chain homotopy equivalence $f\colon \KhBracket\fntIIIVert \to \KhBracket\fntRIIIah$ does not change the~homotopy type of the~mapping cone, see \cite{ChCob}. Hence, $\KhBracket\fntRIIIax\simeq\cone(\KhBracket\fntIIIVert \to \KhBracket\fntRIIIav)$ via degree-preserving chain homotopy equivalences, and similarly for $\KhBracket\fntRIIIbx$.
\end{proof}

%% file: images/khovanov-pict.tex
\begingroup
\psset{unit=1cm}
\def\cobordism{%
	\rput(0,0){\psclip{\pscustom[linestyle=none]{%
			\moveto(-1.40851, 0.73936)
			\lineto(-0.97994, 0.73936)
			\lineto(-0.97994, 0.00325)
			\lineto( 0.41979,-0.52933)
			\curveto(0.41979, 0.47067)( 0.72783, 0.58117)(0.72783,-0.41883)
			\lineto( 1.66534,-0.49630)
			\lineto( 1.66534, 1)
			\lineto(-1.40851, 1)}}
		\psset{linestyle=dashed,dash=1pt 1.5pt,linewidth=0.15pt}
		\psbezier(-1.2,0.84)(-0.6,0.84)( 0.43607,0.19574)( 0.65569,-0.168)
		\psbezier(-1.2,0.84)(-1.8,0.84)(-0.99532,0.19574)(-0.17569,-0.168)
		\psbezier(-0.43924,-0.42)(-1.25766,-0.05626)(-1.08,0.336)(-0.48,0.336)
		\psbezier( 1.63924,-0.42)( 1.41841,-0.05626)( 0.12,0.336)(-0.48,0.336)
		\psbezier( 1.63924,-0.42)( 1.85886,-0.78374)(0.64393,-0.53174)(-0.17569,-0.168)
		\psbezier(-0.43924,-0.42)( 0.38039,-0.78374)(0.87532,-0.53174)( 0.65569,-0.168)
		\psset{linestyle=none,fillcolor=white,fillstyle=solid}
		\psellipse(0.48,-0.462)(0.35,0.35)
		\psellipse(0.24, 0.231)(0.25,0.25)
		\psellipse(-0.9, 0.231)(0.25,0.25)
		\psset{linestyle=dashed,fillstyle=none}
		\psbezier(0.03970,0.37025)(0.12403,0.31650)(0.09682,0.26530)(-0.00150,0.28352)
		\psbezier(0.43596,0.07651)(0.36107,0.14277)(0.39169,0.19477)( 0.48781,0.16718)
		\psbezier(-0.97994,0.00325)(-1.01341,0.09748)(-0.78568,0.15397)(-0.69968,0.10295)
		\psbezier(-0.67676,0.31978)(-0.77519,0.30210)(-1.01856,0.30893)(-1.09755,0.37025)
		\psbezier(0.15998,-0.30309)(0.34781,-0.37181)(0.64296,-0.64475)(0.14597,-0.58989)
		\psbezier(0.82952,-0.50660)(0.63507,-0.45979)(0.81076,-0.42482)(0.65569,-0.168)
	\endpsclip}
	\rput(0,0){\psclip{\pscustom[linestyle=none]{%
			\moveto(-1.40851, 0.73936)
			\lineto(-0.97994, 0.73936)
			\lineto(-0.97994, 0.00325)
			\lineto( 0.41979,-0.52933)
			\curveto(0.41979, 0.47067)( 0.72783, 0.58117)(0.72783,-0.41883)
			\lineto( 1.66534,-0.49630)
			\lineto( 1.8,-0.49630)
			\lineto( 1.8,-1)
			\lineto(-1.6,-1)
			\lineto(-1.6, 0.73936)}}
		\psset{linestyle=solid,linewidth=0.5pt}
		\psbezier(-1.2,0.84)(-0.6,0.84)( 0.43607,0.19574)( 0.65569,-0.168)
		\psbezier(-1.2,0.84)(-1.8,0.84)(-0.99532,0.19574)(-0.17569,-0.168)
		\psbezier(-0.43924,-0.42)(-1.25766,-0.05626)(-1.08,0.336)(-0.48,0.336)
		\psbezier( 1.63924,-0.42)( 1.41841,-0.05626)( 0.12,0.336)(-0.48,0.336)
		\psbezier( 1.63924,-0.42)( 1.85886,-0.78374)(0.64393,-0.53174)(-0.17569,-0.168)
		\psbezier(-0.43924,-0.42)( 0.38039,-0.78374)(0.87532,-0.53174)( 0.65569,-0.168)
		\psset{linestyle=none,fillcolor=white,fillstyle=solid}
		\psellipse(0.48,-0.462)(0.35,0.32)
		\psellipse(0.24, 0.231)(0.25,0.25)
		\psellipse(-0.9, 0.231)(0.25,0.25)
		\psset{linestyle=solid,fillstyle=none}
		\psbezier(0.03970,0.37025)(0.12403,0.31650)(0.09682,0.26530)(-0.00150,0.28352)
		\psbezier(0.43596,0.07651)(0.36107,0.14277)(0.39169,0.19477)( 0.48781,0.16718)
		\psbezier(-0.97994,0.00325)(-1.01341,0.09748)(-0.78568,0.15397)(-0.69968,0.10295)
		\psbezier(-0.67676,0.31978)(-0.77519,0.30210)(-1.01856,0.30893)(-1.09755,0.37025)
		\psbezier(0.15998,-0.30309)(0.34781,-0.37181)(0.64296,-0.64475)(0.14597,-0.58989)
		\psbezier(0.82952,-0.50660)(0.63507,-0.45979)(0.81076,-0.42482)(0.65569,-0.168)
	\endpsclip}
	\rput(0,1.5){%
		\psset{linestyle=solid,linewidth=0.5pt}
		\psbezier(-1.2,0.84)(-0.6,0.84)( 0.43607,0.19574)( 0.65569,-0.168)
		\psbezier(-1.2,0.84)(-1.8,0.84)(-0.99532,0.19574)(-0.17569,-0.168)
		\psbezier(-0.43924,-0.42)(-1.25766,-0.05626)(-1.08,0.336)(-0.48,0.336)
		\psbezier( 1.63924,-0.42)( 1.41841,-0.05626)( 0.12,0.336)(-0.48,0.336)
		\psbezier( 1.63924,-0.42)( 1.85886,-0.78374)(0.64393,-0.53174)(-0.17569,-0.168)
		\psbezier(-0.43924,-0.42)( 0.38039,-0.78374)(0.87532,-0.53174)( 0.65569,-0.168)
		\psset{linestyle=none,fillcolor=white,fillstyle=solid}
		\psellipse(0.48,-0.462)(0.35,0.32)
		\psellipse(0.24, 0.231)(0.25,0.25)
		\psellipse(-0.9, 0.231)(0.25,0.25)
		\psset{linestyle=solid,fillstyle=none}
		\psbezier(0.03970,0.37025)(0.12403,0.31650)(0.09682,0.26530)(-0.00150,0.28352)
		\psbezier(0.43596,0.07651)(0.36107,0.14277)(0.39169,0.19477)( 0.48781,0.16718)
		\psbezier(-0.97994,0.00325)(-1.01341,0.09748)(-0.78568,0.15397)(-0.69968,0.10295)
		\psbezier(-0.67676,0.31978)(-0.77519,0.30210)(-1.01856,0.30893)(-1.09755,0.37025)
		\psbezier(0.15998,-0.30309)(0.34781,-0.37181)(0.81076,-0.42482)(0.65569,-0.168)
		\psbezier(0.82952,-0.50660)(0.63507,-0.45979)(0.44416,-0.62280)(0.14597,-0.58989)
	}
	\rput(0,0){%
		\psset{linewidth=0.5pt,dash=1pt 1.5pt}
		\psline{c-c}(-1.40851, 0.73936)(-1.40851,2.25000)
		\psline{c-c}(-0.97994, 0.00325)(-0.97994,1.50325)
		\psdash\psline{c-c}( 0.08996, 0.31025)( 0.08996,1.221)
		\psdash\psline{c-c}( 0.39409, 0.14265)( 0.39409,1.14265)
		\psline{-c}( 0.08996, 1.221)( 0.08996,1.81025)
		\psline{-c}( 0.39409, 1.14265)( 0.39409,1.64265)
		\psline{c-c}( 1.66534,-0.49630)( 1.66534,1.00370)
		\psbezier{c-c}( 0.41979,-0.52933)( 0.41979,0.47067)( 0.72783, 0.58117)(0.72783,-0.41883)
		\rput(0.58652,0.2782){\psline[linewidth=0.3pt,arrowsize=4pt,arrowlength=1]{->}(0.1,0.1)(-0.4,-0.05)}
	}
}

\def\cubevertex#1#2#3{\begingroup
	\psset{linewidth=0.8pt,linestyle=solid}%
	\begin{pspicture}(-1,-0.8)(1,1.1)%
		\pspolygon[fillcolor=white,fillstyle=solid,linestyle=none](-1,-0.8)(1,-0.8)(1,1.2)(-1,1.2)
		\trefoilresolution{#1}{#2}{#3}%
		\rput*[t](0,-0.85){\small $#1#2#3$}%
	\end{pspicture}%
\endgroup}

\def\cubeedge#1#2#3{\begingroup
	\psset{linewidth=0.5pt,linestyle=solid,unit=0.5}%
	\begin{centerpict}(-0.8,-0.7)(0.8,1)%
		\trefoilresolution{#1}{#2}{#3}%
	\end{centerpict}%
\endgroup}

\begin{pspicture}(0,0)(16.5,13)
%	\pspolygon[linewidth=1pt,linestyle=dotted,linecolor=blue](0,0)(16.5,0)(16.5,13)(0,13)
	% Top left corner: knot diagram
	\rput[tl](0,13){$\left\llbracket\trefoil aca\right\rrbracket$}
	% Bottom: chain complex
	\rput( 2.1, 0.5){\Rnode[href=-0.25]{C0}{\Large $C^0$}}
	\rput( 6.4, 0.5){\Rnode[href=-0.25]{C1}{\Large $C^1$}}
	\rput(10.7, 0.5){\Rnode[href=-0.25]{C2}{\Large $C^2$}}
	\rput(15.0, 0.5){\Rnode[href=-0.25]{C3}{\Large $C^3$}}
	\begingroup
	\psset{linestyle=solid,linewidth=1pt,nodesep=3pt,arrowsize=4.5pt}
	\ncline{->}{C0}{C1}	\naput{$d$}
	\ncline{->}{C1}{C2}	\naput{$d$}
	\ncline{->}{C2}{C3}	\naput{$d$}
	\endgroup
	% Vertical dotted lines --- direct sum
	\pnode( 2.0, 8){col0}
	\pnode( 6.3,12){col1}
	\pnode(10.6,12){col2}
	\pnode(14.9, 8){col3}
	\begingroup
	\psset{linestyle=dotted,linewidth=1pt,nodesep=3pt,arrowsize=5pt}
	\ncline{->}{col0}{C0}
	\ncline{->}{col1}{C1}	%\uput[r]( 6.3,1.7){$\cdot\,\{1\}$}
	\ncline{->}{col2}{C2}	%\uput[r](10.6,1.7){$\cdot\,\{2\}$}
	\ncline{->}{col3}{C3}	%\uput[r](14.9,1.7){$\cdot\,\{3\}$}
	\endgroup
	% Resolutions
	\rput( 2.0, 8){\rnode{000}{\cubevertex{0}{0}{0}}}
	\rput( 6.3,12){\rnode{100}{\cubevertex{1}{0}{0}}}
	\rput( 6.3, 8){\rnode{010}{\cubevertex{0}{1}{0}}}
	\rput( 6.3, 4){\rnode{001}{\cubevertex{0}{0}{1}}}
	\rput(10.6,12){\rnode{110}{\cubevertex{1}{1}{0}}}
	\rput(10.6, 8){\rnode{101}{\cubevertex{1}{0}{1}}}
	\rput(10.6, 4){\rnode{011}{\cubevertex{0}{1}{1}}}
	\rput(14.9, 8){\rnode{111}{\cubevertex{1}{1}{1}}}
	\begingroup
	\psset{linestyle=solid,linewidth=1pt,nodesep=-2pt,labelsep=1pt,arrowsize=5pt}
	\ncline{->}{000}{100}		\nbput{$\scriptstyle {*}00$}	\naput{\cubeedge{a}{0}{0}}
	\ncline{->}{000}{010}		\nbput{$\scriptstyle 0{*}0$}	\naput{\cubeedge{0}{c}{0}}
	\ncline{->}{000}{001}		\nbput{$\scriptstyle 00{*}$}	\naput{\cubeedge{0}{0}{a}}
	\ncline{->}{010}{110}	\nbput[npos=0.15]{$\scriptstyle {*}10$}
												\naput[npos=0.12]{$-\permMM$\cubeedge{a}{1}{0}}
	\ncline{->}{010}{011}	\naput[npos=0.20]{$\scriptstyle 01{*}$}
												\naput[npos=0.85]{$\!\!\!\!\!\!\!-\permMM$\cubeedge{0}{1}{a}}
	\ncline{->}{100}{110}	\nbput[npos=0.50]{$\scriptstyle 1{*}0$}
												\naput{\cubeedge{1}{c}{0}}
	\ncline[border=2pt]{->}{100}{101}	\naput[npos=0.20]{$\scriptstyle 10{*}$}
																		\naput[npos=0.80]{\cubeedge{1}{0}{a}}
	\ncline[border=2pt]{->}{001}{101}	\nbput[npos=0.15]{$\scriptstyle {*}01$}
																		\naput[npos=0.15]{$-\permMM$\cubeedge{a}{0}{1}}
	\ncline{->}{001}{011}	\naput{$\scriptstyle 0{*}1$}
												\nbput{\cubeedge{0}{c}{1}}
	\ncline{->}{110}{111}	\nbput[npos=0.5]{$\scriptstyle 11{*}$}
												\naput{\cubeedge{1}{1}{a}}
	\ncline{->}{101}{111}	\nbput[npos=0.5]{$\scriptstyle 1{*}1$}
												\naput[npos=0.35]{$-\permMM$\cubeedge{1}{c}{1}}
	\ncline{->}{011}{111}	\naput[npos=0.5]{$\scriptstyle {*}11$}
												\nbput{$\!\!\!\!-\permSS$\cubeedge{a}{1}{1}}
	\endgroup
	%% cobordism
	\rput(1.5,3){\begingroup
		\psframe[linestyle=none,fillstyle=solid,fillcolor=white](0.3,-0.6)(0.7,1.5)
		\psset{yunit=0.64,xunit=0.8}
		\cobordism
	\endgroup}
	\pscurve[border=1pt,linewidth=0.3pt,arrowsize=4pt]{->}(4.15,7.5)(4,7)(3,6)(2.7,5.5)(2.2,4.3)
\end{pspicture}
\endgroup

%% file: tables/cochain.tex
\begin{table}[t]
	\psset{unit=1cm}%
	\centering
	\begin{minipage}{5.5cm}
		\begin{center}
			{\Large $\permMM$}
			\par\pictDisMM{}{-}{}{-}
			\par\pictConnMM{}{-}{}{-}
			\par\pictConnX{}{->}{}{<-}
		\end{center}
	\end{minipage}
	\begin{minipage}{4.5cm}
		\begin{center}
			{\Large $\permSS$}
			\par\pictDisSS{}{-}{}{-}
			\par\pictConnSS{}{-}{}{-}
			\par\pictConnX{}{->}{}{->}
		\end{center}
	\end{minipage}
	\vskip\baselineskip
	\begin{minipage}[t]{5cm}
		\begin{center}
			{\Large $\permMS$}
			\par\pictDisMS{1}{-}{2}{-}
			\par\pictConnMS{1}{-}{2}{-}
		\end{center}
	\end{minipage}
	\begin{minipage}[t]{2.5cm}
		\begin{center}
			{\Large $\permTpos$}
			\vskip 0.75\baselineskip
			\pictConnT{}{->}{}{->}
		\end{center}
	\end{minipage}
	\begin{minipage}[t]{2.5cm}
		\begin{center}
			{\Large $\permTneg$}
			\vskip 0.75\baselineskip
			\pictConnT{}{->}{}{<-}
		\end{center}
	\end{minipage}
	\vskip 0.5\baselineskip
	\caption[Diagrams for faces that can appear in a~cube of resolutions]{%
		Diagrams for faces that can appear in a~cube of resolutions, grouped by values of the~commutativity obstruction $\psi$. Thin lines are the~input circles and thick arrows visualize saddle points. Orientations of the~arrows are omitted if $\psi$ does not depend on them. The~small numbers 1 and 2 in the~two configurations placed under the~letter $\permMS$ indicate an~initial order of critical points (the~upper path in \eqref{diag:ex-face}); take $\permSM$ for the~opposite one.}
	\label{tab:psi-values}
\end{table}

%% file: homology.tex
Applying a~chronological TQFT $\F\colon\kChCob \to \Mod\scalars$ to the~complex $\KhCom(D)$ results in a~chain complex $\F\KhCom(D)$ of bigraded $\scalars$-modules. If $\F$ preserves the~relations \textit{S}, \textit{T}, and \textit{4Tu} from Theorem~\ref{thm:invariance}, the~homology of the~chain complex $\F\KhCom(D)$ is invariant under Reidemeister moves. In particular, we can choose the~TQFT from Example~\ref{ex:Kh-tqft} and write $\Kh(L)$ for homology of the~corresponding chain complex. As usual, given a~$\scalars$-module $M$ we define $\Kh(L;M)$ as the~homology of the~chain complex $\F\KhCom(D)\otimes M$.

\begin{proposition}[cf.\ \cite{ChCob}]
	Choose a~link $L$ and consider $\Kh(L)$ with the~collapsed integral grading. Then the~following holds:
	\begin{enumerate}
		\item the~graded Euler characteristic $\chi(\Kh(L))$ is the~Jones polynomial of $L$,
		\item $\Kh(L;\Zev)$ is the~Khovanov homology of $L$ \cite{KhHom}, and
		\item $\Kh(L;\Zodd)$ is the~odd Khovanov homology of $L$ \cite{OddKh}.
	\end{enumerate}
\end{proposition}

One may expect that the~new gradation on $\KhCom(D)$ results in a~stronger invariant. However, as long as $\F(\fntCircle)$ is generated in degrees $(1,0)$ and $(0,-1)$ the~bigraded graded chain complex $\F\KhCom(D)$ carries no more information than the~one with the~collapsed grading.

\begin{lemma}
	Let $\F\colon\kChCob \to \Mod\scalars$ be a~chronological TQFT such that $A:=\F(\fntCircle)$ is supported in degrees $(1,0)$ and $(0,-1)$ only. Then $\F\KhCom^{i}(D)_{p,q}\neq 0$ only if $p=q$.
\end{lemma}
\begin{proof}
	Let $k:=\ell(D_\xi)$ be the~number of circles in a~state $D_\xi$ of the~diagram $D$. A~homogeneous element $u\in\F(D_\xi) = A^{\otimes k}$ has degree $\deg u = (d,d-k)$ for some $d\in\Z$. When regarded as an~element of $\F\KhCom(D)$ we have to adjust the~degree by $\left(\frac{\|\xi\|-k+n_-}{2}-n_-,\frac{\|\xi\|+k+n_+}{2}-n_-\right)$. Hence, $u\in\F\KhCom^i(D)_{p,p}$ with $p = \frac{\|\xi\|-k+n_+}{2} + d - n_-$.
\end{proof}

\begin{corollary}
	The~graded Euler characteristic of $\Kh(L)$, regarded as a~tripple graded $\scalars$-module, is the~Jones polynomial of $L$ evaluated at $\sqrt{uv}$.
\end{corollary}

Despite this unfortunate result, the~additional grading is actually useful, which will become clear in the~next section.

%% file: composite.tex
Given complexes $C$ and $C'$ in a~graded monoidal category, which differentials $d\colon C_i\to C_{i+1}$ and $d'\colon C'_i\to C'_{i+1}$ are graded, we define their tensor product as the~chain complex $C\otimes C'$ satisfying
\begin{align}
	(C\otimes C')_i &:= \bigoplus_{p+q=i} C_p\otimes C'_q,\\
	d|_{C_p\otimes C'_q} &:= d\otimes\id + (-1)^p\id\otimes d'.
\end{align}
The~role of the~tensor product in $\kChCob$ is played by the~disjoint union $\rdsum$, in which case $\KhCom^p(D)\rdsum\KhCom^q(D')$ is represented by diagrams $D_\xi\sqcup D'_{\xi'}$ with $\|\xi\|=p$ and $\|\xi'\|=q$, which can be seen as resolutions of $D\sqcup D'$. However, the~morphisms $d\sqcup\id$ and $\id\sqcup d$ in $\KhCom(D\sqcup D')$ do not commute. The~reason they do in $\KhCom(D)\rdsum\KhCom(D')$ is the~way the~degree shifts are applied: the~objects $D_\xi\{a,b\}\sqcup D'_{\xi'}\{a',b'\}$ and $(D_\xi\sqcup D'_{\xi'})\{a+a',b+b'\}$ are isomorphic but not equal.% We shall now construct such an~isomorphism.

\begin{definition}
	Let $X$ be an~object of a~$G$-graded monoidal category and $g\in G$. The~\emph{degree shift isomorphism} $i_g\colon X\to X\{g\}$ is the~degree $g$ map corresponding to $\id_X\in\Mor(X,X)$ under the~natural bijection $\Mor(X,X\{g\})\cong\Mor(X,X)$.
\end{definition}

We use the~degree shift isomorphisms to identify both $X\{g\}\otimes Y\{h\}$ and $(X\otimes Y)\{g+h\}$ with $X\otimes Y$. There is some choice for the~first isomorphism: we can take either $i_g^{-1}\otimes i_h^{-1}$ or $(i_g\otimes i_h)^{-1}$, and a~short computation reveals the~two morphisms are not equal:
\begin{equation}
	(i_g^{-1}\otimes i_h^{-1})\circ(i_g\otimes i_h)
		= \lambda(-h,g)(i_g^{-1}\circ i_g)\otimes(i_h^{-1}\circ i_h)
		= \lambda(-h,g)\id.
\end{equation}

The~following result suggests to choose $i_g^{-1}\otimes i_h^{-1}$, as signs will then appear only when a~nontrivial morphism occurs as the~second factor (compare it with \eqref{eq:graded-tensor-functoriality}).

\begin{lemma}\label{lem:degree-vs-tensor}
	Choose an~object $X$ and a~morphism $f\colon A\to B$ of degree $\deg f = a-b$ in a~$G$-graded monoidal category. Then the~following diagrams commute
	\begin{gather}
		\label{diag:tensor-f-id}
		\begin{diagps}(14em,15ex)
			\psset{labelsep=2pt}
			\square(0em,1ex)<14em,11ex>|a`*c`*c`a|[%
				A\{a\}\otimes X\{x\}`(A\otimes X)\{a+x\}`B\{b\}\otimes X\{x\}`(B\otimes X)\{b+x\};%
				i_{a+x}\circ(i_a^{-1}\otimes i_x^{-1})`%
				\overline f\otimes\id`\overline{f\otimes\id}`%
				i_{b+x}\circ(i_b^{-1}\otimes i_x^{-1})%
			]
		\end{diagps}\\[1ex]
		\label{diag:tensor-id-f}
		\begin{diagps}(14em,15ex)
			\psset{labelsep=2pt}
			\square(0em,1ex)<14em,11ex>|a`*c`*c`a|[%
				X\{x\}\otimes A\{a\}`(X\otimes A)\{x+a\}`X\{x\}\otimes B\{b\}`(X\otimes B)\{x+b\};%
				i_{x+a}\circ(i_x^{-1}\otimes i_a^{-1})`%
				\id\otimes\overline f`\lambda(a-b,x)\overline{\id\otimes f}`%
				i_{x+b}\circ(i_x^{-1}\otimes i_b^{-1})%
			]
		\end{diagps}
	\end{gather}
	where $\overline{\vphantom{|}\cdot\vphantom{|}}$ denotes appropriate conjugations by degree shift isomorphisms $i_\bullet$.
\end{lemma}
\begin{proof}
	Follows directly from \eqref{eq:graded-tensor-functoriality}.
%	Follows directly from the~third condition for a~graded product in Definition~\ref{def:graded-product}.
\end{proof}

Choose cubes $\mathcal{I}$ and $\mathcal{I'}$ of dimensions $n$ and $n'$ respectively. We define their product as a~cube $\mathcal{I}\otimes\mathcal{I}'$ of dimension $n+n'$ with vertices decorated with tensor products
\begin{align}
	(\mathcal{I}\otimes\mathcal{I}')(\xi\xi') & := \mathcal{I}(\xi) \otimes \mathcal{I}'(\xi'), \\
\intertext{and edges labeled by the~original maps tensored with identity morphisms:}
	(\mathcal{I}\otimes\mathcal{I}')(\zeta\xi')
			&	:= \mathcal{I}(\zeta) \otimes \id_{\mathcal{I}'(\xi')}, \\
	(\mathcal{I}\otimes\mathcal{I}')(\xi\zeta')
			& := \id_{\mathcal{I}(\xi)} \otimes\,\mathcal{I}'(\zeta').
\end{align}
As usual, $\xi$ and $\xi'$ stand for sequences encoding vertices, whereas $\zeta$ and $\zeta'$ encode edges. Notice that $\mathcal I\otimes \mathcal I'$ may not commute even if both $\mathcal I$ and $\mathcal I'$ do.

Consider now link diagrams $D$, $D'$ and their cubes of resolutions $\KhCubeGradedSigned{D}{\epsilon}$, $\KhCubeGradedSigned{D'}{\epsilon'}$ corrected by certain sign assignments $\epsilon\in C^1(I^n;\invScalars*)$ and $\epsilon'\in C^1(I^{n'};\invScalars)$, where $n$ and $n'$ stand for the~number of crossings in $D$ and $D'$ respectively. Lemma~\ref{lem:degree-vs-tensor} suggests which sign assignment to choose for $\KhCubeGraded{D\sqcup D'}$, so that the~corrected cube is isomorphic to $\KhCubeGradedSigned{D}{\epsilon}\rdsum\KhCubeGradedSigned{D'}{\epsilon'}$. Namely, we define $\epsilon\star\epsilon'\in C^1(I^{n+n'};\invScalars)$ as follows:
\begin{align}
		(\epsilon\star\epsilon')(\zeta\xi') &:= \epsilon(\zeta), \\
		(\epsilon\star\epsilon')(\xi\zeta') &:= \begin{cases*}
				(-1)^{\|\xi\|}\permMM^a\permMS^b\epsilon'(\zeta'),& if $D'_{\xi'}$ is a~merge,\\
				(-1)^{\|\xi\|}\permSS^b\permMS^{-a}\epsilon'(\zeta'),& if $D'_{\xi'}$ is a~split.\\
			\end{cases*}
\end{align}
Notice that against the~notation $\epsilon\star\epsilon'$ depends on the~diagram $D'$.

\begin{lemma}\label{lem:KhCube-rdsum}
	The~cochain $\epsilon\star\epsilon'$ is a~sign assignment for $\KhCubeGraded{D\sqcup D'}$, for which the~cube is isomorphic to $\KhCubeGradedSigned{D}{\epsilon}\rdsum\KhCubeGradedSigned{D'}{\epsilon'}$.
\end{lemma}
\begin{proof}
	Consider the~following family of isomorphisms of vertices of the~cubes:
	\begin{equation}\label{eq:isom-for-disjoint-union}
		\begin{gathered}
		\mathllap{(\KhCubeGraded{D}\rdsum\KhCubeGraded{D'})(\xi\xi') =\;}
			\rnode{top}{D_\xi\{a,b\} \sqcup D_{\xi'}\{a',b'\}}
		\\[4ex]
			\rnode{bot}{(D_\xi\sqcup D'_{\xi'})\{a+a',b+b'\}}
			\mathrlap{\;= \KhCubeGraded{D\sqcup D'}(\xi\xi')}
		\end{gathered}
		\diagline{->}{top}{bot}
		\nbput[labelsep=2pt]{\scriptstyle\cong}
		\naput[labelsep=2pt]{\scriptstyle i_{a+a',b+b'}\circ(i_{a,b}^{-1}\sqcup i_{a',b'}^{-1})}
	\end{equation}

	\noindent
	\wrapfigure[r]<3>{\begin{pspicture}(-3.5em,-1ex)(12em,13ex) % (-3.5em,-1ex) for the full picture;
	                                                           % shrinked b/c of \qed
		% isomorphism
		\rput(8.5em,10.5ex){\rnode{top}{$\KhCubeGradedSigned{D}{\epsilon}\ldsum\KhCubeGradedSigned{D'}{\epsilon'}$}}
		\rput(8.5em, 1.5ex){\rnode{bot}{$\KhCubeGradedSigned{D\sqcup D'}{\epsilon\ldsum\epsilon'}$}}
		\diagline{->}{top}{bot}\naput[labelsep=2pt]{$\scriptstyle\cong$}
		% top face
		\dotnode(-3em,12ex){t00}\dotnode(1em,12ex){t01}
		\dotnode(-1em, 9ex){t10}\dotnode(3em, 9ex){t11}
		% bottom face
		\dotnode(-3em,3ex){b00}\dotnode(1em,3ex){b01}
		\dotnode(-1em,0ex){b10}\dotnode(3em,0ex){b11}
		% arrows - top
		\diagline{->}{t00}{t01}\diagline{->}{t01}{t11}
		\diagline{->}{t00}{t10}
		% arrows - bottom
		\diagline{->}{b00}{b01}\diagline{->}{b01}{b11}
		\diagline{->}{b00}{b10}\diagline{->}{b10}{b11}
		% arrows - vertical
		\diagline{->}{t00}{b00}\diagline{->}{t11}{b11}\diagline{->}{t01}{b01}
		\diagline[border=3\pslinewidth]{->}{t10}{b10}
		% top front arrow
		\diagline[border=3\pslinewidth]{->}{t10}{t11}
	\end{pspicture}}
	They form an~isomorphism of cubes due to Lemma~\ref{lem:degree-vs-tensor}, so it remains to check $\epsilon\star\epsilon'$ is indeed a~sign assignment. This follows easily if $S$ is a~face spanned by edges of $\KhCubeGraded{D}$ or $\KhCubeGraded{D'}$, as $d(\epsilon\star\epsilon')(S)$ is equal then to $d\epsilon(S)$ or $d\epsilon(S')$ respectively. In a~mixed case consider the~diagram to the~right, formed by the~face $S$, its analogue in $\KhCubeGradedSigned{D}{\epsilon}\rdsum\KhCubeGradedSigned{D'}{\epsilon'}$, and the~isomorphisms \eqref{eq:isom-for-disjoint-union}. The~four vertical squares commute due to Lemma \ref{lem:degree-vs-tensor}, and the~top horizontal square anticommutes. Hence, the~face $S$ anticommutes (the bottom square), and as it represents a~disjoint permutation relation \eqref{rel:disjoint-union} it must be $-d(\epsilon\star\epsilon')(S) = \psi(S)$.
\end{proof}

\begin{proposition}\label{prop:Kh-disjoint-union}
	Given link diagrams $D$ and $D'$ there is an~isomorphism of bigraded complexes $\KhCom(D\sqcup D')$ and $\KhCom(D)\rdsum\KhCom(D')$, which is natural with respect to graded morphisms.
\end{proposition}
\begin{proof}
	The~isomorphism of cubes from Lemma~\ref{lem:KhCube-rdsum} implies that $\KhBracket{D}_\epsilon\rdsum\KhBracket{D'}_{\epsilon'}\cong\KhBracket{D\sqcup D'}_{\epsilon\star\epsilon'}$. The~thesis follows, as a~global degree shift does not affect the~differential: $\lambda(a-b,x)=1$ in \eqref{diag:tensor-id-f} if $a=b$.
\end{proof}

There is a~similar formula for another operation on links. The~\emph{connected sum} $D\connsum D'$ of two oriented link diagrams with basepoints $D$ and $D'$ is given by cutting the~diagrams at the~basepoints and gluing them in such a~way that their orientations agree:
\begin{center}
	\psset{unit=1em}%
	\begin{pspicture}(-7,-1)(7,1.5)
		\rput(-5,0){%
			\psframe[linewidth=0.5pt](-3,-1)(-1,1)\rput(-2,0){$D$}%
			\psframe[linewidth=0.5pt]( 3,-1)( 1,1)\rput( 2,0){$D\mathrlap{\smash'}$}%
			\psarc[linewidth=1pt](-1,0){0.7}{-90}{90}
			\rput(-0.25pt,1.75pt){\psline[linestyle=none,arrowsize=5pt,arrowlength=0.8,arrowinset=0.5]{->}(-0.3,-0.5)(-0.3,0)}
			\psarc[linewidth=1pt](1,0){0.7}{90}{270}
			\rput(0.25pt,-1.75pt){\psline[linestyle=none,arrowsize=5pt,arrowlength=0.8,arrowinset=0.5]{->}( 0.3, 0.5)( 0.3,0)}
		}%
		\rput( 5,0){%
			\psframe[linewidth=0.5pt](-3,-1)(-1,1)\rput(-2,0){$D$}%
			\psframe[linewidth=0.5pt]( 3,-1)( 1,1)\rput( 2,0){$D\mathrlap{\smash'}$}%
			\pscustom[linewidth=1pt]{%
				\moveto(-1,-0.7)\curveto(-0.5,-0.7)(-0.7,-0.4)(0,-0.4)\curveto(0.7,-0.4)(0.5,-0.7)(1,-0.7)\stroke
				\moveto(-1, 0.7)\curveto(-0.5, 0.7)(-0.7, 0.4)(0, 0.4)\curveto(0.7, 0.4)(0.5, 0.7)(1, 0.7)\stroke
			}
			\psline[linestyle=none,arrowsize=5pt,arrowlength=0.8,arrowinset=0.5]{->}(-0.3,-0.4)( 2pt,-0.4)
			\psline[linestyle=none,arrowsize=5pt,arrowlength=0.8,arrowinset=0.5]{->}( 0.3, 0.4)(-2pt, 0.4)
		}%
		\pnode(-5,0){disjoint}
		\pnode( 5,0){connected}
		\diagline[nodesep=3.5]{->}{disjoint}{connected}\diagaput{$\scriptstyle\connsum$}
	\end{pspicture}%
\end{center}
This operation does not depends on the~exact placement of the~basepoints, but only on the~link components that carry them. Clearly, resolutions of $D\connsum D'$ are exactly the~connected sums of resolutions of $D$ and $D'$: $(D\connsum D')_{\xi\xi'} = D_\xi\connsum D_{\xi'}$. Since the~right connected sum $\rcsum$ behaves like the~right disjoint union $\rdsum$, in~particular the~Lemma~\ref{lem:degree-vs-tensor} holds, the~proofs of Lemma~\ref{lem:KhCube-rdsum} and Proposition~\ref{prop:Kh-disjoint-union} can be easily adapted to this operation. We leave it to the~reader to check the~details.

\begin{proposition}\label{prop:Kh-conn-sum}
	Given link diagram $D$ and $D'$ there is an~isomorphism of bigraded complexes $\KhCom(D\connsum D') \cong \KhCom(D)\rcsum\KhCom(D')$.
\end{proposition}

Consider now $(2m,2n)$-tangles, i.e.\ embeddings of $n+m$ intervals and circles in $\R^2\times I$ with $2n$ of the~endpoints on $\R^2\times\{0\}$ and $2m$ on $\R^2\times\{1\}$. Given a~$(2m,2n)$-tangle $T$ and a~$(2n,2k)$-tangle $T'$ we can compose them by gluing along the~$n$ edpoints to obtain a~$(2m,2k)$-tangle $T'T$. In particular, a~disjoint union of links correspond to composition of $(0,0)$ tangles, whereas a~connected sum can be interpreted as a~composition of a~$(0,2)$-tangle with a~$(2,0)$-tangle. From this viewpoint, Propositions \ref{prop:Kh-disjoint-union} and \ref{prop:Kh-conn-sum} provide formulas for Khovanov complexes for compositions of $(2m,2n)$-tangles with $m,n\leqslant 1$.

Unfortunately, the~ideas from this paper do not immediately generalize to $m,n>1$. The~main issue is the~lack of distinction by the~new degree of the~two types of faces in a~cube of resolutions labeled in Tab.~\ref{tab:psi-values} with $1$ and $\permMM\permSS$. Hence, for general tangles one needs more structure, which is a~subject of future research.

%% file: module.tex
The~even Khovanov homology is known to carry a~module structure over the~Frobenius algebra associated to a~circle \cites{KhDetectsUnlinks,KhPatterns}, and we want to obtain a~similar result for the~generalized complex. In order to have a~universal result, independent of a~chronological TQFT, we shall show $\KhCom(D)$ is a~\emph{module object} over a~certain \emph{algebra object} on a~circle.

Choose a~symmetric monoidal category $\cat{C}$ with a~unit $I$ and symmetry isomorphisms $c_{A,B}\colon A\otimes B\to B\otimes A$.

\begin{definition}\label{def:alg-obj}
	An~\emph{algebra object} in $\cat{C}$ consists of an~object $A$ and an~associative morphism $m\colon A\otimes A\to A$ called a~\emph{product} in $A$. We say that $(A,m)$ is \emph{symmetric} if $m\circ c_{A,A} = m$, and \emph{unital} if the~diagram below commutes
	\begin{equation}
		\begin{diagps}(-6em,-0.5ex)(6em,10ex)
			\hsquares(-5em,0ex)<5em,8ex>|a{npos=0.6}`a{npos=0.45}`b`b`a`a{npos=0.6}`a{npos=0.45}|{<-`->`->`=`->`->`<-}[%
					I\otimes A`A`A\otimes I`%
					A\otimes A`A`A\otimes A;%
					\cong`\cong`%
					\iota\otimes\id``\id\otimes\iota`%
					m`m]
		\end{diagps}
	\end{equation}
	for some morphism $\iota\colon I\to A$.\footnote{
		It can be shown that such a~morphism is unique if it exists.
	}
\end{definition}

One example of an~algebra object is provided by a~circle $\fntCircle$ in the~category $\Z\ChCob$ with the~merge cobordism as a~product. It does not generalize immediately to $\kChCob$, as the~merge is not associative here. However, it is enough to shift the~degree of $\fntCircle$ by $\{-1,0\}$ to obtain an~associative product. It makes the~merge a~graded map.

\begin{lemma}\label{lem:circle-is-alg-obj}
	The~merge cobordism induces on $\fntCircle\{-1,0\}$ a~structure of a~commutative unital algebra object.
\end{lemma}

Before we prove the~lemma we shall extend our pictorial calculus of cobordisms by two degree shift isomorphisms: $\textcobordism[1]({y=0.25}Sh-)(sI)\colon A\to A\{-1,0\}$ and $\textcobordism[1]({y=0.25}Sh+)(sI)\colon A\{-1,0\}\to A$. Both maps have nontrivial degrees, so that they commute with each other and with other generators only up to scaling by certain scalars controlled by the~function $\lambda$. In particular, each time we switch heights of two bars we have to scale the~cobordism by $\permMM$.

\begin{proof}
	Associativity follows from the~chronological relations, and the~fact that the~two isomorphisms are inverse to each other:
	\begin{equation}
		\textcobordism*[3]%
				({y=0.35}Sh+)({y=0.15}Sh+,2)(sI)(M-L)({size=0.65}Sh-)%
				({y=0.35}Sh+)({y=0.15}Sh+,2)(sI)(M-L)({size=0.65}Sh-)
		=
		\textcobordism*[3]({y=0.75}Sh+)({y=0.5}Sh+,2)(Sh+,3)(I)(M-L)(M-L)({size=0.65}Sh-)
		= \permMM
		\textcobordism*[3]({y=0.75}Sh+)({y=0.5}Sh+,2)(Sh+,3)(I)(M-L,2)(M-L)({size=0.65}Sh-)
		=
		\textcobordism*[3]%
			({y=0.35}Sh+,2)({y=0.15}Sh+,3)(sI)(M-L,2)({size=0.65}Sh-,2)%
			({y=0.35}Sh+)({y=0.15}Sh+,2)(sI)(M-L)({size=0.65}Sh-).
	\end{equation}
	We left as an~exercise to check the~other axioms of an~algebra object.
\end{proof}

\begin{definition}\label{def:mod-obj}
	A~\emph{left module object} over a~unital algebra object $A$ in $\cat{C}$ consists of an~object $M$ and a~morphism $a\colon A\otimes M\to M$ such that the~following diagram commutes:
	\begin{equation}
		\begin{diagps}(-8em,-0.5ex)(6em,12ex)
			\node AAM(-8em,10ex)[A\otimes A\otimes M]	\node AM1(0em,10ex)[A\otimes M]	\node IM(6em,10ex)[I\otimes M]
			\node	AM2(-8em,0ex)[A\otimes M]						\node M(0em,0ex)[M]
			\arrow{->}[AAM`AM1;m\otimes\id]				\arrow{<-}[AM1`IM;\iota\otimes\id]
			\arrow|b|{->}[AAM`AM2;\id\otimes a]		\arrow|b|{->}[AM1`M;m]
			\arrow{->}[AM2`M;m]										\arrow{->}[IM`M;\cong]
		\end{diagps}
	\end{equation}
	The~morphism $a\colon A\otimes M\to M$ is called an~\emph{action} of $A$ on $M$. Given two module objects $M$ and $N$ over $A$ a~morphism $f\colon M\to N$ in $\cat{C}$ is a~\emph{homomorphism of module objects} if it commutes with the~actions, i.e.\ $m_N\circ(\id\otimes f) = f\circ m_M\colon A\otimes M\to N$.
\end{definition}

We define right module objects and homomorphisms between them likewise. Given two algebra objects $A$ and $B$ we say that $M$ is an~$(A,B)$-\emph{bimodule object} if it is a~left $A$-module object, a~right $B$-module object, and the~two actions commute. An~$(A,A)$-bimodule $M$ is \emph{symmetric} if the~diagram below commutes as well:
\begin{equation}
	\begin{diagps}(0em,-0.5ex)(6em,10ex)
		\Dtriangle<6em,8ex>[A\otimes M`M`M\otimes A;`c_{A,M}`]
	\end{diagps}
\end{equation}

Any algebra object is at the~same time a~bimodule object over itself, and so is $\fntCircle\{-1,0\}$. We can generalize this to any family of circles, perhaps with a~shifted degree, as long as we distinguish one of them. To show this we shall again extend our pictorial calculus by adding the~pictures
\begin{equation}
	\psset{unit=1.3}
	\textcobordism[5]({y=0.5,deg=x,size=5}Sh)(I)
\end{equation}
for the~degree shift isomorphisms $(k\fntCircle)\{a\}\to(k\fntCircle)\{a+x\}$, where $k\fntCircle$ stands for a~disjoint union of $k$ circles and $a,x\in\Z\times\Z$. Clearly, we have the~equalities
\begin{gather}
	\psset{xunit=1.3}
	\textcobordism[5]({y=0.4,deg=x,size=5}Sh)(I)(sI)({y=-0.4,deg=y,size=5}Sh)
	=
	\textcobordism[5](sI)(I)({y=-0.75,deg=x+y,size=5}Sh)\mathrlap{,\text{ and}}\\
	\psset{xunit=1.3}
	\textcobordism[5]({y=0.4,deg=x,size=5}Sh)(I)(sI)({y=-0.4,deg=-x,size=5}Sh)
	=
	\textcobordism[5](sI)(I)\mathrlap{.}
\end{gather}

We distinguish one of the~circles in $(k\fntCircle)\{a\}$ by placing a~basepoint $p$ on it, and we define the~left and right actions of $\fntCircle\{-1,0\}$ by merging the~circle with $k\fntCircle$ at $p$ from the~left and right hand side respectively, everything conjugated by the~degree shift isomorphism accordingly:
\begin{equation}
	\psset{xunit=4mm,yunit=6mm}\def\COBxsize{1}\def\COBysize{2}\COBsetstyle
	\begin{centerpict}(0,-0.15)(11,4.15)
		\rput( 0,0)\COBembcircle
		\rput( 6,0)\COBembcircle
		\rput( 6,4)\COBcircle
		\pscustom{%
				\moveto(0,0)\lineto(0,1)
				\curveto(0,2.5)(6,2)(6,3)
				\lineto(6,4)\stroke}%
		\pscustom{%
				\moveto(1,0)\lineto(1,1)
				\curveto(1,1.7)(5,1.8)(5.5,2.0)
				\curveto(6.1,2.1)(6,2.4)(6,1)
				\lineto(6,0)\stroke}%
		\psdot(6.5,-0.1)\psdot(6.5,3.9)\psline(7,0)(7,4)
		\psarrow(5.9,1.8)(5.2,1.85)
		\rput{-20}(0,1){\COBposShift{x=0.45,y=0.15}(0,0)}
		\COBcylinder( 2,0)( 2,4)
		\COBcylinder( 4,0)( 4,4)
		\COBcylinder( 8,0)( 8,4)
		\COBcylinder(10,0)(10,4)
		\COBshift{deg=-a,size=5}(2,0.7)
		\COBshift{deg=a,size=5}(2,3.3)
	\end{centerpict}
	\qquad\text{and}\qquad
	\begin{centerpict}(1,-0.15)(-10,4.15)
		\rput( 0,0)\COBembcircle
		\rput(-6,0)\COBembcircle
		\rput(-6,4)\COBcircle
		\pscustom{%
				\moveto(1,0)\lineto(1,1)
				\curveto(1,2.5)(-5,2)(-5,3)
				\lineto(-5,4)\stroke}%
		\pscustom{%
				\moveto(0,0)\lineto(0,1)
				\curveto(0,1.7)(-4,1.8)(-4.5,2.0)
				\curveto(-5.1,2.1)(-5,2.4)(-5,1)
				\lineto(-5,0)\stroke}%
		\psdot(-5.5,-0.1)\psdot(-5.5,3.9)\psline(-6,0)(-6,4)
		\psarrow(-4.9,1.8)(-4.2,1.85)
		\COBcylinder( -2,0)( -2,4)
		\COBcylinder( -4,0)( -4,4)
		\COBcylinder( -8,0)( -8,4)
		\COBcylinder(-10,0)(-10,4)
		\COBposShift(0,0)
		\COBshift{deg=-a,size=5}(-10,1.0)
		\COBshift{deg=a,size=5}(-10,3.3)
	\end{centerpict}
\end{equation}
The~pictures are not totally symmetric---compare the~orders of degree shift isomorphisms.

\begin{lemma}\label{lem:cube-of-bimod}
	The~above actions equip $k\fntCircle$ with a~structure of a~symmetric bimodule object over $\fntCircle\{-1,0\}$. In particular, given a~based link diagram $L$ the~cube $\KhCubeGraded{D}$ is a~cube of symmetric bimodule objects and module homomorphisms.
\end{lemma}
\begin{proof}
	All axioms of a~symmetric bimodule for $k\fntCircle$ can be checked easily using the~pictorial calculus of cobordism as in the~proof of Lemma~\ref{lem:circle-is-alg-obj}. Hence, vertices of $\KhCubeGraded{D}$ are decorated with symmetric bimodules objects. To show that an~edge morphism preserves the~left action of $\fntCircle\{-1,0\}$ we represent it in the~following computation by a~saddle conjugated with degree shift isomorphisms:
	\begin{equation}
		\psset{xunit=4mm,yunit=7mm}\def\COBxsize{1}\def\COBysize{1.7}\COBsetstyle
		\begin{split}
		\pictActionFirst
		 &= \pictActionSecond \\
		 &= \pictActionThird
		  = \pictActionForth.
		\end{split}
	\end{equation}
	The~case of the~right action is proved likewise.
\end{proof}

\begin{corollary}\label{cor:KhCom-module}
	The~generalized Khovanov complex $\KhCom(D)$ of a~based link diagram $D$ is a~complex of symmetric bimodule objects over $\fntCircle\{-1,0\}$.
\end{corollary}

The~bimodule structure on $\KhCom(D)$ is a~link invariant. Indeed, one can always perform Reidemeister moves beyond a~small neighborhood of the~basepoint, perhaps using the~isotopy through infinity, and the~chain homotopy equivalences associated to these Reidemeister moves commute with the~action of $\fntCircle\{-1,0\}$. In particular, we can move the~basepoint freely along a~component of a~link.

On the~other hand, moving a~basepoint to a~different component of a~link may change the~module structure. Following the~idea from \cite{KhDetectsUnlinks} we can choose a~basepoint on every component of a~link, which results in a~bimodule structure on $\KhCom(D)$ over the~algebra object $\bigsqcup_c(\fntCircle\{-1,0\})$ consisting of as many copies of the~shifted circle, as there are components in $L$. Again, this structure descends to homology, but we have to work harder to proof this structure is invariant under Reidemeister move: with more that one basepoint we cannot avoid passing them through crossings.

\begin{theorem}\label{thm:inv-module-structure}
	Given a~based link diagram $D$, the~bimodule structure on $\KhCom(D)$ is preserved up to isomorphism when the~basepoint is moved through a~crossing:
	\begin{equation}
		\begin{centerpict}(-0.5,-0.5)(0.5,0.5)
			\pscurve(-0.5,-0.5)(-0.2,0.2)(0.5,0.5)
			\psline[border=3\pslinewidth](-0.5,0.5)(0.5,-0.5)\psdot(0.1,-0.1)
		\end{centerpict}
		\to/<->/
		\begin{centerpict}(-0.5,-0.5)(0.5,0.5)
			\pscurve(-0.5,-0.5)(0.2,-0.2)(0.5,0.5)
			\psline[border=3\pslinewidth](-0.5,0.5)(0.5,-0.5)\psdot(-0.1,0.1)
		\end{centerpict}
		\qquad\text{and}\qquad
		\begin{centerpict}(-0.5,-0.5)(0.5,0.5)
			\psline(-0.5,0.5)(0.5,-0.5)\psdot(0.1,-0.1)
			\pscurve[border=3\pslinewidth](-0.5,-0.5)(-0.2,0.2)(0.5,0.5)
		\end{centerpict}
		\to/<->/
		\begin{centerpict}(-0.5,-0.5)(0.5,0.5)
			\psline(-0.5,0.5)(0.5,-0.5)\psdot(-0.1,0.1)
			\pscurve[border=3\pslinewidth](-0.5,-0.5)(0.2,-0.2)(0.5,0.5)
		\end{centerpict}
	\end{equation}
	In particular, given a~link $L$ with $c$ components, the~bimodule structure on $\KhCom(L)$ over $\bigsqcup_c(\fntCircle\{-1,0\})$ is an~invariant of $L$.
\end{theorem}
\begin{proof}
	Invariance of the~bimodule structure under passing a~dot through a~crossing is equivalent to saying that the~following pairs of cobordisms
	\begin{gather}\psset{unit=7.5mm}
	\begin{movie}[h](2,2){4}
		\movieclip{
			\pscurve(-0.8,-0.8)(-0.5,0.2)(0.3, 0.8)
			\psbezier[border=3\pslinewidth](-0.8, 0.8)(-0.2, 0)(0,-0.2)(0.8,-0.8)
			\pscircle[border=3\pslinewidth](0.3,0.3){0.3}
		}
		\movieclip{
			\psline(-0.8,-0.8)(0.8, 0.8)
			\psbezier[border=3\pslinewidth](-0.8, 0.8)(-0.2, 0)(0,-0.2)(0.8,-0.8)
			\pscircle[border=3\pslinewidth](0.3,0.3){0.3}
		}
		\movieclip{
			\pscurve(-0.8,-0.8)(0.2,-0.5)(0.8,0.3)
			\psbezier[border=3\pslinewidth](-0.8, 0.8)(-0.2, 0)(0,-0.2)(0.8,-0.8)
			\pscircle[border=3\pslinewidth](0.3,0.3){0.3}
		}
		\movieclip{
			\pscurve(-0.8,-0.8)(0.2,-0.5)(0.8,0.3)
			\pscustom{\moveto(-0.8,0.8)
				\curveto(-0.2,0)(0,0)(0,0.3)
				\psarcn(0.3,0.3){0.3}{180}{-90}
				\curveto(0,0)(0,-0.2)(0.8,-0.8)
				\stroke
				\stroke[linecolor=white,linewidth=7\pslinewidth]
		}}
	\end{movie}
	\quad\to/<->/\quad
	\begin{movie}[h](2,2){2}
		\movieclip{
			\pscurve(-0.8,-0.8)(0.2,-0.5)(0.8,0.3)
			\psbezier[border=3\pslinewidth](-0.8, 0.8)(-0.2, 0)(0,-0.2)(0.8,-0.8)
			\pscircle[border=3\pslinewidth](0.3,0.3){0.3}
		}
		\movieclip{
			\pscurve(-0.8,-0.8)(0.2,-0.5)(0.8,0.3)
			\pscustom{\moveto(-0.8,0.8)
				\curveto(-0.2,0)(0,0)(0,0.3)
				\psarcn(0.3,0.3){0.3}{180}{-90}
				\curveto(0,0)(0,-0.2)(0.8,-0.8)
				\stroke
				\stroke[linecolor=white,linewidth=7\pslinewidth]
		}}
	\end{movie}
	\\[1ex]
	\psset{unit=7.5mm}
	\begin{movie}[h](2,2){4}
		\movieclip{
			\psbezier(-0.8, 0.8)(-0.2, 0)(0,-0.2)(0.8,-0.8)
			\pscircle(0.3,0.3){0.3}
			\pscurve[border=3\pslinewidth](-0.8,-0.8)(-0.5,0.2)(0.3, 0.8)
		}
		\movieclip{
			\psbezier(-0.8, 0.8)(-0.2, 0)(0,-0.2)(0.8,-0.8)
			\pscircle(0.3,0.3){0.3}
			\psline[border=3\pslinewidth](-0.8,-0.8)(0.8, 0.8)
		}
		\movieclip{
			\psbezier(-0.8, 0.8)(-0.2, 0)(0,-0.2)(0.8,-0.8)
			\pscircle(0.3,0.3){0.3}
			\pscurve[border=3\pslinewidth](-0.8,-0.8)(0.2,-0.5)(0.8,0.3)
		}
		\movieclip{
			\pscustom{\moveto(-0.8,0.8)
				\curveto(-0.2,0)(0,0)(0,0.3)
				\psarcn(0.3,0.3){0.3}{180}{-90}
				\curveto(0,0)(0,-0.2)(0.8,-0.8)
				\stroke
			}
			\pscurve[border=3\pslinewidth](-0.8,-0.8)(0.2,-0.5)(0.8,0.3)
		}
	\end{movie}
	\quad\to/<->/\quad
	\begin{movie}[h](2,2){2}
		\movieclip{
			\psbezier(-0.8, 0.8)(-0.2, 0)(0,-0.2)(0.8,-0.8)
			\pscircle(0.3,0.3){0.3}
			\pscurve[border=3\pslinewidth](-0.8,-0.8)(0.2,-0.5)(0.8,0.3)
		}
		\movieclip{
			\pscustom{\moveto(-0.8,0.8)
				\curveto(-0.2,0)(0,0)(0,0.3)
				\psarcn(0.3,0.3){0.3}{180}{-90}
				\curveto(0,0)(0,-0.2)(0.8,-0.8)
				\stroke
			}
			\pscurve[border=3\pslinewidth](-0.8,-0.8)(0.2,-0.5)(0.8,0.3)
		}
	\end{movie}\end{gather}
	induce isomorphic operations on complexes. For that it is enough to show that placing a~circle on one side of a~strand is the~same as placing it on the~other side and moving over or under the~piece of the~link. We shall prove this only when the~link diagram is the~unknot---the~general case then follows from Proposition~\ref{prop:Kh-conn-sum}, as every link $L$ with a~basepoint is a~connected sum of a~based unknot with $L$.
	
	We begin with computing the~chain map induced by passing a~circle over the~unknot---it is given by a~composition of the~chain homotopy equivalences used in the~proof of invariance of the~chain complex under the~second Reidemeister move---see the~diagram in Fig.~\ref{fig:passing-circle}, in which we show only the~essential fragment of the~unknot. Instead of drawing the~cobordisms, which would make the~diagram illegible, we used the~symbols $\mu, \Delta, \eta, \epsilon$ to denote respectively a~merge, a~split, a~birth, and a~death, with the~exception of two morphisms that are of a~particular interest to us. In addition we use the~following conventions:
	\begin{itemize}
		\item the~two crossings in \fntCircOver\ are decorated with arrows pointing inwards,
		\item when enumerating circles, the~one being moved is always put first.
	\end{itemize}
	With the~chosen decoration by arrows, the~commutativity cocycle $\psi$ takes the~value $\permMM\permSS$ on the~middle square, see Tab.~\ref{tab:psi-values} on page~\pageref{tab:psi-values}.
	
	\begin{figure}[t]
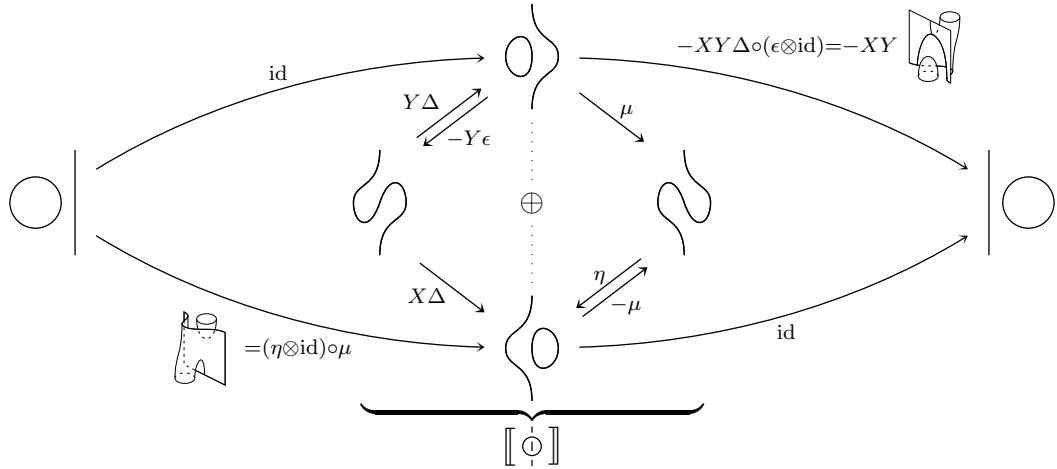

	\begin{displaymath}
		\psset{linewidth=0.5pt}\begin{array}{lcr}
		&\rnode{01}{\tangleCircOverhv}&	\\[2ex]
		\rnode{left}{\tangleCircLeft}			\hskip 3cm&
		\rnode[l]{00}{\tangleCircOverhh}\hskip 3cm
		\rnode[r]{11}{\tangleCircOvervv}	&\hskip 3cm
		\rnode{right}{\tangleCircRight}	\\[2ex]
		&\underbrace{\hskip 17.5mm\rnode{10}{\tangleCircOvervh}\hskip 17.5mm}_{\displaystyle\KhBracket*{\fntCircOver}}&
		\end{array}
		\diagline[linestyle=dotted,nodesep=5pt]{-}{01}{10}\ncput*{\oplus}
		\everypsbox{\scriptstyle}
		\diagline[offset=2pt]{->}{00}{01}\naput[npos=0.55]{\permSS\Delta}
		\diagline[offset=2pt]{->}{01}{00}\naput[labelsep=2pt,npos=0.45]{-\permSS\epsilon}
		\diagline{->}{00}{10}\nbput[labelsep=1pt,npos=0.55]{\permMM\Delta}
		%\diagline[offset=2pt]{->}{10}{00}\naput{-\permSS\epsilon}
		\diagline{->}{01}{11}\naput[labelsep=1pt,npos=0.45]{\mu}
		%\diagline[offset=2pt]{->}{11}{01}\naput{-\eta}
		\diagline[offset=-2pt]{->}{10}{11}\nbput[labelsep=2pt,npos=0.4]{-\mu}
		\diagline[offset=-2pt]{->}{11}{10}\nbput[labelsep=2pt,npos=0.6]{\eta}
		\diagarc[arcangle= 15,angleA= 45,angleB=180]{->}{left}{01}\naput[labelsep=2pt]{\id}
		\diagarc[arcangle=-15,angleA=-45,angleB=180]{->}{left}{10}\nbput[labelsep=-2ex]{\psset{unit=1.25}\fntCobPassLeftNonID = (\eta\otimes\id)\circ\mu}
		\diagarc[arcangle= 15,angleB= 135,angleA=0]{->}{01}{right}\naput[labelsep=-2ex]{-\permMM\permSS\Delta\circ(\epsilon\otimes\id) = -\permMM\permSS\!\!\psset{unit=1.25}\fntCobPassRightNonID}
		\diagarc[arcangle=-15,angleB=-135,angleA=0]{->}{10}{right}\nbput[labelsep=2pt]{\id}
	\end{displaymath}
	\caption{The~chain map associated with a~circle moving over a~link diagram.}%
	\label{fig:passing-circle}%
	\end{figure}

	The~left homomorphism is the~inclusion $f\colon \KhBracket*{\fntCircLeft} \to \KhBracket*{\fntCircOver}$, and the~right one is the~retraction $g\colon \KhBracket*{\fntCircOver} \to \KhBracket*{\fntCircRight}$ from the~proof of invariance of the~complex under the~second Reidemeister move; the~two backward maps in the~middle complex are pieces of chain homotopies (compare with Fig.~\ref{fig:R-II}). Regarding this diagram as in $\kChCob$ and placing the~circle acting on the~knot, i.e.\ the~one visible in full, as the~left most one we compute the~total map
	\begingroup\psset{unit=1cm}%
	\begin{equation}\label{eq:total-passing-map}
		\omega = \pictPassLower - \permMM\permSS\pictPassUpper,
	\end{equation}
	where again the~convention is that the~orienting arrows points towards left or back, and deaths are oriented clockwise. 	The~expression \eqref{eq:total-passing-map} simplifies under the~relation \emph{4Tu} into a~disjoint union of a~cobordism acting on a~circle with a~vertical wall:
	\begin{equation}\psset{unit=0.9}\begin{split}
		\pictPassLower - \permMM\permSS\pictPassUpper
			&= \permMM\permMS^{-1}\left(\permMM\pictPassLowerTu - \permMS\pictPassUpperTu\right) \\
			&= \permMM\permMS^{-1}\left(\permMS\pictPassIdTu - \permSS\pictPassTorusTu\right) 
			= \pictPassId - \permMM\permMS^{-1}\pictPassTorus.
	\end{split}\end{equation}
	\endgroup
	Denote by $\varphi\colon\fntCircle\{-1,0\}\to\fntCircle\{-1,0\}$ the~cobordism acting on the~circle. A~direct computation shows $\varphi^2=\id$
	\begin{equation}
		\varphi^2
		  = \textcobordism[1]({y=0.2}Sh+)(sI)(I)(Sh-)
		  - 2\permMM\permMS^{-1}
		  		\textcobordism[1]({size=0.45,y=0.15}Sh+)(sD)(Tu)({size=0.65,y=-0.15}Sh-)
		  + \permMS^{-1}(\permMM+\permSS)
	  			\textcobordism[1]({size=0.45,y=0.15}Sh+)(sD)(Tu)({size=0.65,y=-0.15}Sh-)
		  = \textcobordism[1]({y=0.2}Sh+)(sI)(I)(Sh-).
	\end{equation}

	\noindent\wrapfigure[r]{%
		\textcobordism[2]({y=0.3}Sh+)({y=0.15}Sh+,2)(sI)(B2)(M-L)(slTu)(M-L)({size=0.65}Sh-)%
	}
	so that $\varphi$ is an~isomorphism. In the~above computation we used the~fact that the~polynomial $\permMM-\permSS$ annihilates a~punctured torus, see Theorem~\ref{thm:kChCob-nondeg}. It remains to check that it is a~morphism of algebra objects, i.e.\ $\varphi\circ m = m\circ(\varphi\rdsum\varphi)$ where $m$ is the~merge cobordism
	$m:=\textcobordism[2]({y=0.35}Sh+)({y=0.15}Sh+,2)(sI)(M-L)({size=0.65}Sh-)$.
	This equality follows from the~\emph{4Tu} relation applied to the~cobordism shown to the~right. We left it to the~reader to check the~details.
	
	\parshape=0	% For some reason this paragraph wants to be indented as the previous one.
				% It is not as it should be, thence this explicit declaration.
	The~other case, when a~circle is moved below the~unknot, results exactly in the~same map.
\end{proof}

Now choose your favorite chronological TQFT $\F$ that satisfies the~relations \emph{S}, \emph{T}, and \emph{4Tu}. It intertwines the~disjoint union with the~tensor product over $\scalars$, and the~connected sum with a~tensor product over the~algebra $A':=\F(\fntCircle)\{1,0\}$; it follows directly from Lemma~\ref{lem:circle-is-alg-obj} that $A'$ is an~associative algebra.

A~punctured torus vanishes in the~odd Khovanov homology, making $\varphi=\id$. Hence, the~action does not depend on the~placement of the~circle.
	
In the~even case, one computes using the~Khovanov's TQFT that $\F\varphi$ preserves $v_+$, but it takes $v_-$ to $-v_-$. In other words, $\F\varphi$ is the~conjugation on $A$ defined as $\bar v_{\pm} := \pm v_{\pm}$. Hence, to define the~action of the~algebra $A:=\F(\fntCircle)$ properly, we can color the~regions of the~knot diagram $D$ black and white in a~checkerboard manner; if the~circle is merged to $D$ from a~white region use the~usual multiplication, but conjugate $A$ first if the~circle is merged from a~black region.

Lemma~\ref{lem:cube-of-bimod} together with Theorem~\ref{thm:inv-module-structure} implies that given an~$c$-component link $L$ with a~diagram $D$ the~homology $\Kh(L) := H(\F\KhCom(D))$ admits a~module structure over the~algebra $(A')^{\otimes c}$ that is independent on the~diagram chosen. We are now ready to state the~formulas for generalized Khovanov homology of composite links.

\begin{theorem}\label{prop:H-union-sum}
	Given two link diagrams $D$, $D'$, and a~chronological TQFT $\F$ there is an~isomorphism of complexes $\F\KhCom(D\sqcup D')\cong\F\KhCom(D)\underset{\scalars}\otimes\F\KhCom(D')$. Moreover, if the~diagrams are based, then $\F\KhCom(D\connsum D')\cong\F\KhCom(D)\underset{A'}\otimes\F\KhCom(D')$.
\end{theorem}
\begin{proof}
	The~first isomorphism follows from Proposition~\ref{prop:Kh-disjoint-union} and monoidality of $\F$. The~second is a~consequence of Proposition~\ref{prop:Kh-conn-sum} and Corollary~\ref{cor:KhCom-module}.
\end{proof}

The~sign assignment $\epsilon\star\epsilon'$ we chose for $\KhCubeGraded{D\sqcup D'}$ implies that
\begin{equation}\label{eq:diff-for-tensor}
	d(x\otimes y) = \begin{dcases*}
		dx\otimes y + (-1)^i \permMM^a\permMS^b\, x\otimes dy, & if $d\otimes\id$ is a~merge,\\
		dx\otimes y + (-1)^i \permSS^a\permMS^b\, x\otimes dy, & if\/ $\id\otimes d$ is a~split,
\end{dcases*}
\end{equation}
for $x\otimes y\in\F\KhBracket{D\sqcup D'}\cong\F\KhBracket{D}\otimes\F\KhBracket{D'}$. Since generalized Khovanov complexes are obtained from the~brackets by not necessarily integral shifts, one must be careful with a~choice of an~isomorphism $\F\KhBracket{D}\otimes\F\KhBracket{D'}\cong\F\KhCom(D)\otimes\F\KhCom(D')$. Indeed, taking a~tensor product of the~degree shift isomorphisms may require square roots of the~generators $\permMM$, $\permSS$, and $\permMS$. For example, in the~following diagram
\begin{equation}
	\begin{diagps}(0,-0.5ex)(15em,12ex)
		\square<15em,10ex>[%
			\F\KhBracket\fntCircle \otimes \F\KhBracket\fntCircle`
			\F\KhBracket{\fntCircle \fntCircle}`
			\F\KhCom(\fntCircle) \otimes \F\KhCom(\fntCircle)`
			\F\KhCom(\fntCircle \fntCircle);
			\cong`i_{(-1/2,1/2)}\otimes i_{(-1/2,1/2)}`i_{(-1,1)}`\cong]
	\end{diagps}
\end{equation}
the~bottom map takes $v_+\otimes v_+$ into $\permMM^{1/2}\permMS^{1/2} v_+\otimes v_+$. To overcome this problem, we replace the~left isomorphism with the~naive map $x\otimes y \mapsto i(x)\otimes i'(y)$ where $i\colon \KhBracket{D} \to \KhCom(D)$ and $i \colon \KhBracket{D'} \to \KhCom(D')$ are the~appropriate degree shift isomorphisms. Although not a~canonical one, it is still a~chain map.

Finally, the~above discussion apply also to the~connected sum of links. In~particular, a~similar formula to \eqref{eq:diff-for-tensor} is true for the~differential in $\F\KhCom(D\connsum D')$.